\documentclass{article}
\usepackage[utf8]{inputenc}
 \usepackage[T1]{fontenc} 
\usepackage{amsmath, amssymb, amscd, amsthm, comment, amsxtra,mathrsfs, enumerate}
\usepackage{graphicx}
\usepackage{subfigure}
\usepackage[pdftex, linktocpage=true]{hyperref}
\usepackage{enumitem}
\usepackage[small]{titlesec}

\usepackage{euscript}
\usepackage[format=plain,labelfont=bf,up,width=.99\textwidth]{caption}
\usepackage[toc,page]{appendix}
\usepackage[all]{xy}%
\usepackage[textwidth=16.5cm,textheight=21cm]{geometry}
\usepackage{tikz-cd}
\usepackage{xfrac}

\usepackage{arcs}
\newtheorem{thm}{Theorem}[section]
\newtheorem{lem}[thm]{Lemma}
\newtheorem{prop}[thm]{Proposition}
\newtheorem{propdef}[thm]{Proposition-Definition}
\newtheorem{cor}[thm]{Corollary}

\newtheorem*{claim*}{Claim}
\newtheorem*{thm*}{Theorem}
\newtheorem{defn}[thm]{Definition}

\newtheorem{rem}[thm]{Remark}

\newtheorem{assum}[thm]{Assumption}

\newtheorem*{conj*}{Conjecture}
\newtheorem*{rem*}{Remark}
\newtheorem*{lem*}{Lemma}

\newtheorem*{def*}{Definition}
\newtheorem*{fact*}{Fact}
\newtheorem*{assumption*}{Assumption on $s$}
\newtheorem*{assum*}{Assumption on $\tau$}
\newtheorem*{remnota*}{Remark on the notation}

\newtheorem{thmx}{Theorem}

\setenumerate{fullwidth,itemindent=\parindent,listparindent=\parindent,itemsep=0ex,partopsep=0pt,parsep=0ex}
\newlist{steps}{enumerate}{1}
\setlist[steps, 1]{label = \textbf{Step \arabic*}:}

\def\C{{\mathbb C}}
\def\D{{\mathbb D}}

 \def\epsilon{{\varepsilon}}

\renewcommand{\arg}{\operatorname{arg}}

\renewcommand{\Im}{\operatorname{Im}}

\newcommand{{\q}}{q_{\theta}}

\newcommand{{\Zak}}{\mathcal{Z}_{\theta}}

\title{Parabolic Implosion in the Parameter Space of Cubic Polynomials}
\author{Runze Zhang }
\date{\small}

\begin{document}
\maketitle
\begin{abstract}
    Parabolic implosion describes the enrichment of Julia sets when a parabolic fixed point is perturbed. It is also natural to study parabolic implosion in parameter spaces. In particular, when one perturbs properly the family of cubic polynomials having a stable parabolic fixed point into nearby families, the enrichment of the bifurcation loci occurs. We investigate the topology of such enrichment in the parameter space of cubic polynomials and relate it to the corresponding enrichment of Julia sets of quadratic polynomials, the latter of which has been studied systematically by P. Lavaurs in the 80s.
\end{abstract}
\tableofcontents

\section{Introduction}
Parabolic implosion is a remarkable phenomenon in complex dynamics. It describes the enrichment of the Julia set when a parabolic periodic point is perturbed. The study of parabolic implosion was begun by Douady and Hubbard \cite{DoHu1} who introduced Ecalle cylinder into the study of the bifurcation for parabolic points. Later, P. Lavaurs studied systematically this topic in his Phd thesis \cite{La,Do}. He studied the perturbation of the quadratic polynomial $\mathfrak{p}_\lambda(z) = z^2+\lambda z$ at $\lambda =1$ and showed that as $\lambda_n$ approaches $\lambda=1$ from certain directions, the limiting filled-in Julia set $\limsup_n K_{\lambda_n}$ is strictly smaller than the filled-in Julia set $K_1$ of the limiting map $\mathfrak{p}_1$ (equivalently, the limiting Julia set presents richer structures than the original one). A natural question is to ask what the enrichment consists of:  
$$\text{\textbf{Question 1.} Describe the difference } {K}_1\setminus\limsup_n {K}_{\lambda_n}.$$ 

To be more precise, Lavaurs showed the following:

\begin{thm}[Lavaurs]\label{thm.Lavaurs1}
    Let $\lambda_n = e^{2\pi i\alpha_n}$ with $\mathfrak{Re}\,\alpha_n>0$, $N_n\to+\infty$ a sequence of positive integers and $\tau\in\mathbb{C}$ such that
  $$1/\alpha_n = N_n -\tau +o(1),\quad n\to+\infty.$$
  Then $\mathfrak{p}^{N_n}_{\lambda_n}$ converges uniformly on compact subsets of the parabolic basin of $\mathfrak{p}_1$ to a transcendental map $L_{\mathfrak{p}_1,\tau}$\footnote{Nowadays known as the {\it Lavaurs map}.}. 
\end{thm}
 The filled-in Julia set $K(L_{\mathfrak{p}_1,\tau})$ for $L_{\mathfrak{p}_1,\tau}$ is defined to be the complement in $K_1$ of the following escaping region
 \begin{equation}\label{eq.escaping-quadratic}
     E := \{z\in\mathrm{int}\, {K_1};\,\exists N\geq 1 \text{ s.t. }L_{\mathfrak{p}_1,\tau}^N(z)\in\mathbb{C}\setminus{K_1}\}.
 \end{equation}
  
  Lavaurs was able to partially solve Question 1:
\begin{thm}[Lavaurs]\label{thm.Lavaurs2}
Under the same hypothesis as Theorem \ref{thm.Lavaurs1}, we have
$$\limsup_n K_{\lambda_n}\subset K(L_{\mathfrak{p}_1,\tau})\subsetneq K_1.$$ 
Moreover, if we set $E^1 := \{z\in\mathrm{int}\, {K_1};\,L_{\mathfrak{p}_1,\tau}(z)\in\mathbb{C}\setminus{K_1}\}$ to be the escaping set of the first level and suppose that the critical point $-\frac{1}{2}$ of $\frak{p}_1$ is not in $E^1$, then 
\begin{itemize}
    \item each connected component $C\subset E^1$ is a croissant (see Definition \ref{defn.croissant});
    \item each connected component $C\subset E^1$ is attached to a unique point $z\in\partial K_1$ (that is, $\overline{C}\cap\partial K_1=\{z\}$) and $z$ is an iterated pre-image of the parabolic fixed point $0$;
    \item for each $z\in\partial K_1$ which is an iterated pre-image of $0$, there is a unique connected component $C\subset E^1$ attached to it.
\end{itemize}
\end{thm}

It also comes natural to study parabolic implosion in parameter spaces in the following sense: when one perturbs a family of rational maps having a stable parabolic fixed point into families having a stable near parabolic fixed point, the enrichment of the bifurcation loci occurs. One typical result of this aspect is due to Buff-Ecalle-Epstein \cite{BEE}, who studied the enrichment of degenerate parabolic parameters for the family of quadratic rational maps. In \cite{CEP}, it has been asked whether one can describe the enrichment for the cubic polynomial family:
\begin{equation}\label{eq.family_lambda-a}
   f_{\lambda,a}(z) =  \lambda z+az^2 + z^3,\,\,(\lambda,a)\in\mathbb{C}^2.
\end{equation}

Following the notation introduced by Milnor \cite{Mi}, let us denote by $\mathrm{Per}_1(\lambda) := \{f_{\lambda,a};\,a\in\mathbb{C}\}$ the $\lambda$-slice of cubic polynomials. For each root of unity $\lambda$, the author defined in \cite{Z} a compact subset $\mathcal{K}_\lambda$ called the {\it central part} of the connectedness locus $\mathcal{C}_\lambda$ for $\mathrm{Per}_1(\lambda)$:
\begin{defn}[Parabolic component and central part]\label{defn.central-part}
     Let $\lambda$ be a root of unity. A parabolic component $\mathcal{U}\subset\mathrm{Per}_1(\lambda)$ is an open connected component of the following locus:
     $$\{a\in\C; \text{ both critical points of 
 }f_{\lambda,a} \text{ are attracted by }z=0\}.$$
     A parabolic component is called adjacent if both critical points are contained in the same Fatou component; it is called bitransitive if both critical points belong to two different periodic Fatou components in the same cycle.
     
     The \textbf{central part} $\mathcal{K}_\lambda$ is defined to be the closure of the union of all parabolic components $\mathcal{U}$ satisfying the following condition: there exists a simple curve $\gamma\subset\mathcal{C}_\lambda$ starting from $\mathcal{U}$ and ending at an adjacent or bitransitive parabolic component, such that $\#(\gamma\cap\partial\mathcal{C}_\lambda)<+\infty$ and $\gamma\setminus(\gamma\cap\partial\mathcal{C}_\lambda)$ is contained in a union of parabolic components.
\end{defn}

The central part $\mathcal{K}_\lambda$ can be characterized as the closure of the parameters $a$ such that $f_{\lambda,a}$ is {\it not} quadratic-like\footnote{Here “quadratic-like” is in the sense of Douady-Hubbard's polynomial-like renormalization \cite{DoHu2}.} renormalizable. It is shown in \cite{Z} that $\mathcal{K}_\lambda$ is dynamically modeled over the quadratic filled-in Julia set $K_\lambda$:
\begin{thm}[{\cite[Theorem A]{Z}}]
 Let $\lambda = e^{2\pi ip/q}$ be a root of unity. The central part $\mathcal{K}_\lambda$ is a locally connected full continuum. Moreover, there exists a dynamically defined double covering: 
\[\mathfrak{G}:\mathcal{K}_\lambda\setminus\mathcal{I}\longrightarrow K({\mathfrak{p}})\setminus\bigcup^{q-1}_{i=0} \overline{P_i}\] 
where $\mathcal{I}\subset\mathcal{K}_\lambda$ is a Jordan arc passing through $a=0$; $(P_i)_i$ are $q$ attracting parabolic petals contained respectively in the $q$ immediate basins of $\mathfrak{p}_\lambda$.
\end{thm}
Since we have just seen from the result of Lavaurs that $K_{\lambda_n}$ presents the parabolic implosion phenomenon, the above result together with numerical pictures provides strong evidence that parabolic implosion should occur in a corresponding way for $\mathcal{K}_{\lambda_n}$ as $\lambda_n$ tends to $1$.\\

In this article, we aim to answer partially the following question, which is the parameter analog to Question 1:
$$\text{\textbf{Question 1'.} Describe the difference }\mathcal{K}_1\setminus\limsup_n \mathcal{K}_{\lambda_n}.$$

Just like in the dynamical plane of $\mathfrak{p}_\lambda$ where one expects that $\limsup_n {K}_{\lambda_n} = K(L_{\mathfrak{p}_1,\tau})$, the correct candidate for $\limsup_n\mathcal{K}_{\lambda_n}$ should be the “central part” $\mathcal{K}_L$ for the cubic Lavaurs map family:
$$L_a = \lim\limits_{n\to+\infty}f_{\lambda_n,a}^{N_n}.$$ 

Although it seems coherent to keep in mind the idea that $\mathcal{K}_L$ is the closure of the collection of parameters $a$ such that $L_a$ is {\it not} “quadratic-like” renormalizable, it will be rather technical to define the notion “quadratic-like” for $L_a$ since $L_a$ may have a virtual attracting/parabolic basin or Siegel disk that touches the boundary of its maximal domain of definition. Instead, we consider directly the following {\it escaping locus} for the family $L_a$, which is the analog to the escaping region (\ref{eq.escaping-quadratic}):
$$\boldsymbol{\mathrm{Esc}} := \{a\in \mathrm{int}\,{\mathcal{K}_1};\,\exists N\geq 1  \text{ s.t. }L^N_a(c_1) \text{ or }L^N_a(c_2) \in \mathbb{C}\setminus K(f_{1,a})\},$$
where $c_1,c_2$ are the two critical points of $f_{1,a}$ and $K(f_{1,a})$ is the filled-in Julia set of $f_{1,a}$. Now we are ready to state our main results.

\begin{thmx}\label{thm.implosion-main1}
     Let $\lambda_n = e^{2\pi i\alpha_n}$ with $|\lambda_n|\leq 1$ and $\lim\limits_{n\to\infty}\lambda_n=1$. Suppose there exists $N_n\to+\infty$ such that 
     \begin{equation*}
         1/\alpha_n = N_n -\tau +o(1),\quad n\to+\infty,\quad\text{for some }\tau\in\C
     \end{equation*}
     (In fact, we have automatically $\mathfrak{Im}\,\tau\geq0$, see Lemma \ref{lem.assumtau}). Then $$\boldsymbol{\mathrm{Esc}}\subset \mathcal{K}_1\setminus\limsup_n \mathcal{K}_{\lambda_n}.$$ 
     Let $\mathcal{E}$ be a connected component of $\boldsymbol{\mathrm{Esc}}$ and $\mathbb{H} := \{z;\,\mathfrak{Re}\,z>0\}$ be the right half-plane. There exists a conformal dynamical parametrization $\Phi_\mathcal{E}:\mathcal{E}\longrightarrow \mathbb{H}$ that gives the position of the escaping critical point $c_*=c_1$ or $c_2$ under the Bõttcher coordinate $\phi^\infty_a$ of $f_{1,a}$ in the following sense: 
     $$\exp\,\circ\,\,\Phi_{\mathcal{E}}(a)= \phi^\infty_a\circ L_a(c_*).$$
\end{thmx}

We say that a parameter $a\in\mathrm{Per}_1(1)$ is {\it Misiurewicz parabolic} if one critical point of $f_{1,a}$ is an iterated pre-image of the parabolic fixed point $z=0$.
Let us set 
$$\boldsymbol{\mathrm{Esc}}^1 := \{a\in\boldsymbol{\mathrm{Esc}};\,L_a(c_1) \text{ or }L_a(c_2) \in \mathbb{C}\setminus K(f_{1,a})\}$$
to be the escaping locus of the first level. As an analog to Theorem \ref{thm.Lavaurs2}, we are able to give a more detailed topological description of $\boldsymbol{\mathrm{Esc}}^1$: 

\begin{thmx}\label{thm.implosion-main2}
      Under the same hypothesis as Theorem \ref{thm.implosion-main1}, let $\mathcal{E}$ be any connected component of $\boldsymbol{\mathrm{Esc}}^1$. Then 
      \begin{enumerate}[label=\upshape(\Roman*)]
          \item\label{thm.implosion-main2.first}${\partial\mathcal{E}}\cap\partial\mathcal{K}_1$ is a singleton and this point is either Misiurewicz parabolic or degenerate parabolic, i.e. $a=0$;
          \item\label{thm.implosion-main2.second} conversely, if $a_0\in\partial\mathcal{K}_1$ is Misiurewicz parabolic, then there is a unique component $\mathcal{E}\subset\mathbf{Esc}^1$ such that ${\partial\mathcal{E}}\cap\partial\mathcal{K}_1 = \{a_0\}$; if $a_0=0$, then there are exactly two components attached at $0$.

          \item\label{thm.implosion-main2.third} if $a_n\in\mathcal{E}$ is a sequence accumulating at $\partial\mathcal{E}$, then $a_n$ accumulates at $\partial\mathcal{K}_1$ if and only if $\Phi_{\mathcal{E}}(a_n)\to 0$ or $\infty$, where $\Phi_{\mathcal{E}}:\mathcal{E}\longrightarrow\mathbb{H}$ is the parametrization given by Theorem \ref{thm.implosion-main1}.
          
\vspace{0.2cm}

           If we assume furthermore that $\mathfrak{Im}\tau>0$ or $\tau\in\mathbb{Q}$, then:

\vspace{0.2cm}
           
           \item\label{thm.implosion-main2.fourth} if $\mathcal{E}'\subset\mathbf{Esc}^1$ is a component different from $\mathcal{E}$, then $\overline{\mathcal{E}}\cap\overline{\mathcal{E}'} = \emptyset$ unless both $\mathcal{E},\mathcal{E}'$ are attached to $a=0$;
          \item\label{thm.implosion-main2.fifth} $\mathcal{E}$ is a croissant (see Definition \ref{defn.croissant}).      
          \end{enumerate}
          
\end{thmx}

\iffalse
We believe that Theorem \ref{thm.implosion-main2} will be a key ingredient to prove that the escaping locus for $L_a$ is dynamically modeled over the escaping region of the quadratic Lavaurs map $L_{\mathfrak{p}_1,\tau}$:
\paragraph{\textbf{Parameter-dynamical correspondence.}} There is a natural dynamical double covering between $\overline{\boldsymbol{\mathrm{Esc}}^1}$ and $\overline{E^1}$, where $E^1 := \{z\in \mathrm{int}\,{K_1};\,L_{\mathfrak{p}_1,\tau}(z)\in\mathbb{C}\setminus K_1\}$ is the first level escaping region of $L_{\mathfrak{p}_1,\tau}$.\\

A key step to prove the above statement is to show that the boundary of every component of ${\boldsymbol{\mathrm{Esc}}^1}$ is the union of two Jordan curves. For components attached to Misiurewicz parabolic parameters, this follows from a transversality argument of holomorphic motion; for the component attached at $a=0$, one needs to prove a combinatorial rigidity theorem for parameters on $\partial\mathcal{E}$. However in both cases, we need to know a priori how $\partial \mathcal{E}$ intersects with $\partial\mathcal{K}_1$. This is where Theorem \ref{thm.implosion-main2} will play its role.
\fi

\begin{figure}[ht] %H为当前位置，!htb为忽略美学标准，htbp为浮动图形
\centering %图片居中 %最终文档中希望显示的图片标题
\includegraphics[width=0.7\textwidth]{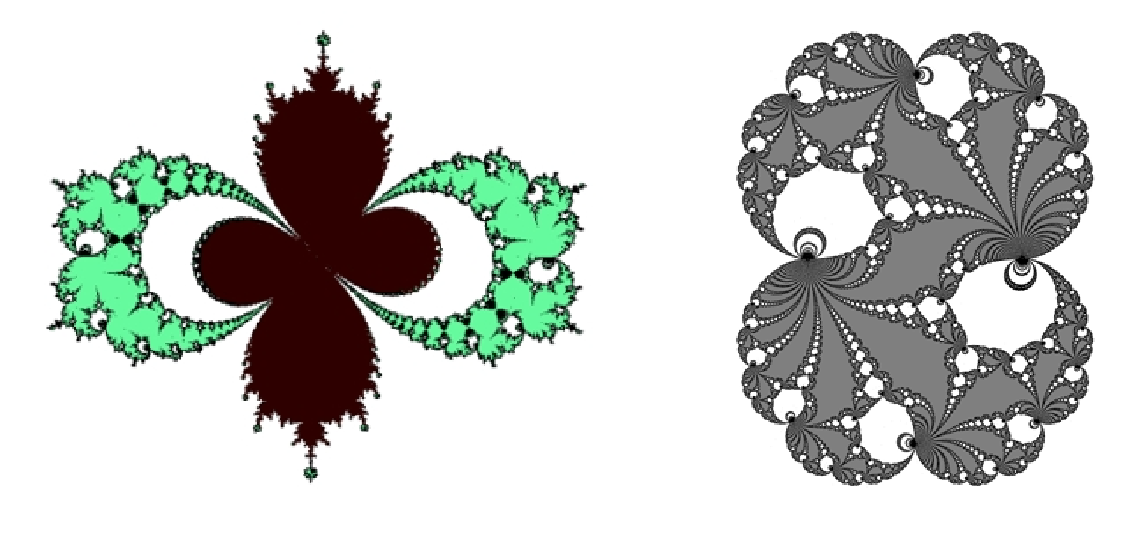}
\caption{Left: $\limsup_n\mathcal{K}_{\lambda_n}$. Right: $\limsup_n{K}_{\lambda_n}$ with $\lambda_n = e^{2\pi i/n}$.} 
\label{fig.implosion-paradym} %用于文内引用的标签
\end{figure}

\paragraph{\textbf{Idea of the proof for Theorem \ref{thm.implosion-main1} and \ref{thm.implosion-main2}}.} We define the parametrization $\Phi_\mathcal{E}$ in Theorem \ref{thm.implosion-main1} by locating the escaping critical value in the Böttcher-Lavaurs coordinate (Definition \ref{def.botthcer-lavaurs}). Using quasi-conformal deformation, we can find a parameter $a_*\in\mathcal{E}$, such that $\Phi_{\mathcal{E}}(a_*) = 1$. Then it turns out that the Branner-Hubbard motion with based point $1$ gives the inverse mapping of $\Phi_{\mathcal{E}}$. This shows that $\Phi_\mathcal{E}$ is a conformal isomorphism. 

The key point of Theorem \ref{thm.implosion-main2} is \ref{thm.implosion-main2.third}. The main subtlety here is that the sequence $a_n$ may converge to the degenerate parabolic parameter $a=0$, where the classical holomorphic dependence on Fatou coordinates fails. In fact, if we assume furthermore that the sequence $a_n$ does not accumulate to $a=0$, then \ref{thm.implosion-main2.third} will be much less harder (see Lemma \ref{lem.accumulation.inside}). In order to overcome this subtlety, we distinguish two types of components in $\boldsymbol{\mathrm{Esc}}^1$: disjoint type and adjacent type (see Definition \ref{def.disjoint-adjacent}). 
\begin{itemize}
    \item For disjoint type component $\mathcal{E}_{\mathrm{dis}}$, we will show that $0\not\in\partial\mathcal{E}_{\mathrm{dis}}$ (Lemma \ref{lem.nondouble.parabolic}). To do this, we need to analyze carefully the dynamics of near degenerate parabolic mappings.
    \item For adjacent type component $\mathcal{E}_{\mathrm{adj}}$, we will show that $0\in\partial \mathcal{E}_{\mathrm{adj}}$ and $\mathcal{E}_{\mathrm{adj}}$ does not intersect any other points on $\partial\mathcal{K}_1$ (Lemma \ref{lem.adjacent.accumulation}).
\end{itemize}

\paragraph{\textbf{Outline of the paper.}} In \S \ref{sec.oudkerk}, we outline the parabolic bifurcation theory for degenerate parabolic germs developed by R. Oudkerk \cite{Ou}. In \S \ref{sec.parabolic-degenerate}, we apply this theory to the cubic polynomial family. In \S \ref{sec.parabolicslice} we recall some results in the parabolic slice $\mathrm{Per}_1(1)$ established in \cite{Z}. \S \ref{sec.dynamical} is devoted to the description of the escaping region in the dynamical plane of cubic Lavaurs map. In \S \ref{sec.main1} and \S \ref{sec.main2}, we prove Theorem \ref{thm.implosion-main1} and \ref{thm.implosion-main2} respectively.\\

\section{Preparation}
\subsection{Fatou coordinates for degenerate parabolic germs}\label{sec.oudkerk}
In this section we recall some important results in \cite{Ou}. Let's consider the parabolic germ $$G_0(z) = z+z^{\nu+1}+\mathcal{O}(z^{\nu+2}).$$
We say that the germ is {\it degenerate parabolic} if $\nu\geq 2$. Take $r_0$ small enough such that $K_0 := \overline{\mathbb{D}_{2r_0}}\subset\mathrm{Dom}(G_0)$ and $0$ is the only fixed point in $K_0$. Take a small neighborhood $\mathcal{N}_0$ of $G_0$ with respect to the compact-open topology such that for all $G\in\mathcal{N}_0$, $K_0\subset\mathrm{Dom}(G)$. Shrink $\mathcal{N}_0$ if necessary, we may assume that every $G\in\mathcal{N}_0$ has the form
$$G(z) = z+(z-\sigma_0)^{m_0}(z-\sigma_1)^{m_1}...(z-\sigma_r)^{m_r}u_G(z),$$
where $\sigma_0,...,\sigma_r\in \mathbb{D}_{r_0}$, $\sum m_i = \nu+1$, $u_G(z)$ is close to $1$ for $z\in K_0$. Denote by $\mathrm{Fix}(G) = \{\sigma_0,...,\sigma_r\}$. For $G\in\mathcal{N}_0$, $\sigma\in\mathrm{Fix}(G)$, the {\it holomorphic index at $\sigma$} is defined by 
$$\iota(G,\sigma) := \frac{1}{2\pi i}\oint_\sigma\frac{dz}{z-f(z)}.$$ 
Notice that if $G'(\sigma)\not= 1$, then $\iota(G,\sigma) = \frac{1}{1-G'(\sigma)}$. The following lemma is elementary.
\begin{lem}\label{lem.iota.attracting}
    If $G'(\sigma)\not=1$, then $\sigma$ is an attracting fixed point if and only if
 $\mathfrak{Re}\,\iota(G,\sigma)>\frac{1}{2}$.
\end{lem}

Let $m(\sigma)$ be the multiplicity of $\sigma$. The {\it $\jmath$-index} is defined by  
\begin{equation}\label{eq.j-index}
    \jmath(G,\sigma):=\left\{
\begin{aligned}
&\frac{-2\pi i}{\log(1-\iota^{-1}(G,\sigma))}, \text{ if } m(\sigma) = 1  \\
&2\pi i(\iota(G,\sigma)-\frac{m}{2}), \text{ if } m(\sigma) > 1 
\end{aligned}\right.
\end{equation}

For $k=1,...,\nu$, set $z_{k,-} := r_0e^{2\pi i(k-1)/\nu}$, $z_{k,+} := r_0e^{\pi i/\nu}z_{k,-}$. Let 
$\epsilon\in\{+,-\}$. Consider the differential equation in real time 
$$\dot{z} = i(G(z)-z),\,\,z(0) = z_{k,\epsilon}.$$
By continuous dependence of solutions, shrinking $\mathcal{N}_0$ if necessary, there exists the first backward/forward moment $t_-<0<t_+$ such that the maximal solution $\gamma_{k,\epsilon,G}:I_{k,\epsilon,G}\longrightarrow\mathbb{C}$ satisfies $\gamma_{k,\epsilon,G}(t_+),\gamma_{k,\epsilon,G}(t_-)\in\partial\mathbb{D}_{r_0/2}$ and $\gamma_{k,\epsilon,G}$ intersects transversally $\partial\mathbb{D}_{r_0/2}$ at $t_+,t_-$, since $\gamma_{k,\epsilon,G_0}(t)$ converges to $0$ as $t\to+\infty$ and $-\infty$. Denote $l_{k,\epsilon,G} := \gamma_{k,\epsilon,G}(I_{k,\epsilon,G})$. We say that $G\in\mathcal{N}_0$ is {\it well-behaved}, denoted by $G\in\mathcal{WB}$, if $I_{k,\epsilon,G} = \mathbb{R}$ and $\gamma_{k,\epsilon,G_0}((-\infty,t_-))\subset\mathbb{D}_{r_0/2}$, $\gamma_{k,\epsilon,G_0}((t_+,+\infty))\subset\mathbb{D}_{r_0/2}$. From now on, the $\mathcal{N}_0$ and $r_0$ are always fixed.

\begin{lem}[{\cite[Corollary 3.7.5]{Ou}}]
      There exists $M>0$, such that $G\in\mathcal{WB}$ if $G\in\mathcal{N}_0$ satisfies the following condition:
    \[|\sum_{\sigma\in X}\mathfrak{Im}\,\iota(G,\sigma)|>M,\,\,\forall\,\, \emptyset\subsetneq X\subsetneq \mathrm{Fix}(G).\]
\end{lem}

\begin{thm}[{\cite[Proposition 2.3.2]{Ou}}]\label{thm.combinatorics.traject}
   Suppose $G\in\mathcal{WB}$, then 
\begin{enumerate}[label=\upshape(\Roman*)]
    \item\label{thm.combinatorics.traject:first} For each $k$ and $\epsilon$, $I_{k,\epsilon,G} = \mathbb{R}$, $\gamma_{k,\epsilon,G}(t)$ converges to $\sigma_{\pm}\in\mathrm{Fix}(G)$ as $t\to\pm\infty$. 
    \item\label{thm.combinatorics.traject:second} For any $\sigma\in\mathrm{Fix}(G)$, there exists $\gamma_{k,\epsilon,G}(t)$ such that $\gamma_{k,\epsilon,G}(-\infty) = \sigma$ or $\gamma_{k,\epsilon,G}(+\infty) = \sigma$. 
    \item\label{thm.combinatorics.traject:third} For all $k$, $\gamma_{k,-,G}(+\infty) = \gamma_{k,+,G}(+\infty)$, $\gamma_{k,-,G}(-\infty) = \gamma_{k-1,+,G}(-\infty)$ ($\gamma_{0,+,G} := \gamma_{\nu,+,G}$).
    \item\label{thm.combinatorics.traject:fourth} For each $k$ and $\epsilon$, either $\gamma_{k,\epsilon,G}(+\infty) = \gamma_{k,\epsilon,G}(-\infty) = \sigma$ (in which case $\sigma$ is a parabolic fixed point), or there exists a unique $j\not=k$, $\overline{\epsilon}\not=\epsilon$, such that $$\gamma_{k,\epsilon,G}(-\infty) = \gamma_{j,\overline{\epsilon},G}(-\infty)\not=\gamma_{k,\epsilon,G}(+\infty)=\gamma_{j,\overline{\epsilon},G}(+\infty).$$
    \item\label{thm.combinatorics.traject:fifth} For each $k$ and $\epsilon$, $l_{k,\epsilon,G}\cap G(l_{k,\epsilon,G}) = \emptyset$. Denote by $S_{k,\epsilon,G}$ the closed bounded Jordan domain surrounded by $\overline{l_{k,\epsilon,G}\cup G(l_{k,\epsilon,G})}$ not containing $\infty$. These different $S_{k,\epsilon,G}$ can only intersect one another at $\mathrm{Fix}(G)$. The set $S'_{k,\epsilon,G} := S_{k,\epsilon,G}\setminus(\gamma_{k,\epsilon,G}(+\infty)\cup\gamma_{k,\epsilon,G}(-\infty))$ is called a fundamental region.
\end{enumerate}
\end{thm}

For $G\in\mathcal{WB}$, we associate it with the {\it gate structure} $\mathbf{gate}(G) := (\mathrm{gate}_1(G),...,\mathrm{gate}_\nu(G))$, where
\[\mathrm{gate}_k(G):=\left\{
\begin{aligned}
&*, \text{ if } \overline{l_{k,+,G}} \text{ is a Jordan curve} \\
&j, \text{ if } \overline{l_{k,+,G}}\cup\overline{l_{j,-,G}} \text{ is a Jordan curve} 
\end{aligned}\right.
\]
A vector $\mathscr{G} = (\mathscr{G}_1,...,\mathscr{G}_\nu)$ is called {\it admissible} if it is realised by the gate structure of some $G\in \mathcal{WB}$. The collection of $G\in\mathcal{WB}$ having the gate structure $\mathscr{G}$ is denoted by $\mathcal{WB}(\mathscr{G})$.

\begin{defn}[The set $U_{k,\epsilon,G}$]
    For each $k$, if $\mathrm{gate}_k(G) = *$, define the set $U_{k,+,G}$ to be the Jordan domain bounded by $G^{-2}(\overline{l_{k,+,G}})$; if $\mathrm{gate}_k(G) = j$, define the $U_{k,+,G} = U_{j,-,G}$ to be the Jordan domain bounded by $G^{-2}(\overline{l_{k,+,G}})\cup G^{2}(\overline{l_{j,-,G}})$; finally if $j$ does not appear in $\mathbf{gate}(G)$, then define $U_{j,-,G}$ to be the Jordan domain bounded by $G^{2}(\overline{l_{j,-,G}})$.
\end{defn}

\begin{defn}[$\mathrm{Fix}^u(k,G)\text{ and }\mathrm{Fix}^l(k,G)$]
Let $G\in\mathcal{WB}$, $k$ be such that $\mathrm{gate}_k(G)\not = *$. Notice that $\mathbb{D}_{r_0/2}\setminus{U_{k,+,G}}$ has two connected components. Denote by $\mathrm{Upp}(k,G)$ the one containing $\gamma_{k,+,G}(+\infty)$ and $\mathrm{Low}(k,G)$ the other. Define 
$$\mathrm{Fix}^u(k,G) :=\mathrm{Fix}(G)\cap\mathrm{Upp}(k,G),\,\,\mathrm{Fix}^l(k,G) :=\mathrm{Fix}(G)\cap\mathrm{Low}(k,G).$$
\end{defn}

\begin{thm}[{\cite[Proposition 2.3.14, Lem. 3.4.1]{Ou}}]
    Let $G\in\mathcal{WB}$. For each $k,\epsilon$, there exists an analytic univalent map (called the Fatou coordinate) $\phi_{k,\epsilon,G}:U_{k,\epsilon,G}\longrightarrow\mathbb{C}$ such that \[\phi_{k,\epsilon,G}(G(z)) = \phi_{k,\epsilon,G}(z)+1,\,\,z,G(z)\in U_{k,\epsilon,G}.\]
    Moreover $\phi_{k,\epsilon,G}$ is unique up to adding a constant (in particular $\phi_{k,+,G}$ and $\phi_{j,-,G}$ differ by a constant if $\mathrm{gate}_k(G) = j$), $\phi_{k,\epsilon,G}(l_{k,\epsilon,G})$ is almost a vertical line: the error is controlled by $\mathcal{O}(w^{-\beta})$ for some $0<\beta<1$ (not depending on $G$) as $\mathfrak{Im}\,w\to\infty$.
\end{thm}

Let $G\in\mathcal{WB}(\mathscr{G})$. Define the $\mathscr{G}_k$-{\it lifted phase} by  
\[\tilde{\tau}_k(G):=\left\{
\begin{aligned}
&\phi_{j,-,G}-\phi_{k,+,G}, \text{ if } \mathscr{G}_k = j\\
&\infty, \text{ if } \mathscr{G}_k = *
\end{aligned}\right.
\]

Fix a large $\eta>0$ (not depending on $G$). Define the {\it upper/lower horn} by 
$$S^u_{k,\epsilon,G} := \{z\in S'_{k,\epsilon,G};\,\mathfrak{Im}\,\phi_{k,\epsilon,G}>\eta\}$$
$$S^l_{k,\epsilon,G} := \{z\in S'_{k,\epsilon,G};\,\mathfrak{Im}\,\phi_{k,\epsilon,G}<-\eta\}.$$ 
\begin{thm}[{\cite[Lem. 3.4.5]{Ou}}]\label{Thm.hornmap}
      If $\eta>0$ is large enough, then for any $z\in S^u_{k,-,G}$, there exists a smallest integer $p$ such that $G^p(z)\in S^u_{k,+,G}$. For any given Fatou coordinates $\phi_{k,\epsilon,G}$, define the {\it lifted horn map} by
      $$\tilde{E}^{k,u}_G: \phi_{k,-,G}(S^u_{k,-,G})\longrightarrow\mathbb{C},\,\,w\mapsto \phi_{k,+,G}\circ G^p\circ\phi^{-1}_{k,-,G}(w)-p.$$
      The lifted horn map can be extended to $\{\mathfrak{Im}\,w>\eta\}$ and satisfies $\tilde{E}^{k,u}_G(w+1) = \tilde{E}^{k,u}_G(w)+1$.

      Similarly one can define $\tilde{E}^{k,l}_G$. There exists constants $L^{k,u}_G,L^{k,l}_G$ such that $$\lim_{\mathfrak{Im}\,w\to+\infty}\tilde{E}^{k,u}_G(w)=L^{k,u}_G,\,\,\lim_{\mathfrak{Im}\,w\to-\infty}\tilde{E}^{k,l}_G(w)= L^{k,l}_G.$$
Moreover $$\sum_k (L^{k,l}_G - L^{k,u}_G) = -\sum_{\sigma\in\mathrm{Fix}(G)}\jmath(G,\sigma).$$
\end{thm}

Now we can introduce the {\it preferred normalization} of Fatou coordinates.
\begin{thm}[{\cite[Proposition 2.3.14, Proposition 2.4.13]{Ou}}]\label{thm.oudkerk}
     Let $\mathscr{G} = (\mathscr{G}_1,...,\mathscr{G}_\nu)$ be an admissible vector. For $G\in\mathcal{WB}(\mathscr{G})$, there exists a normalization of $\phi_{k,\epsilon,G}$, such that 
     \begin{enumerate}[label=\upshape(\Roman*)]
         \item\label{thm.oudkerk:first} $G\mapsto (\phi_{k,\epsilon,G}:U_{k,\epsilon,G}\longrightarrow\mathbb{C})$ is continuous with respect to the compact-open topology, 
         \item\label{thm.oudkerk:second} $G\mapsto\overline{U_{k,\epsilon,G}}$ is continuous with respect to the Hausdorff topology,
         \item\label{thm.oudkerk:third} all the constants $L^u_{k,G},L^l_{k,G}$ in Theorem \ref{Thm.hornmap} are 0 except for $L^{l}_{1,G}$. 
         \item\label{thm.oudkerk:fourth} let $\sigma_0(G) := \gamma_{1,-,G}(-\infty)$, then 
         \[
         \tilde{\tau}_k(G):=\left\{
         \begin{aligned}
         &+\sum_{\sigma\in\mathrm{Fix}^u(k,G)}\jmath(G,\sigma), \text{ if } \sigma_0(G)\not\in {Fix}^u(k,G)\\
         &-\sum_{\sigma\in\mathrm{Fix}^l(k,G)}\jmath(G,\sigma), \text{ if } \sigma_0(G)\not\in {Fix}^l(k,G).
         \end{aligned}\right.
         \]     
         \item\label{thm.oudkerk:fifth} let $G_n\in\mathcal{WB}(\mathscr{G})$ be a sequence such that $\tilde{\tau}_k(G_n)\to-\infty$ for all $k$, then 
         \[
         U_{k,s,G_n}\to\left\{
         \begin{aligned}
         &\overline{U_{k,+,G_0}}\cup\overline{U_{j,-,G_0}}, \text{ if } s=+,\,j=\mathscr{G}_k\not=*\\
         &\overline{U_{k,-,G_0}}\cup\overline{U_{j,-,G_0}}, \text{ if } s=-,\,\exists j\text{ s.t. }\mathscr{G}_j=k\\
         &\overline{U_{k,s,G_0}}\quad\quad\quad\quad\,\,\,\,\text{ otherwise.} 
         \end{aligned}\right.
         \]     
         \[\Phi_{k,s,G_n}:U_{k,s,G_n}\longrightarrow\mathbb{C}\text { converges to } \Phi_{k,s,G_0}:U_{k,s,G_0}\longrightarrow\mathbb{C}.
         \]
         \item\label{thm.oudkerk:sixth} Under the same hypothesis in \ref{thm.oudkerk:fifth}, for every $z\in \bigcap_{n}S'_{k,+,G_n}$ and $n\geq 1$ there exists a smallest integer $N_n$ such that 
         $$\phi_{k,-,G_n}(G_n^{N_n}(z)) = \phi_{k,+,G_n}(z)+N_n+\tau_k(G_n).$$
         Moreover $N_n$ tends to infinity as $n$ goes to infinity.
\end{enumerate}
\end{thm}

\begin{defn}[Preferred normalization]\label{def.prefered.nor}
      Let $\tilde{\phi}_{k,\epsilon,G}$ be the Fatou coordinates normalized by $\tilde{\phi}_{k,\epsilon,G}(z_{k,\epsilon}) = 0$. Starting from $\tilde{\phi}_{1,+,G}$, modify in the counter-clockwise direction the normalization of all $\tilde{\phi}_{k,\epsilon,G}$ except for $\tilde{\phi}_{1,-,G}$ such that Theorem \ref{thm.oudkerk} \ref{thm.oudkerk:third} holds. This modified normalization is called the preferred normalization.  
\end{defn}

\subsection{Gate structure for near parabolic-degenerate cubic polynomials}\label{sec.parabolic-degenerate}
Consider the parabolic cubic polynomial family
\begin{equation}\label{eq.family_a1}
   f_{a}(z) =  z+az^2 + z^3,\,\,a\in\mathbb{C}.
\end{equation}
For $a\in\mathbb{C}^*$, define $\tau(a) := \frac{-2\pi i}{\log(1+a^2)}$ with $\mathfrak{Im}\,\log(\cdot)\in(-\pi,\pi]$.

We deduce from the theory in Section \ref{sec.oudkerk} the following:
\begin{lem}\label{lem.perturb.Fatoucoordi}
Fix $0< \eta\ll 1$. Let $S_\eta\subset\mathbb{C}$ be the open sector in the first quadrant bounded by rays $y = \tan(\eta)\cdot x$ and $y = \tan(\frac{\pi}{2}-\eta)\cdot x$. Set $\Omega_{\eta,r} = S_\eta\cap{\mathbb{D}}(r)$. Then for $r$ small (not depending on $\eta$)
\begin{itemize}
    \item $f_a$ has gate structure $(*,2)$ for $a\in\overline{\Omega_{\eta,r}}\setminus\{0\}$: there are simply connected non intersecting domains $P_a,P_{1,\pm,f_a}$ with piece-wise smooth boundaries ($P_a = P_{2,-,f_a} = P_{2,+,f_a}$), such that $\{0\} = \partial P_a\cap\partial P_{1,+,f_a}\cap\partial P_{1,-,f_a}$, $-a\in\partial P_a$. Moreover $\overline{P_{1,\pm,f_a}},\overline{P_{a}}$ vary continuously in $a$ with respect to the Hausdorff topology.
    
    \item for $a\in\overline{\Omega_{\eta,r}}\setminus\{0\}$, there are attracting $(+)$ and repelling $(-)$ Fatou coordinates with the preferred normalization (see Definition \ref{def.prefered.nor}) $\phi_{1,\pm,f_a}: P_{1,\pm,f_a}\longrightarrow \mathbb{C}$, $\phi_{2,\pm,f_a}: P_{a}\longrightarrow \mathbb{C}$ varying analytically for $a\in{\Omega_{\eta,r}}$ and continuously for $a\in\overline{\Omega_{\eta,r}}\setminus\{0\}$; Moreover $\phi_{2,-,f_a} = \phi_{2,+,f_a}+\tau(a)$, $\phi_{1,\pm,f_a},\phi_{2,\pm,f_a}$ converge in the compact-open topology to $\phi_{k,\pm,f_0}: P_{k,\pm,f_0}\longrightarrow \mathbb{C}$, $k=1,2$; $\lim\limits_{a\to0}\overline{P_{1,\pm,f_a}}=\overline{P_{1,\pm,f_0}}$, $\lim\limits_{a\to0}\overline{P_{a}}=\overline{P_{2,+,f_0}\cup P_{2,-,f_0}}$ in the Hausdorff topology.
\end{itemize}
\end{lem}

\begin{proof}
Consider the differential equation of real time
\begin{equation}\label{eq.vectorfield.f_a}
    \dot{z} = i(f_a(z)-z),\,\,z(0) = z_{k,\epsilon}.
\end{equation}
Its equilibrium is $z=0$ and $z=-a$. By a direct computation, $z=-a$ is a sink if $\mathfrak{Im}\,f'(-a)>0$ and a source if $\mathfrak{Im}\,f'(-a)<0$. When $a\in\Omega_{\eta,r}$, $z=-a$ is a sink. Therefore by Theorem \ref{thm.combinatorics.traject} \ref{thm.combinatorics.traject:first}, the four backward trajectories ($t\to-\infty$) of (\ref{eq.vectorfield.f_a}) must converge to $z=0$. Now $$\mathbb{D}_{r_0/2}\setminus\bigcup_{k,\epsilon}\gamma_{k,\epsilon,f_a}([-\infty,0])$$
has four connected components, among which there is a unique component $V$ containing $-a$. We name these for components $V_1,V_2,V_3,V_4$ in the counter clockwise direction with $\partial V_1$ intersecting $\gamma_{1,+,f_a},\gamma_{1,-,f_a}$. 

In order to prove that $f_a$ has gate structure $(*,2)$ for all $a\in\Omega_{\eta,r}$, it suffices to find one parameter $a_0$ having gate structure $(*,2)$ since the gate structure is an open property. Consider the parameter $a_0 = re^{\pi i/4}$ with some small $r$ to be determined latter. To prove $f_{a_0}$ has gate structure $(*,2)$, it suffices to show that the fixed point $-a_0$ belongs to $V_3$, and this is reduced to show that $\gamma_{1,-,f_{a_0}}(+\infty) = 0$. 

Let us consider Equation (\ref{eq.vectorfield.f_a}) with $a = a_0$ and recall the initial condition $z(0) = z_{1,-} = r_0>0$ (which does not depend on $r$). One can solve directly the equation and get 
\begin{equation}\label{eq.solution}
    Y(z)-Y(r_0) = it,\,\,\text{where}\,\,
Y(z) = \frac{(z\log(a_0+z)-a_0-z\log(z))}{a_0^2z}.
\end{equation}
Notice that the logarithm is defined along the trajectory of $z$ and $$\log(a_0+z) = \log|a_0+z| + i(\arg(a_0+z) + 2k_z\pi),$$
where $k_z$ is an integer that depends on $z$. Suppose by contradiction that $z\to -a_0$ as $t\to +\infty$. Comparing the real part on both sides of (\ref{eq.solution}), it is not hard to see by plugging in $a_0 = re^{\pi i/4}$ that $k_z$ has to be $0$ for $t$ large enough. Take the real part on both sides of (\ref{eq.solution}). A straightforward computation gives
$$\frac{3\pi}{4r^2}+\left(\lim_{z\to-a_0}\frac{\arg(z+a_0)}{r^2}\right)+\frac{\sqrt{2}}{2rr_0}-\frac{\arg(r_0+a_0)}{r^2} = 0.$$
But clearly this cannot hold for $r$ small enough, leading to a contradiction.

\end{proof}

Under a similar proof as above, we can get:

\begin{lem}\label{lem.perturb.Fatoucoordi'}
Under the same hypothesis as in Lemma \ref{lem.perturb.Fatoucoordi}, we have
\begin{itemize}
    \item $f_a$ has gate structure $(2,*)$ for $a\in\overline{-i\Omega_{\eta,r}}\setminus\{0\}$: there are simply connected non intersecting domains $P_a,P_{2,+,f_a},P_{1,-,f_a}$ with piece-wise smooth boundaries ($P_a = P_{2,-,f_a} = P_{2,+,f_a}$), such that $\{0\} = \partial P_a\cap\partial P_{2,+,f_a}\cap\partial P_{1,-,f_a}$, $-a\in\partial P_a$. Moreover $\overline{P_{2,+,f_a}},\overline{P_{1,-,f_a}},\overline{P_{a}}$ vary continuously in $a$ with respect to the Hausdorff topology.
    
    \item for $a\in\overline{-i\Omega_{\eta,r}}\setminus\{0\}$, there are attracting $(+)$ and repelling $(-)$ Fatou coordinates with the preferred normalization $\phi_{1,+,f_a},\phi_{2,-,f_a}: P_{a}\longrightarrow \mathbb{C}$, $\phi_{2,+,f_a}: P_{2,\pm,f_a}\longrightarrow \mathbb{C}$, $\phi_{1,-,f_a}: P_{1,-,f_a}\longrightarrow \mathbb{C}$ varying analytically for $a\in{-i\Omega_{\eta,r}}$ and continuously for $a\in\overline{-i\Omega_{\eta,r}}\setminus\{0\}$; Moreover $\phi_{2,-,f_a} = \phi_{1,+,f_a}-\tau(a)$, $\phi_{1,\pm,f_a},\phi_{2,\pm,f_a}$ converge in the compact-open topology to $\phi_{k,\pm,f_0}: P_{k,\pm,f_0}\longrightarrow \mathbb{C}$, $i=1,2$; $\lim\limits_{a\to0}\overline{P_{2,+,f_a}}=\overline{P_{2,+,f_0}}$, $\lim\limits_{a\to0}\overline{P_{1,-,f_a}}=\overline{P^{-}_{0,1}}$, $\lim\limits_{a\to0}\overline{P_{a}}=\overline{P_{1,+,f_0}\cup P_{2,-,f_0}}$ in the Hausdorff topology.
\end{itemize}
\end{lem}
\begin{figure}[ht]%H为当前位置，!htb为忽略美学标准，htbp为浮动图形
\centering %图片居中
\includegraphics[width=0.7\textwidth]{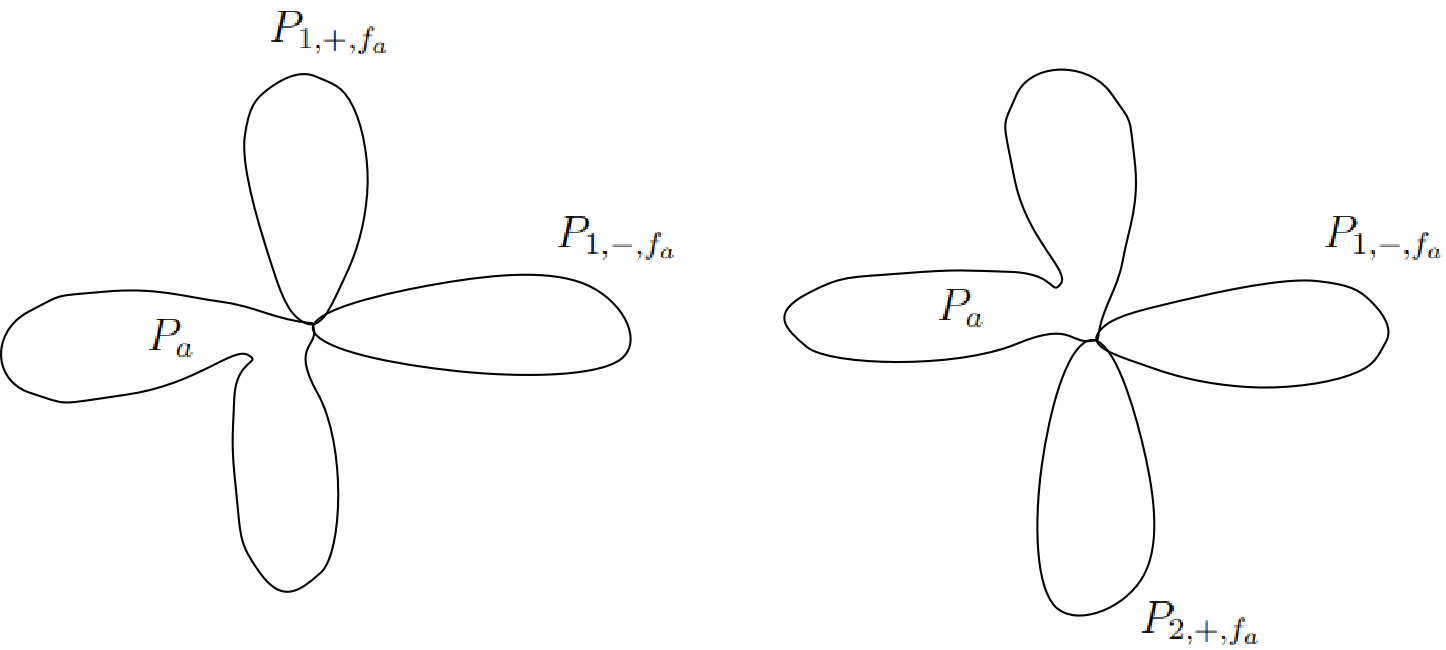} %插入图片，[]中设置图片大小，{}中是图片文件名
\caption{Gate structure $(*,2)$ on the left and $(2,*)$ on the right.} %最终文档中希望显示的图片标题
\label{fig.gate-structure} %用于文内引用的标签
\end{figure}

Using the same argument as \cite[Proposition 2.2]{Ta}, one can show the following:
\begin{lem}[Orbit correspondence]\label{lem.orbit.correspond}
Let $X\subset P_{2,+,f_0}$ and $Y\subset P_{2,-,f_0}$ be compact sets. Let $\Omega_{\eta,r}$ and $\phi_{k,\pm,f_a}$ be as in Lemma \ref{lem.perturb.Fatoucoordi}. If $\eta$ and $r$ are sufficiently small, then given any continuous functions $x:\overline{\Omega_{\eta,r}}\times[0,1]\longrightarrow X$ and $y:\overline{\Omega_{\eta,r}}\times[0,1]\longrightarrow Y$, holomorphic in $a$ for $a\in\Omega_{\eta,r}$, the equations
\begin{equation}\label{eq.orbit-correspondence}
    f^n_a(x(a,t)) = y(a,t) \text{ and } \phi_{2,+,f_a}(x(a,t)) -\tau(a) + n = \phi_{2,-,f_a}(y(a,t))
\end{equation}
have a unique solution $a_n(t)\in\Omega_{\eta,r}$ in common which depends continuously on $t$, provided that $n$ is large enough. Moreover, $a_n(t)\to 0$ uniformly on $t$ when $n\to\infty$. 
\end{lem}

\begin{lem}\label{lem.escaping.strip}
   Let $\psi_{1,-,f_a}$ be the extension of $(\phi_{1,-,f_a})^{-1}$ by the relation $f_a\circ\psi_{1,-,f_a}(z) = \psi_{1,-,f_a}(z+1)$.
   Then for all $\tau\in\C$ with $\mathfrak{Im}\,\tau\geq0$ and $a\in\Omega_{\eta,r}$, there exists a region $U_{\tau,a}\subset P_{1,+,f_a}$ such that 
   \begin{itemize}
       \item $\psi_{1,-,f_a}\circ T_\tau \circ\phi_{1,+,f_a}(U_{\tau,a})\subset K^c_a$;
       \item $\bigcup_{n\in\mathbb{Z}}(\phi_{1,+,f_a}(U_{\tau,a})+n)$ contains a uniform $T_1$-invariant strip $B$ (not depending on $a$);
       \item the pull back of $U_{\tau,a}$ by $f_{a}$ will enter $P_a$.
   \end{itemize} 
\end{lem}
\begin{proof}
    It is easy to see that $\jmath(f_0,0) = -3\pi i$ (recall (\ref{eq.j-index})). If we take the Fatou coordinates $\tilde{\phi}_{k,\epsilon,f_0}$ with normalization $\tilde{\phi}_{k,\epsilon,f_0}(z_{k,\epsilon}) = 0$ (notice that this is different from the preferred nomalisation in Definition \ref{def.prefered.nor}), then by symmetry, $\tilde{\phi}_{2,+,f_0}(z)=\tilde{\phi}_{1,+,f_0}(-z)$, $\tilde{\phi}_{2,-,f_0}(z)=\tilde{\phi}_{1,-,f_0}(-z)$. Therefore the corresponding constants $\tilde{L}^{k,l}_{f_0},\tilde{L}^{k,u}_{f_0}$ associated to the lifted horn maps (recall Theorem \ref{Thm.hornmap}) satisfy 
    $$\tilde{L}^{1,u}_{f_0} = \overline{\tilde{L}^{1,l}_{f_0}} = \tilde{L}^{2,u}_{f_0} = \overline{\tilde{L}^{2,l}_{f_0}}.$$
    By Theorem \ref{Thm.hornmap}, we have 
    $$L^{1,l}_{f_0}-L^{1,u}_{f_0}+L^{2,l}_{f_0}-L^{2,u}_{f_0} = 3\pi i.$$
    Set $w := \tilde{L}^{1,u}_{f_0}$.
    Thus $\mathfrak{Im}\,w = -\frac{3\pi}{4}$. Then the preferred Fatou coordinates satisfy 
    $$\phi_{1,-,f_0} = \tilde{\phi}_{1,-,f_0},\,\,\phi_{1,+,f_0} = \tilde{\phi}_{1,+,f_0}-\mathfrak{Re}\,w+\frac{3\pi i}{4},$$
    $$\phi_{2,-,f_0} = \tilde{\phi}_{2,-,f_0}+\frac{3\pi i}{2},\,\,\phi_{2,+,f_0} = \tilde{\phi}_{2,+,f_0}-\mathfrak{Re}\,w+\frac{9\pi i}{4}.$$
    On the other hand, $\psi_{1,-,f_0}(\mathbb{R})$ is the $0$-angle external ray of $f_0$ by symmetry. Therefore there exists a $T_1$-invariant open strip $\tilde{B}$ containing $\mathbb{R}$ such that $\psi_{1,-,f_0}(\tilde{B})\subset\mathbb{C}\setminus K_{0}$. By continuity dependence of Fatou coordinate (Lemma \ref{lem.perturb.Fatoucoordi}), we can take $\tilde{B}$ smaller if necessary so that for all $a\in\Omega_{\eta,r}$, $\psi_{1,-,f_a}(\tilde{B})\subset \mathbb{C}\setminus K(f_{1,a})$. By symmetry, $\tilde{\phi}_{1,+,f_0}^{-1}(\tilde{B})$ contains $i\mathbb{R}\cap P_{1,+,f_0}$. Therefore by taking $\tilde{B}$ smaller again if necessary, ${\phi}_{1,+,f_0}^{-1}\circ T_{-\tau}(\tilde{B})$ contains a region that can be written as $S\setminus \overline{S'}$, where $S'\subset S$ are sepals of $f_0$ intersecting with $P_{1,+,f_0}$ and $P_{2,-,f_0}$. The lemma follows by continuity dependence on the Fatou coordinate.
\end{proof}

\begin{prop}\label{Propositionexistence}
For any $\tau\in\C$ with $\mathfrak{Im}\,\tau\geq0$, there exists a curve $\mathcal{G}_\tau$ converging to the degenerate parabolic parameter $0$ so that for all $a\in\mathcal{G}_\tau$ 
\begin{itemize}
    \item two critical points of $f_a$ are contained in the immediate parabolic basin of $0$;
    \item there is a critical point $c_a$ of $f_a$ such that $\psi_{1,-,f_a}\circ T_\tau\circ\phi_{1,+,f_a}(c_a)\in \mathbb{C}\setminus K(f_{a})$.
\end{itemize}
 
\end{prop}
\begin{proof}
By symmetry, $\mathbb{R}_{>0}$ coincides with the $0$-angle external ray $R_0$ of $f_0$. Hence $R_0$ will enter $P_{1,-,f_0}$. Let $\tilde{Y}_0\subset R_0\cap P_{1,-,f_0}$ be a closed subinterval parametrized by $\tilde{y}(t)$, $t\in[0,1]$ such that $f_0(\tilde{y}(0)) = y(1)$. Let $h:\overline{\Omega_{\eta,r}}\times R_0\longrightarrow\C$ be the continuous motion of the external ray $R_0$ induced by the Böttcher coordinate. We define 
\begin{equation}\label{eq.shooting.externalray}
    \hat{Y}_a := \phi_{1,+,f_a}^{-1}\circ T_{-\tau}\circ\phi_{1,-,f_a}(h(a,\tilde{Y}_0)).
\end{equation}
By Lemma \ref{lem.escaping.strip} and continuity dependence on the Fatou coordinate, there exists $N$ such that for all $a\in\overline{\Omega_{\eta,r}}$, we can pull back univalently $\hat{Y}_a$ by $f_a^N$ to get $Y_a$, such that $Y_a\Subset P_a\cap P_{2,-,f_0}$. Let $y(a,t)$ be the induced parametrization of $Y_a$.

Let $c_0$ be the critical point of $f_0$ contained in the lower-half plane. Notice that there is an analytic continuation of the critical point $c_a$ of $f_a$ for $a\in\Omega_{\eta,r}$ so that $c_a\to c_0$ as $a\to0$. Since there exists an $M$ such that $f^M_0(c_0)\in P_{2,+,f_0}$, one may take $r$ smaller if necessary such that $f^M_a(c_a)\in  P_{2,+,f_0}$ for all $a\in\overline{\Omega_{\eta,r}}$. Apply Lemma \ref{lem.orbit.correspond} to $y(a,t)$ and $f^M_a(c_a)$, we get $a_n(t)$ such that $f^{n+M}_{a_n(t)}(c_{a_n(t)}) = y(a_n(t),t)$. Notice that $f^{n+1+M}_{a_n(0)}(c_{a_n(0)}) = y(a_n(0),1)$. Thus by uniqueness of $a_n(t)$, $a_{n+1}(1) = a_n(0)$. Therefore $\mathcal{G}_\tau := \bigcup_n \{a_n(t);\,t\in[0,1]\}$ is a curve converging to $a=0$. The same argument as in the end of the proof of \cite[Lem. 3.6]{Ta} will show that if one extends $y(a_n(t),t)$ by pulling back by $f^{n+M}_{a_n(t)}$, then $c_{a_n(t)}$ belongs to the extended curve. This implies that $c_{a_n(t)}$ is also contained in the immediate parabolic basin of $0$. By Lemma \ref{lem.escaping.strip}, we have $\psi_{1,-,f_a}\circ T_\tau\circ\phi_{1,+,f_a}(c_a)\in \mathbb{C}\setminus K(f_{a})$.
\end{proof}

For $\lambda$ close to 1, $z=0$ splits into two fixed points $0$ and $\sigma_{\lambda,a}$ for $f_{\lambda,a}$. The following lemma gives a condition when $\sigma_{\lambda,a}$ is attracting:
\begin{lem}\label{lem.parabolic.attracting}
Let $\lambda_n = e^{2\pi i\alpha_n}\in\overline{\mathbb{D}}$ converge to $1$ ($\mathfrak{Im}\,\alpha_n\geq 0$) and $N_n\to+\infty$, $\tau\in\mathbb{C}$ such that
$$N_n-\frac{1}{\alpha_n} = \tau+o(1).$$
Then for all $a\in\C$ satisfying 
\begin{equation}\label{eq.ctau}
    |a^2-\frac{c_\tau}{2}|<\frac{c_\tau}{2}\,\quad\text{ where }\quad c_\tau = \frac{1}{1+\frac{1}{2\pi}\mathfrak{Im}\,\tau},
\end{equation}
$\sigma_{\lambda_n,a}$ is always an attracting fixed point for $n$ large enough.
\end{lem}
\begin{proof}
For $\lambda = e^{2\pi i \alpha}$, it is easy to calculate the holomorphic indices $$\mathfrak{Re}\,\iota(f_{\lambda,a},\sigma_{\lambda,a}) =\frac{1}{2}-\frac{1}{2}\mathfrak{Im}\,\frac{\sin2\pi\alpha}{1-\cos2\pi\alpha},\quad\iota(f_a,0) = \frac{1}{a^2}.$$ 
Hence if $n$ is large enough, $\mathfrak{Re}\,\iota(f_{\lambda_n,a},\sigma_{\lambda_n,a}) = \frac{1}{2}+\frac{1}{2\pi}\mathfrak{Im}\,\tau$; if $a$ satisfies (\ref{eq.ctau}), then $\mathfrak{Re}\,\iota(f_a,0)>1+\frac{1}{2\pi}\mathfrak{Im}\,\tau$. By Rouché's Theorem,
$$\lim_{n\to\infty}\iota(f_{\lambda_n,a},\sigma_{\lambda_n,a})+\iota(f_{\lambda_n,a},0) = \iota(f_a,0).$$
Taking the real part on both sides, we get $\mathfrak{Re}\,\iota(f_{\lambda,a},\sigma_{\lambda,a})>\frac{1}{2}$. Hence $\sigma_{\lambda_n,a}$ is attracting by Lemma \ref{lem.iota.attracting}.
\end{proof}

\subsection{Parabolic slice $\mathrm{Per}_1(1)$}\label{sec.parabolicslice}
From now on to the end of the article, we will work with the following more convenient family
\begin{equation}\label{eq.family-gs}
    {g}_{s}(z) =  z\left(1-\frac{s+1/s}{2}z + \frac{1}{3}z^2\right),\,\,s\in \mathbb{C}^*,
\end{equation}
where the two finite critical points $s,1/s$ of $g_s$ are marked out. We say that a parameter $s$ is {\it Misiurezwicz parabolic} if one of the critical points $s,1/s$ will eventually hit $z=0$ under the iteration of $g_s$. The families (\ref{eq.family-gs}) and (\ref{eq.family_a1}) have the following obvious relation:

\begin{lem}\label{lem.relation-gs-fa}
    $g_{s}$ is conjugate to $f_{a}$ by $z\mapsto \sqrt{3}\cdot z$ with $a = \sigma(s) := - \sqrt{{3}}\cdot\frac{s+1/s}{2}$. Moreover $\sigma$ is doubly ramified at $s=\pm1$.
\end{lem}

Notice that $z=0$ is a degenerate parabolic fixed point for $g_s$ if and only if $s=\pm i$. For $s\in\mathbb{C}^*\setminus\{i,-i\}$, denote by $B^*_s$ the immediate parabolic basin of $g_s$ at 0. Let us consider the union of adjacent parabolic components:
\begin{equation}\label{eq.adjacentlocus}
    \{s\in\mathbb{C}^*\setminus\{i,-i\};\,s,1/s\in B^*_s\}
\end{equation}

We summarize some of the results in \cite{Z} as follows:
\begin{thm}\label{thm.runze}
   The set (\ref{eq.adjacentlocus}) is open and has exactly two components $\mathcal{H}\subset\mathbb{H},-\mathcal{H}\subset\mathbb{-H}$ ($\mathbb{H} := \{s\in\mathbb{C};\mathfrak{Re}\,s>0\}$ is the right half-plane), both of which are symmetric with respect to $s\mapsto 1/s$. Moreover $\mathcal{H}$ has the following properties:
   
   \begin{enumerate}[label=\upshape(\Roman*)]
       \item\label{thm.rurnze:first} ({\cite[Theorem A]{Z}}.) $\partial\mathcal{H}$ is a Jordan curve and Misiurezwicz parabolic parameters are dense in $\partial\mathcal{H}$.
       \item\label{thm.rurnze:second}  The projection of $\overline{\mathcal{H}}$ to family (\ref{eq.family_a1}) via Lemma \ref{lem.relation-gs-fa} is exactly $\mathcal{K}_1\cap\overline{\mathbb{H}}$ (recall $\mathcal{K}_1$ in Definition \ref{defn.central-part}).
       \item\label{thm.rurnze:third} ({\cite[Proposition 3.33]{Z}}.) There exists a curve $\mathcal{I}\subset\mathcal{H}$ symmetric under $s\mapsto1/s$ connecting $s=i$ and $-i$, cutting $\mathcal{H}$ into two connected components $\mathcal{D}_+,\mathcal{D}_-$ (see Figure \ref{fig.region_D+}) such that there is a conformal isomorphism $$\Phi:\mathcal{D}_+\longrightarrow B(\mathfrak{p}_1)\setminus\overline{\Omega},\quad\text{ satisfying }\quad\phi(\Phi(s)) = \phi_s(s),$$ 
       where
       \begin{itemize}
           \item $B(\mathfrak{p}_1)$ is the parabolic basin of $\mathfrak{p}_1(z) = z+z^2$;
           \item $\phi$ is the attracting Fatou coordinate of $\mathfrak{p}_1$ normalized by $\phi(-\frac{1}{2})=0$;
           \item $\Omega$ is the attracting petal of $\mathfrak{p}_1$ such that $\phi(\Omega)=\mathbb{H}$;
           \item $\phi_s$ is the attracting Fatou coordinate of $g_s$ nomalized by $\phi_s(1/s)=0$.
       \end{itemize}
   \end{enumerate}
\end{thm}

The following lemma follows immediately from Theorem \ref{thm.runze} \ref{thm.rurnze:third}:
\begin{lem}\label{lem.unbounded.curveR}
    Let $\mathcal{R}\subset\mathcal{H}$ be a curve that is not compactly contained in $\mathcal{H}$. Let $s_0\in\partial\mathcal{H}$ be an accumulation point of $\mathcal{R}$ and $\mathcal{V}$ be any neighborhood of $s_0$. For $s\in\C\setminus\{i,-i\}$, let $\phi_s: B^*_s\longrightarrow\C$ be any attracting Fatou coordinate of $g_s$. Then $\phi_s(s)-\phi_s(1/s)$ is unbounded for $s\in\mathcal{R}\cap\mathcal{V}$.
\end{lem}
\iffalse
\begin{proof}
    First let us notice that $\phi_s(s)-\phi_s(1/s)$ does not depend on the normalization of $\phi_s$. We may take $\mathcal{V}$ small enough, such that the "$g_s$ version" of Diagram \ref{diag.commu.f_a} holds for $s\in\mathcal{H}\cap\mathcal{V}$. Let $h_s$ be the conjugating map between $g_s$ and $\mathfrak{p}_1(z) = z+z^2$.
    Let $\phi$ be any Fatou coordinate for $\mathfrak{p}_1$. Then for $s\in\mathcal{V}\cap\mathcal{H}$, we have
    \begin{equation}\label{eq.phis-to-phi}
        \phi_s(s)-\phi_s(1/s) = \phi_s(g_s(s))-\phi_s(g_s(1/s)) =\phi(h_s(g_s(s)))-\phi(h_s(g_s(1/s))).
    \end{equation}
     Recall the parameter internal equipotential $\mathcal{E}_n$ and filled-in internal equipotential $\hat{\mathcal{E}}_n$ constructed in \S \ref{subsec.parameter-internal}. We consider these parameter objects for the family $g_s$ and denote them by $\mathcal{E}^{cm}_n,\hat{\mathcal{E}}^{cm}_n$ respectively. From their landing properties (Proposition \ref{Propositionlanding.equi.bitran}), $\mathcal{R}$ is either contained in some $\hat{\mathcal{E}}^{cm}_N$; or it intersects $\mathcal{E}^{cm}_n$ for all $n$. If it is the first case, $h_s(g_s(s)),h_s(g_s(1/s))$ are contained in $\hat{E}_n(\mathfrak{p})$ (the filled-in internal equipotential for $\mathfrak{p}$). Clearly the right hand side of (\ref{eq.phis-to-phi}) is not bounded for $s\in\mathcal{R}$. If it is the second case, then we can take $s_n\in\mathcal{R}\cap\mathcal{E}^{cm}_n$. Thus $\phi(h_s(g_s(s)))-\phi(h_s(g_s(1/s)))\in E_n(\mathfrak{p})$ and tends to infinity.
\end{proof}
\fi

Finally, we have the following topological description on the Julia set of $g_s$:
\begin{thm}[{\cite[Theorem 1,2]{Ro}}]\label{thm.pascale}
    For $s\in\mathbb{C}^*\setminus\{i,-i\}$, the boudary of the immediate parabolic basin $\partial B^*_s$ is always a Jordan curve. If moreover $s\in\overline{\mathcal{H}}$, then the Julia set $J_s$ can be decomposed into the union of $\overline{B^*_s}$ and countable many pairwise-disjoint limbs attached to $\partial B^*_s$: 
    $$J_s=\overline{B^*_s}\cup\left(\bigsqcup_{n} \mathrm{Limb}_n\right).$$
\end{thm}

\begin{figure}[ht]%H为当前位置，!htb为忽略美学标准，htbp为浮动图形
\centering %图片居中
\includegraphics[width=0.5\textwidth]{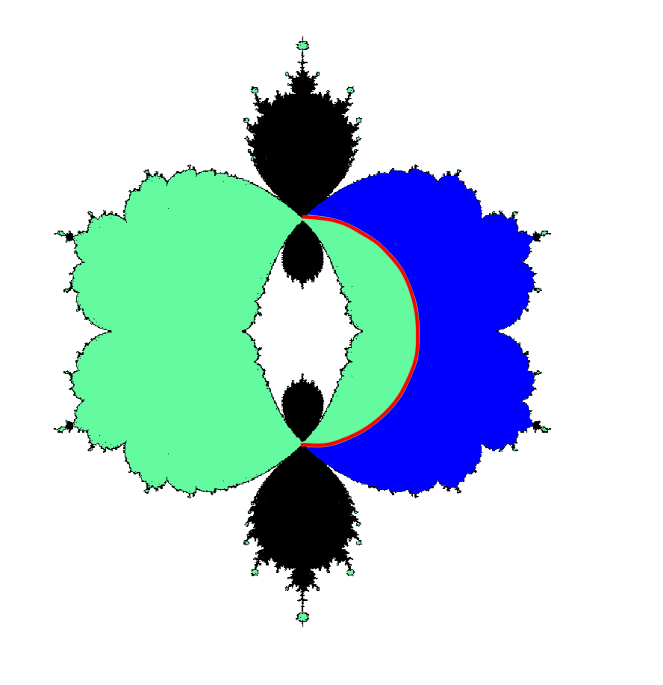} %插入图片，[]中设置图片大小，{}中是图片文件名
\caption{The curve $\mathcal{I}$ is colored in red; the region $\mathcal{D}_+$ is colored in blue.} %最终文档中希望显示的图片标题
\label{fig.region_D+} %用于文内引用的标签
\end{figure}

\section{Cubic Lavaurs map and escaping region}\label{sec.dynamical}

\subsection{Basic setting}\label{subsec.basic-setting}
Let $s\in\mathbb{C}^*\setminus\{i,-i\}$. Recall that $B^*_s$ is the immediate parabolic basin at $0$ of $g_s$. Since $B^*_s$ contains at least one critical point of $g_s$, we make the following assumption:
\begin{assum}[Assumption on $s$]\label{assum.s1}
   $s\in\mathbb{C}^*\setminus\{i,-i\}$, $1/s\in B^*_s$.
\end{assum}

 Let $\phi_s:B^*_s\longrightarrow\C$ be the attracting Fatou coordinate normalized by $\phi_s(1/s) = 0$. Let $\psi_s^{\bullet}:\C\longrightarrow\C$ be the extended map of the inverse repelling Fatou coordinate, normalized by 
$$\tilde{H}_s(w) := \phi_s\circ\psi_s^{\bullet}(w) = w+o(1),\text{ as }\mathfrak{Im}\,w\to+\infty.$$
We denote $\psi_{s,\tau} := \psi_s^{\bullet}\circ T_\tau$.

\begin{defn}[$T_1$-invariant curves $\mathfrak{l}^u_{s,\tau},\mathfrak{l}^l_{s,\tau}$]\label{defn.invariant-curves}
By Theorem \ref{thm.pascale}, $(\psi_{s,\tau})^{-1}(\partial B^*_s)$ consists of two $T_1$-invariant curves. Denote by $\mathfrak{l}^u_{s,\tau},\mathfrak{l}^l_{s,\tau}$ these two curves such that $\mathfrak{l}^u_{s,\tau}$ (resp. $\mathfrak{l}^l_{s,\tau}$) is the boundary of the component of $\mathbb{C}\setminus(\mathfrak{l}^u_{s,\tau}\cup \mathfrak{l}^l_{s,\tau})$ that contains an upper-(resp. lower-)half plane. Let $\tilde{W}^u_{s,\tau}$ be the upper component of $\mathbb{C}\setminus(\mathfrak{l}^u_{s,\tau}\cup \mathfrak{l}^l_{s,\tau})
$. 
\end{defn}

\begin{defn}[Horn maps]\label{def.hornmap}
    For $\tau\in\C$, set $T_\tau:z\mapsto z+\tau$ to be the translation. The {\it lifted horn map} of phase $\tau$ is defined by $$\tilde{H}_{s,\tau} :\tilde{W}^u_{s,\tau}\longrightarrow\C,\quad w\mapsto \tilde{H}_s\circ T_\tau(w).$$
    The corresponding {\it projective horn map} $H_{s,\tau}:W^u_{s,\tau}\longrightarrow\C$ is the projection of $\tilde{H}_{s,\tau}$ via $\mathrm{ixp}(w) := e^{2\pi iw}$, where $W^u_{s,\tau} = \mathrm{ixp}(W^u_{s,\tau})$. It is immediate that $H_{s,\tau}(0) = 0$, $H_{s,\tau}'(0) = e^{2\pi i\tau}$.
\end{defn}

\begin{defn}[Lavaurs map]\label{def.lavaursmap}
    The {\it Lavaurs map} $L_{s,\tau}:B^*_s\longrightarrow\C$ for $g_s$ of phase $\tau$ is defined by 
$$L_{s,\tau} :=\psi_s^\bullet\circ T_\tau \circ\phi_s = \psi_{s,\tau}\circ\phi_s.$$
\end{defn}

Due to Lavaurs and Shishikura, the limiting dynamics of $g_s$ is encoded by Lavaurs map:
\begin{thm}[{\cite[Proposition 3.2.2]{Sh2}}]\label{thm.shishikura-lavaurs}
    Let $\lambda\in\C^*$ and denote $g_{\lambda,s}= \lambda\cdot g_s$.
    Let $\lambda_n = e^{2\pi i \alpha_n}$ with $\alpha_n\in\C$ such that $\lim\limits_{n\to\infty}\lambda_n=1$. Assume that there exists a sequence of positive integers $N_n$ converging to $+\infty$ and $\tau\in\C$ such that 
    \begin{equation}\label{eq.limphase}
        \lim_{n\to\infty}(N_n-\frac{1}{\alpha_n}) = \tau.
    \end{equation}
    
    Then for all $s\in\C^*\setminus\{i,-i\}$, $g_{\lambda_n,s}^{N_n}$ converges uniformly to $L_{s,\tau}$ on compact subsets of $B^*_s$.
\end{thm}

\begin{remnota*}
    In the rest of the article, we will omit the subscript “$\tau$” in every notation if it will be irrelevant to the statement or the proof. For example, we will write $L_s= L_{s,\tau}$, $\phi_s = \phi_{s,\tau}$, $\psi_s = \psi_{s,\tau}$ and $H_s = H_{s,\tau}$, $\tilde{H}_s = \tilde{H}_{s,\tau}$, etc..
\end{remnota*}

\subsection{Singular values}
Let us first recall the definition of singular set:
\begin{defn}[Singular set]
    Let $W\subset\hat{\C}$ be an open subset, $f:W\longrightarrow\hat{\C
    }$ be a non locally constant holomorphic map. We call $v\in f(W)$ an asymptotic value of $f$, if there exists a path $\gamma\subset W$ parametrized by $\gamma(t)$, $t\geq 0$ such that $$\forall\text{ compact }K\subset W,\,\gamma\cap(W\setminus K)\not = \emptyset,\quad\lim_{t\to+\infty}f(\gamma(t)) = v.$$  
    The singular set of $f$ is defined by 
    $$S(f) := \overline{\{v\in f({W});\,v\text{ is a critical value or asymptotic value of }f\}}.$$
    If $S(f)$ is finite, then $f$ is called finite type.
\end{defn}
The following resulst is well-known:
\begin{prop}[{\cite[Proposition 1]{Er}}]\label{Propositioncoveringmap}
    $f:W\setminus f^{-1}(S(f))\longrightarrow f(W)\setminus S(f)$ is a non ramified covering map.
\end{prop}

It is immediate from Definition \ref{def.hornmap} that both $\tilde{H}_{s}$ and $H_{s}$ have at most two critical values:
\begin{lem}\label{lem.two-crit-value}
    The two possible critical values of $\tilde{H}_{s}$ are $\phi_s(s)$ and $\phi_s(1/s)$; the two of $H_{s}$ are $\mathrm{ixp}\circ\phi_s(s), \mathrm{ixp}\circ\phi_s(1/s)$.
\end{lem}

According to Lavaurs \cite[Proposition 1.5.2]{La}, it is easy to verify from Definition \ref{def.lavaursmap} that the critical points of $L_s = L_{s,\tau}$ are classified into following two types:
\begin{propdef}\label{def.twotypes.criticalpoint}
    A critical point $c$ of $L_{s,\tau}$ satisfies either 
\begin{itemize}
    \item $g_s^m(c) = s$ or $1/s$ for some $m\geq 0$; in this case we say that $c$ is {\it absolute}.
    \item $L_{s,\tau-n}(c)  = s$ or $1/s$ for some $n\geq 1$; in this case we say that $c$ is {\it moving}.
\end{itemize}
\end{propdef}

\begin{lem}\label{lem.asymptotic.psi}
    Let $s\in\mathbb{C}^*\setminus\{i,-i\}$. The extended inverse Fatou coordinate $\psi_s:\mathbb{C}\longrightarrow\mathbb{C}$ has a unique asymptotic value $z=0$; the attracting Fatou coordinate $\phi_s:B_s^*\longrightarrow\C$ has no asymptotic value. The Lavaurs map $L_{s}:B_s^*\longrightarrow\C$ has a unique asymptotic value $z=0$; the projective horn map $H_s:W^u_s\longrightarrow\mathbb{C}$ has two unique asymptotic values $z=0,\infty$. Therefore $H_s$ is a finite type map by Lemma \ref{lem.two-crit-value}.
\end{lem}
\begin{proof}
    We prove that $0$ is the unique asymptotic value for $\psi_s$; the result for $\phi_s$ will be similar. It is clear that $0$ is an asymptotic value, since if one takes a path $\gamma$ contained in the left-half plane and $\gamma$ converges to infinity, then $\psi_s(\gamma)$ converges to $z=0$.

    It remains to prove the uniqueness. Notice that the critical value set of $\psi_s$ equals the post-critical set $P_{g_s}$ of $g_s$. Thus it suffices to verify that $\psi_s$ satisfies the following lifting property: let $\gamma\subset\C\setminus(\{0\}\cup \overline{P_{g_s}})$ be any curve parametrized by $t\in[0,1]$; let $w\in\C$ be such that $\psi_s(w) = \gamma(0)$, then there exists a curve $\tilde{\gamma}= \tilde{\gamma}(t)$ such that $\tilde{\gamma}(0) = a$, $\psi_s(\tilde{\gamma}(t)) = \gamma(t)$.

    Take a pair of an ample attracting petal $\hat{P}_+$ and an ample repelling petal $\hat{P}_-$ of $g_s$ (that is, $U = \hat{P}_+\cup\hat{P}_-\cup\{0\}$ is an open neighborhood of $g_s$) such that $\gamma\cap U= \emptyset$. Take $n_0\geq 0$ large enough, such that $w-n_0$ belongs to a left-half plane on which $\psi_s$ is injective and $\psi_s(w-n_0)\in \hat{P}_-$. Since $g_s^{n_0}:\C\setminus g_s^{-n_0}(\overline{P_{g_s}})\longrightarrow\C\setminus\overline{P_{g_s}}$ is a covering map (notice that $P_{g_s} = P_{g_s^{n_0}}$), there exists $\eta = \eta(t)$ such that $g_s^{n_0}(\eta(t)) = \gamma(t)$ with $\eta(0) = \psi_s(w-n_0)$. 
    Since $\gamma\cap P_{g_s} = \emptyset$, the diameter of $g_{s}^{-n}(\gamma)$ tends to 0 as $n$ goes to infinity by the shrinking lemma. Thus by modifying $n_0$ larger if necessary, we may suppose that $\eta\subset U$. Notice that $\eta\cap\hat{P}_+= \emptyset$, otherwise $\gamma = g_s^{n_0}(\eta)$ intersects with $\hat{P}_+$. Thus $\tilde{\gamma}(t) := (\psi|_{\hat{P}_-})^{-1}(\eta(t))+n_0$ will satisfy the desired lifting property.

    Finally we verify that $0$ is the unique asymptotic value for $L_s$ (it will be similar for $H_s$). For two composable analytic maps $f_1,f_2$, we have the trivial inclusion $S(f_1\circ f_2) \subset S(f_1)\cup f_1(S(f_2))$. Thus it suffices to prove that $0$ is an asymptotic value for $L_s$. This is clear since we can take a curve $\xi\subset B^*_s$ converging to $0$, such that the imaginary part of $\phi_s(\xi)$ tends to infinity. Then $\psi_s(T_\tau(\xi))$ converges to $0$.
\end{proof}

\begin{lem}\label{lem.at-least-one}
     If $\mathfrak{Im}\,\tau\geq 0$, then at least one of $\phi_{s,\tau}(s),\phi_{s,\tau}(1/s)$ belongs to $\overline{\tilde{W}^u_{s,\tau}}$.
\end{lem}
\begin{proof}
    Since $H_{s,\tau}'(0) = e^{2\pi i\tau}$ (recall Definition \ref{def.hornmap}), $|H_{s,\tau}'(0)| \leq 1$ if $\mathfrak{Im}\,\tau\geq 0$. Suppose the contrary that $\phi_{s,\tau}(s),\phi_{s,\tau}(1/s)\not\in\overline{\tilde{W}^u_{s,\tau}}$. One can perturb $\tau$ to $\tau'$ such that $|H_{s,\tau'}'(0)| <1$ while $\phi_{s,\tau'}(s),\phi_{s,\tau'}(1/s)\not\in\overline{\tilde{W}^u_{s,\tau'}}$. Therefore the two critical values $\mathrm{ixp}\circ\phi_{s,\tau'}(s),\mathrm{ixp}\circ\phi_{s,\tau'}(1/s)$ of $H_{s,\tau'}$ does not belong to $\overline{W^u_{s,\tau'}}$. On the other hand, $0$ is an attracting fixed point of $H_{s,\tau'}$ and $H_{s,\tau'}$ is a finite type map by Lemma \ref{lem.asymptotic.psi}. However, an attracting fixed point of a non mobius transformation finite type map attracts at least one critical value (cf. \cite[Lemma 4.5.2]{Sh2}). This leads to a contradiction.
\end{proof}

From now on, we always make without exception the following assumption on $\tau$:
\begin{assum}[Assumption on $\tau$]\label{assum.tau}
    $\mathfrak{Im}\,\tau\geq 0$.
\end{assum}

The following lemma justifies why we make Assumption \ref{assum.tau}:   
\begin{lem}\label{lem.assumtau}
    Let $g_{\lambda_n,s}$ be as in Theorem $\ref{thm.shishikura-lavaurs}$ such that $\lim\limits_{n\to\infty}g_{\lambda_n,s}^{N_n} = L_{\tau,s}$. If $|\lambda_n|\leq 1$, then $\mathfrak{Im}\,\tau\geq 0$.
\end{lem}
\begin{proof}
   This follows by a direct computation: $|\lambda_n|\leq 1$ implies that $\mathfrak{Im}\,\alpha_n\geq0$, so $\mathfrak{Im}\,\frac{1}{\alpha_n}\leq0$.
\end{proof}

\subsection{Escaping region $E_s$}
For $s\in\C^*$, we denote by $K_s$ the filled-in Julia set of $g_s$.
\begin{defn}\label{def.escape-region}
   Under Assumption \ref{assum.s1}, the {\it escaping region} and the {\it $N$-th escaping region} for $L_{s,\tau}$ are defined by
$$E_s := \{z\in B^*_s;\,\exists N\geq 1,L_s^N(z)\in \mathbb{C}\setminus{K_s}\}.$$
$$E_s^N := \{z\in B^*_s;\,L_s^N(z)\in \mathbb{C}\setminus {K_s}\}.$$
\end{defn}

\begin{rem}\label{rem.Estau}
    It is immediate from the definition that $E_s^N$ is open and $L_{s,\tau}^{-1}(E_s^N) = E_s^{N+1}$ for all $N\geq 1$. 
\end{rem}

\begin{lem}\label{lem.gs-proper}
    Let $C$ be a component of $E_s^N$. Then $g_s|_{C}$ is a proper map. In particular, $g_s(C)$ is also a component of $E_s^N$.
\end{lem}
\begin{proof}
    Suppose the contrary that $g_s|_{C}$ is not proper. Then there exists a sequence of points $z_n\in C$ converging to $z_0\in\partial C$ such that $g_s(z_0)\in g_s(C)$. On the other hand, $g_s(C)\subset E_s^N$ since $g_s$ and $L_s$ commute. This contradicts with $z_0\in\partial C$.
\end{proof}

\begin{lem}\label{lem.atmost-one-escape}
   Let $\lambda_n,N_n$ and $\tau$ be the same as in Theorem \ref{thm.shishikura-lavaurs}. If we assume additionally that $|\lambda_n|\leq 1$, then $E_s$ contains at most one of $s$ and $1/s$.
\end{lem}
\begin{proof}
Suppose the contrary that $E_s$ contains both $s,1/s$. Then there exists integers $m_1,m_2$ such that $L_{s,\tau}^{m_1}(s)\in\mathbb{C}\setminus{K_s}$ and $L_{s,\tau}^{m_2}(1/s)\in\mathbb{C}\setminus{K_s}$. By Theorem \ref{thm.shishikura-lavaurs}, both $g_{\lambda_n,s}^{m_1N_n}(s)$ and $g_{\lambda_n,s}^{m_2N_n}(1/s)$ are also contained in $C\setminus K_s$, i.e., both $s,1/s$ are attracted by $\infty$. This contradicts the assumption $|\lambda_n|\leq 1$.
\end{proof}

\iffalse
It is not hard to show the following lemma using the maximum modulus principle:
\begin{lem}\label{lem.escapregion.simplyconnec}
 Let $s\in\overline{\mathcal{H}}\setminus\{i,-i\}$. The $N$-th escaping region $E_s^N$ is an open set. Every component $C$ of $E_s^N$ is simply connected.
\end{lem}
\begin{proof}
    By \cite[Theorem 5.1]{Do}, the filled-in Julia set of a polynomial moves upper-semicontinuously. Thus $E_s^N$ is open. Now let us prove the simple connectivity of $C$. Suppose the contrary, then there exists a closed curve $\gamma\subset C$ enclosing a point $z_0\in \partial C$. Therefore $L_{s,\tau}^N$
\end{proof}
\fi

Fix any $s\in\overline{\mathcal{H}}\setminus\{i,-i\}$. By Lemma \ref{lem.atmost-one-escape}, at least one of $s$ and $1/s$ does not escape to $\C\setminus K_s$ under $L_s$. To fix the idea, let us make the following assumption (compare with Assumption \ref{assum.s1}):
\begin{assum}[Further assumption on $s$]\label{assum.s2}
  $s\in\overline{\mathcal{H}}\setminus\{i,-i\}$, $1/s\in B^*_s\setminus E_s$.
\end{assum}

\iffalse
\begin{rem}
    If $s\in\mathcal{H}$, then $\partial B^*_s$ coincides with $J_s$; hence $\mathfrak{l}^l_s,\mathfrak{l}^u_s$ are $T_1$-invariant curves and coincide with $c^l_s,c^u_s$ in Definition \ref{defn.invariant-curves}.
\end{rem}
\fi

Let us recall that $\mathfrak{l}^u_s := \mathfrak{l}^u_{s,\tau}$  and $\mathfrak{l}^l_s :=\mathfrak{l}^l_{s,\tau}$ in Definition \ref{defn.invariant-curves} (we omit the subscript $\tau$) bound a $T_1$-invariant open strip that coincides with $\psi_{s}^{-1}(\mathbb{C}\setminus \overline{B^*_s})$. One can verify easily the following:
\begin{lem}\label{lem.psi}
   Under Assumption $\ref{assum.s2}$, the following diagram commutes
      \begin{equation}\label{diag.psi}
    \begin{tikzcd}
    &&\mathbb{H}\arrow[d, "e^z"]\\
  \psi^{-1}_{s}(\mathbb{C}\setminus K_{s})\arrow[r, "\psi_{s}"]\arrow[rru, bend left = 10, "\Psi_{s}" swap]  & \mathbb{C}\setminus K_{s} \arrow[r] \arrow[r,"\phi^\infty_{s}"] & \C\setminus\overline{\mathbb{D}} 
   \end{tikzcd}\quad,
\end{equation}
where
\begin{itemize}
    \item $\psi_s:\psi^{-1}_s(\mathbb{C}\setminus K_s)\longrightarrow\mathbb{C}\setminus K_s$ is the inverse of the extended repelling Fatou coordinate and is a universal covering map;
    \item $\phi^\infty_s$ is the Böttcher coordinate of $g_s$ at $\infty$, normalized by $\phi_s^\infty(z) = z+o(1)$;
    \item $\Psi_s$ is the lift such that $\Psi_s(\mathfrak{R}_s(0)) =\mathbb{R}^+$, where $\mathfrak{R}_s(0)$ is the unique connected component of $\psi_s^{-1}(R^\infty_s(0))$ that is $T_1$-invariant ($R^\infty_s(0)$ is the external ray of angle 0 for $g_s$).
\end{itemize}  
\end{lem}

\begin{defn}[Croissant and bi-croissant]\label{defn.croissant}
    A simply connected domain $D\subset\C$ is a croissant (resp. bi-croissant), if $\overline{D}$ is homeomorphic to $\overline{\D}\setminus B$ (resp. $\overline{\D}\setminus(B_1\cup B_2)$), where $B\subset\D$ (resp. $B_1,B_2\subset\D$, $\overline{B_1}\cap \overline{B_2}=\emptyset$) is any open round disk such that $\partial{B}\cap\partial\D$ (resp. $\partial B_i\cap\partial\D$, $i=1,2$) is a single point.
\end{defn}

\begin{lem}[See Figure \ref{fig.escaperegion1}]\label{lem.classification.C1}
   Under Assumption $\ref{assum.s2}$, there is a unique $g_s$-invariant connected component $C^1_{s,(0)}$ of $E^1_s$. Moreover for $s\in\mathcal{H}$, we have the following classification of $\partial{C^1_{s,(0)}}$:
    \begin{enumerate}[label=\upshape(\Roman*)]
        \item\label{lem.classification.C1:first} a croissant: $\partial C^1_{s,(0)}$ is the union of two Jordan curves $\overline{\mathfrak{L}^u_s},\overline{\mathfrak{L}^l_s}$ such that $\overline{\mathfrak{L}^u_s}$ encloses $\overline{\mathfrak{L}^l_s}$ and $\overline{\mathfrak{L}^u_s}\cap\overline{\mathfrak{L}^l_s} = \{0\}$. The restrictions 
        $$g_s:C^1_{s,(0)}\longrightarrow C^1_{s,(0)},\quad \phi_s:C^1_{s,(0)}\longrightarrow \psi_s^{-1}(\mathbb{C}\setminus K_s)$$
        are conformal, therefore $L_s:C^1_{s,(0)}\longrightarrow\mathbb{C}\setminus K_s$ is the universal covering. % Moreover if $s\in \partial{C^1_{s,(0)}}$, then $s\in {\mathfrak{L}^u_s}$.
        \item\label{lem.classification.C1:second} a Jordan disk: $\partial C^1_{s,(0)}$ is the union of two Jordan arcs $\overline{\mathfrak{L}^u_s},\overline{\mathfrak{L}^l_s}$, both of which link $0$ to the unique repelling fixed point $\alpha_s$ of $g_s$. One has the same properties as case \ref{lem.classification.C1:first} for $g_s,\phi_s,L_s$. % Moreover if $s\in\partial{C^1_{s,(0)}}$, then $s\in {\mathfrak{L}^l_s}$.
        \item\label{lem.classification.C1:third} $\partial C^1_{s,(0)}$ is the closure of the union of countably many Jordan arcs and Jordan curves. Every Jordan arc is an iterated pre-image of $\overline{\mathfrak{L}^u_s}$; every Jordan curve is an iterated pre-image of $\overline{\mathfrak{L}^l_s}$. Moreover $s\in C^1_{s,(0)}$, the restriction $g_s:C^1_{s,(0)}\longrightarrow C^1_{s,(0)}$ is a degree 2 ramified covering; $L_s:C^1_{s,(0)}\longrightarrow \mathbb{C}\setminus K_s$ is an infinitely ramified covering.
    \end{enumerate}
If $s\in\partial\mathcal{H}\setminus\{i,-i\}$, then $C^1_{s,(0)}$ is always of case \ref{lem.classification.C1:first} by modifying "Jordan curves" into "compact sets".
\end{lem}
\begin{proof}
 Let us fix $v>0$ large enough ($v$ may depend on $s$) such that $\phi_s$ sends an attracting petal $P_s$ of $g_s$ injectively to $\mathbb{H}_v$ and that both $\mathbb{H}_v\cap\mathfrak{l}^u_s$, $\mathbb{H}_v\cap\mathfrak{l}^l_s$ are connected, forward invariant under $T_1$. Notice that $C^1_{s,(0)}$ is unique since any $g_s$-invariant component of $E_s^1$ contains $$(\phi_s|_{P_s})^{-1}(\psi^{-1}_s(\mathbb{C}\setminus K_s)\cap\mathbb{H}_v).$$ 
 
 Denote by $\mathfrak{L}^u_s$ (resp. $\mathfrak{L}^l_s$) the image of $\mathbb{H}_v\cap\mathfrak{l}^u_s$, $\mathbb{H}_v\cap\mathfrak{l}^l_s$ under $(\phi_s|_{P_s})^{-1}$. Since $\mathfrak{L}^u_s,\mathfrak{L}^l_s$ are forward invariant under $g_s$, we can pull them back inductively by $g_s$. During the pullback process, there are three possibilities:
\begin{enumerate}
    \item Neither of the pullbacks of $\mathfrak{L}^u_s,\mathfrak{L}^l_s$ hits a critical point of $g_s$. Then the pull back of $\mathfrak{L}^u_s,\mathfrak{L}^l_s$ are uniquely defined. We abuse the notation and still denote the pulled back set by $\mathfrak{L}^u_s,\mathfrak{L}^l_s$. Since $\mathfrak{L}^u_s,\mathfrak{L}^l_s$ are $g_s$-invariant, they must land at a fixed point of $g_s$, i.e. $z=0$ or the unique repelling fixed point $\alpha_s$ of $g_s$. Thus we have the following subcases: 
    \begin{enumerate}
        \item[(1.1)] Both $\mathfrak{L}^u_s,\mathfrak{L}^l_s$ land at $0$. 
        \item[(1.2)] Both $\mathfrak{L}^u_s,\mathfrak{L}^l_s$ land at $\alpha_s$.
        \item[(1.3)] One of $\mathfrak{L}^u_s,\mathfrak{L}^l_s$ lands at $\alpha_s$ and the other lands at $0$.
    \end{enumerate}
 In subcase (1.1), we claim that $\mathfrak{L}^u_s$ encloses $\mathfrak{L}^l_s$: indeed, otherwise $\mathfrak{L}^u_s$ will enclose a Jordan disk $\mathfrak{D}$ such that $\phi_s|_{\mathfrak{D}}$ is injective and $\overline{W^u_s}\subset \phi_s({\mathfrak{D}})$ (recall $W^u_s$ in Definition \ref{defn.invariant-curves}). In particular, $\overline{W^u_s}$ contains neither $\phi_s(s)$ nor $\phi_s(1/s)$. This contradicts Lemma \ref{lem.at-least-one}. Hence (1.1) leads to \ref{lem.classification.C1:first}. Clearly (1.2) leads to \ref{lem.classification.C1:second} by taking $C^1_{s,(0)}$ to be the Jordan domain bounded by $\overline{\mathfrak{L}^u_s\cup\mathfrak{L}^l_s}$. Finally for (1.3), $\mathfrak{L}^u_s$ cannot land at $0$ by the same argument as (1.1). Hence $\mathfrak{L}^u_s$ lands at $\alpha_s$ and $\mathfrak{L}^u_s$ lands at $0$, which leads to \ref{lem.classification.C1:third}.  

 \vspace{0.2cm}
    \item At least one of the pullbacks of $\mathfrak{L}^u_s,\mathfrak{L}^l_s$ hits exactly one of the critical points $s,1/s$. To fix the idea, let's say that one of them hits $s$ and it will be similar for $1/s$. We have the following two subcases:
\begin{enumerate}
    \item[(2.1)] $\mathfrak{L}^u_s$ hits $s$. Then there are exactly two $g_s$-invariant pullbacks of $\mathfrak{L}^u_s$ and a unique $g_s$-invariant pullback of $\mathfrak{L}^l_s$. Since the two pullbacks of $\mathfrak{L}^u_s$ are invariant by $g_s$, they land at $0$ or $\alpha_s$. Notice that they cannot both land at $\alpha_s$: otherwise there will be two segments of these two pullbacks starting from $s$ to $\alpha_s$, enclose a Jordan domain $V$. However by $g_s$-invariance of the pullbacks, $ g_s(\partial V)\not\subset\overline{V}$. This implies that $g_s(V)$ is unbounded, which is a contradiction since $g_s$ is a polynomial. Denote by $\mathfrak{L}^{u,+}_s$ the pullback landing at $\alpha_s$ and $\mathfrak{L}^{u,-}_s$ the one landing at $0$. Denote by $\mathfrak{L}^{u,+}_s$ the pullback that is above the other one $\mathfrak{L}^{u,-}_s$. Therefore, we have either:
     \begin{enumerate}
        \item[-(2.1a)] $\mathfrak{L}^{u,+}_s$ lands at $0$, $\mathfrak{L}^{u,-}_s$ lands at $\alpha_s$. Then $\mathfrak{L}^l_s$ is beneath $\mathfrak{L}^{u,-}_s$. If the pullback of $\mathfrak{L}^{l}_s$ does not hit $1/s$, then $\mathfrak{L}^{l}_s$ either lands at $0$, which leads to \ref{lem.classification.C1:third}; or it lands at $\alpha_s$, which leads to \ref{lem.classification.C1:second}. If the pullback of $\mathfrak{L}^{l}_s$ hits $1/s$, then similarly there will be exactly two pullbacks $\mathfrak{L}^{l,+}_s, \mathfrak{L}^{l,-}_s$ so that $\mathfrak{L}^{l,+}_s$ is above $\mathfrak{L}^{l,-}_s$. Notice that they cannot both land at $0$, otherwise $\mathfrak{L}^{l,+}_s$ will need to cross $\mathfrak{L}^{u,-}_s$. They cannot both land at $\alpha_s$ by the same argument as (2.1). Hence $\mathfrak{L}^{l,+}_s$ lands at $\alpha_s$ and $\mathfrak{L}^{l,-}_s$ lands at $0$. Abusing the notation by setting $\mathfrak{L}^{l}_s = \mathfrak{L}^{l,+}_s$, $\mathfrak{L}^{u}_s = \mathfrak{L}^{u,-}_s$, we get \ref{lem.classification.C1:second}.

        \item[-(2.1b)] $\mathfrak{L}^{u,+}_s$ lands at $0$, $\mathfrak{L}^{u,-}_s$ lands at $0$. Then $\mathfrak{L}^{u,-}_s$ must enclose $\mathfrak{L}^{l}_s$, which leads to \ref{lem.classification.C1:first}.
        \item[-(2.1c)] $\mathfrak{L}^{u,+}_s$ lands at $\alpha_s$, $\mathfrak{L}^{u,-}_s$ lands at $0$. Then again $\mathfrak{L}^{u,-}_s$ must enclose $\mathfrak{L}^{l}_s$, which leads to \ref{lem.classification.C1:first}.
    \end{enumerate}
\end{enumerate}

\vspace{0.1cm}

 \item [(2.2)] $\mathfrak{L}^l_s$ hits $s$. Similarly as (2.1), let $\mathfrak{L}^{l,+}_s, \mathfrak{L}^{l,-}_s$ be the two $g_s$-invariant pullbacks of $\mathfrak{L}^l_s$. We have the following subcases:
   \begin{enumerate}
        \item[-(2.2a)] $\mathfrak{L}^{l,+}_s$ lands at $0$, $\mathfrak{L}^{l,-}_s$ lands at $\alpha_s$. Then $\mathfrak{L}^{l,+}_s$ must enclose $\mathfrak{L}^{u}_s$. But this is impossible by a similar argument as (1.1).

        \item[-(2.2b)] $\mathfrak{L}^{l,+}_s$ lands at $0$, $\mathfrak{L}^{l,-}_s$ lands at $0$. Then this is also impossible as (2.2a).
        \item[-(2.2c)] $\mathfrak{L}^{l,+}_s$ lands at $\alpha_s$, $\mathfrak{L}^{u,-}_s$ lands at $0$. Then $\mathfrak{L}^{u}_s$ must be above $\mathfrak{L}^{l,+}_s$. If the pullback of $\mathfrak{L}^{u}_s$ does not hit $1/s$, then $\mathfrak{L}^{l}_s$ either lands at $0$, which is impossible by a similar argument as (1.1); or it lands at $\alpha_s$, which leads to \ref{lem.classification.C1:second}. If the pullback of $\mathfrak{L}^{l}_s$ hits $1/s$, then under a similar argument as (2.1a), it will lead to \ref{lem.classification.C1:second}.
\end{enumerate}

 \vspace{0.2cm}
3. One of the pullbacks of $\mathfrak{L}^u_s,\mathfrak{L}^l_s$ hits both of the critical points $s,1/s$. 
\begin{enumerate}
    \item[(3.1)] $\mathfrak{L}^u_s$ hits both $s,1/s$. Then it has three $g_s$-invariant pullbacks $\mathfrak{L}^{u,+}_s,\mathfrak{L}^{u,0}_s,\mathfrak{L}^{u,-}_s$ such that $\mathfrak{L}^{u,+}_s$ is on the top and $\mathfrak{L}^{u,0}_s$ is in the middle. Under a similar argument as (1.1), one can show that $\mathfrak{L}^{u,0}_s$ lands at $\alpha_s$ and the other two land at $0$. Hence $\mathfrak{L}^{l}_s$ must be enclosed by $\mathfrak{L}^{u,-}_s$, which will lead to \ref{lem.classification.C1:first}.
     \item[(3.2)] $\mathfrak{L}^u_s$ hits both $s,1/s$. Similarly as (3.2), we get that $\mathfrak{L}^{u}_s$ must be enclosed by $\mathfrak{L}^{u,+}_s$. This is impossible by a similar argument as (1.1).
\end{enumerate}

    \iffalse
     \item The pullback of $\mathfrak{L}^
     l_s$ hits the critical point $s$. Similarly there are two ways to pull back $\mathfrak{L}^
     l_s$. Denote by $\mathfrak{L}^{l,+}_s$ the pullback landing at $\alpha_s$ and $\mathfrak{L}^{l,-}_s$ the one landing at $0$. Then $\mathfrak{L}^u_s$ has to land at $\alpha$, otherwise it lands at 0 and will intersect $\mathfrak{L}^{l,+}_s$. 
     \fi
\end{enumerate}

\end{proof}

\begin{figure}[ht]%H为当前位置，!htb为忽略美学标准，htbp为浮动图形
\centering %图片居中
\includegraphics[width=\textwidth]{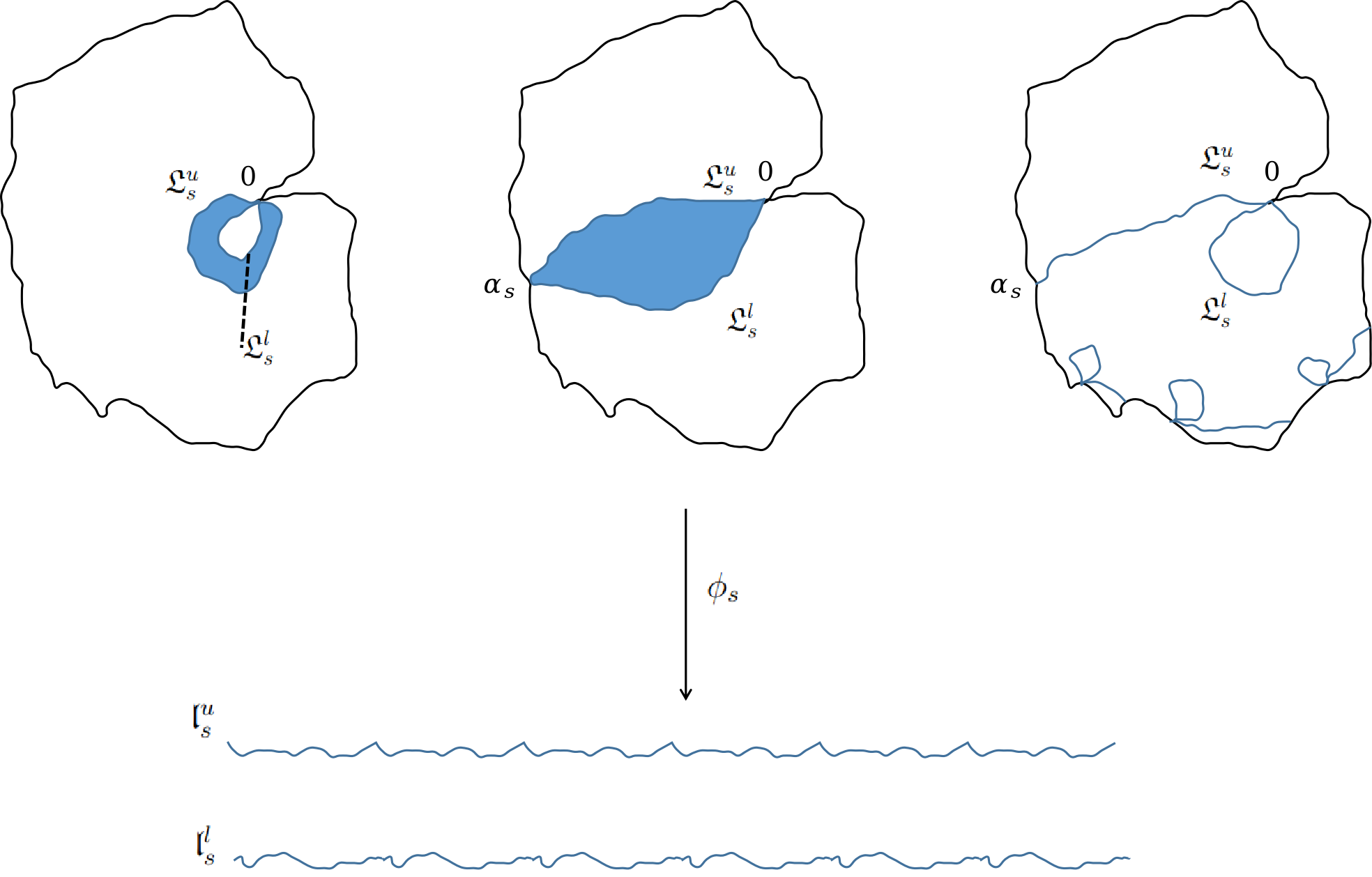} %插入图片，[]中设置图片大小，{}中是图片文件名
\caption{Schematic picture of $C^1_{s,(0)}$ in case \ref{lem.classification.C1:first} \ref{lem.classification.C1:second} and \ref{lem.classification.C1:third} in Lemma \ref{lem.classification.C1} from left to right. The outer dark curve is the Julia set of $g_s$.} %最终文档中希望显示的图片标题
\label{fig.escaperegion1} %用于文内引用的标签
\end{figure}

\begin{lem}\label{lem.properness.lavaurs}
  Under Assumption $\ref{assum.s2}$, let $N\geq 1$ be such that $E^n_s$ contains no critical point of $g_s$ for all $n< N$. Then for each component $C^N\subset E^N_s$, there exists a smallest integer $k\geq 0$ such that $g^k_s\circ L^{N-1}_s$ sends $C^N$ onto $C^1_{s,(0)}$. Moreover, if $C^1_{s,(0)}$ satisfies case \ref{lem.classification.C1:first} or \ref{lem.classification.C1:second} in Lemma \ref{lem.classification.C1}, then $g^k_s\circ L^{N-1}_s:C^N\longrightarrow C^1_{s,(0)}$ is either an isomorphism or a degree two ramified covering.
\end{lem}
\begin{proof}
    If $N =1$, then $C^1$ is necessarily a $g_s$-pre-image of $C^1_{s,(0)}$ by the relation $\phi_s\circ g_s = \phi_s +1$. The rest of the statement follows from Lemma \ref{lem.classification.C1} \ref{lem.classification.C1:first} and \ref{lem.classification.C1:second}.

    Suppose $N\geq 2$. By definition of $E^N_s$ and the case $N=1$ above, there exists a smallest $k\geq 0$ such that $g_s^k\circ L_s^{N-1}(C^N)\subset C^1_{s,(0)}$. By Lemma \ref{lem.asymptotic.psi}, $L_s:B^*_s\setminus L_s^{-1}(0)\longrightarrow\C\setminus\{0\}$ is a branched covering. Thus $L_s^{N-1}|_{C^N}$ is a branched covering. By Lemma \ref{lem.gs-proper}, $g_s^k\circ L_s^{N-1}:C^N \longrightarrow C^1_{s,(0)}$ is a branched covering. In particular, $g_s^k\circ L_s^{N-1}(C^N) = C^1_{s,(0)}$. 
     Consider two cases respectively:

    \begin{itemize}
        \item[\textbf{(a).}] {\it $C^1_{s,(0)}$ satisfies case \ref{lem.classification.C1:first} in Lemma \ref{lem.classification.C1}.} We prove by induction on $1\leq n \leq N$ the following claim: 
        
 \begin{claim*}
     There exists a unique connected component $C^n_{s,(0)}$ of $E^N_s$ that is $g_s$-invariant; moreover, $C^n_{s,(0)}$ is a croissant attached at $z=0$, $\partial C^n_{s,(0)}$ consists of two curves $\mathfrak{L}^{u,n}_s,\mathfrak{L}^{l,n}_s$ such that $\mathfrak{L}^{u,n}_s$ encloses $\mathfrak{L}^{l,n}_s$ and $\mathfrak{L}^{l,n-1}_s$ encloses $\mathfrak{L}^{u,n}_s$. In particular, $C^n_{s,(0)}$ does not contain $s$.
 \end{claim*}

\begin{proof}[Proof of the claim]
    Assume that the claim is proved for $n<N$. Let us set $\mathfrak{l}^{u,n}_s = \phi_s(\mathfrak{L}^{u,n}_s)$ and $\mathfrak{l}^{l,n}_s = \phi_s(\mathfrak{L}^{l,n}_s)$.  Recall that $\mathfrak{l}^{u,1}_s = \mathfrak{l}^{u}_s$ and $\mathfrak{l}^{l,1}_s = \mathfrak{l}^{l}_s$ are the two translation $+1$ invariant curves that are sent to $J_s$ by $\psi_s$. Without loss of generality, we may suppose that $\mathfrak{l}^{u,2}_s$ is beneath $\mathfrak{l}^{l,1}_s$ (the case where $\mathfrak{l}^{l,2}_s$ is above $\mathfrak{l}^{u,1}_s$ will be  similar). By induction, $\mathfrak{L}^{l,n-1}_s$ encloses $\mathfrak{L}^{u,n}_s$. Therefore $\mathfrak{l}^{u,n}_s$ is beneath $\mathfrak{l}^{l,n-1}_s$. Since $C^n_{s,(0)}$ is a $g_s$-invariant croissant by induction, both $\mathfrak{L}^{u,n}_s,\mathfrak{L}^{l,n}_s$ enter a repelling petal of $g_s$. Hence $\psi_s^{-1}(\mathfrak{L}^{u,n}_s)$ (resp. $\psi_s^{-1}(\mathfrak{L}^{l,n}_s)$) has a $T_1$-invariant component $\mathfrak{l}^{u,n+1}_s$, (resp. $\mathfrak{l}^{l,n+1}_s$). Let $\mathfrak{B}$ be the strip bounded by $\mathfrak{l}^{u,n+1}_s, \mathfrak{l}^{l,n+1}_s$. Then $\mathfrak{B}$ is beneath $\mathfrak{l}^{l,n}_s$ since $\mathfrak{L}^{l,n-1}_s$ encloses $\mathfrak{L}^{u,n}_s$ and $\psi_s(\mathfrak{l}^{l,n}_s) = \psi(\mathfrak{L}^{l,n-1}_s)$.

Let $\mathfrak{D}^{l,n}_s$ be the Jordan disk enclosed by $\mathfrak{L}^{l,n}_s$. Let $\mathfrak{H}^{l,n}_s$ be the lower-half plane bounded by $\mathfrak{l}^{l,n}_s$. By Lemma \ref{lem.classification.C1} \ref{lem.classification.C1:first}, $\phi_s:\mathfrak{D}^{l,n}_s\longrightarrow \mathfrak{H}^{l,n}_s$ is bijective. Hence $C^{n+1}_{s,(0)} := (\phi_s|_{\mathfrak{D}^{l,n}_s})^{-1}(\mathfrak{B})$ is a croissant which is enclosed by $\mathfrak{L}^{l,n}_s$. This proves the existence of $C^{n+1}_{s,(0)}$ and its desired properties.

Suppose that there is another $g_s$-invariant $\tilde{C}^{n+1}_{s,(0)}$. Since $L_s,g_s$ commute, $L_s(\tilde{C}^{n+1}_{s,(0)})$ is also $g_s$-invariant. By the uniqueness of $C^n_{s,(0)}$, $L_s(\tilde{C}^{n+1}_{s,(0)}) = C^n_{s,(0)}$. Since $\psi_s$ is injective on some left-half plane with large negative real part, there exists a unique component of $\psi_s^{-1}(C^n_{s,(0)})$ that is $T_1$-invariant (which has to be $\mathfrak{B}$). By the relation $\phi_s\circ g_s = \phi_s +1$, there is a unique component of $\phi_s^{-1}(\mathfrak{B})$ which is $g_s$-invariant.
\end{proof}
        \item[ \textbf{(b).}] {\it $C^1_{s,(0)}$ satisfies \ref{lem.classification.C1:second} in Lemma \ref{lem.classification.C1}.} We prove that any component $\mathfrak{B}'$ of $\psi_s^{-1}(C^1_{s,(0)})$ is {\it not} $T_1$-invariant, which implies that $E^n_s$ contains no $g_s$-invariant component for $n\geq 2$. Suppose the contrary, then $\mathfrak{B}'$ will intersect some left-half plane on which $\psi_s$ is injective. Therefore $C^1_{s,(0)}$ intersects some repelling petal of $g_s$, in which the dynamics is attracting by $z=0$ under the action of $g_s^{-1}$. But points in $C^1_{s,(0)}$ are attracted by $\alpha_s$ under the action of $g_s^{-1}$, a contradiction.

    \end{itemize}

Thus in both Cases \textbf{(a)} and \textbf{(b)}, the orbit $C^N\to g_s(C^N)\to...\to g_s^{k-1}(C^N)$ contains the critical point $s$ at most one time (otherwise some $g_s^l(C^N)$ containing $s$ will be $g_s$-invariant). We make the following two claims: 
\begin{itemize}
    \item[(\romannumeral1)]If some $g_s^l(C^N)$ contains $s$, then $g_s^k(C^N)$ contains no critical point of $L_s$. To see this, notice that it cannot contain an absolute critical point, otherwise $g_s^l(C^N)$ $g_s$-invariant; it cannot contain a moving critical point, otherwise $E^{N-1}_s$ will contain a critical point of $g_s$. Hence $L_s^N$ sends $g_s^k(C^N)$ bijectively to $C^{1}_{s,(0)}$. 
    \item[(\romannumeral2)]If the orbit $C^N\to g_s(C^N)\to...\to g_s^{k-1}(C^N)$ does not contain $s$, then $g_s^k(C^N)$ contains at most one critical point of $L_s$. To see this:
\begin{itemize}
    \item Clearly $g_s^k(C^N)$ contains no moving critical point, otherwise $E^{N-1}_s$ will contain $s$.
    \item If $g_s^k(C^N)$ contains two absolute critical points $c_1,c_2$ such that $g_s^{k_i}(c_i) = s$, $i=1,2$, then $k_1=k_2$, otherwise there will exist a $g_s$-invariant component of $E^N_s$ containing $s$. This implies that there exists a smallest $k'\leq k_1$ such that $g_s^{k'}(c_1) = g_s^{k'}(c_2)$. Hence $g_s:g_s^{k+k'-1}(C^N)\longrightarrow g_s^{k+k'}(C^N)$ has degree at least two and $g_s^{k+k'-1}(C^N)$ contains $s$. Thus $k'=k_1=k_2$, otherwise $g_s^{k+k'-1}(C^N)$ will be $g_s$-invariant. Thus $g_s^k(C^N)$ contains $s$, and it is $g_s$-invariant. This is a contradiction. 
\end{itemize}
Immediately from the above two claims, we get that $L_s^N:g_s^k(C^N)\longrightarrow C^{1}_{s,(0)}$ is at most a degree two covering, which completes the proof. 
\end{itemize}
\end{proof}

\subsection{Parametrization of the escaping region $E_s$}
Suppose that $s$ satisfies Assumption \ref{assum.s1} and  assume furthermore that $1/s\not\in E_s$. Following \cite{Ka}, for an integer $N\geq 0$ and a real number $\omega\geq 0$, we say that $L_s$ is $(N,\omega)$-{\it nonescaping} if either
\begin{itemize}
    \item $G_s(L_s^N(s)) = \omega$
    \item $L_s^{N'}(s) \in K_s\setminus B^*_s$ for some $1\leq N'< N$ and $\omega=0$.
\end{itemize}
where $K_s$ is the filled-in Julia set of $g_s$ and $G_s$ is the green function of $g_s$. Notice that a $(N,\omega)$-nonescaping Lavaurs map $L_s$ is automatically $(N-1,0)$-nonescaping if $N\geq 1$; $L_s$ is $(0,\omega)$-nonescaping with $\omega> 0$ if and only if $s\in\mathbb{C}\setminus K_s$.

Let $\phi^\infty_s$ be the Böttcher coordinate of $g_s$ at $\infty$, normalized by $\phi_s^\infty(z) = z+o(1)$. It is a classical result that if $s\in K_s$, then $\phi^\infty_s$ can be injectively extended to the whole attracting basin of infinity $B_s(\infty)$ with $\phi_s^\infty(B_s(\infty)) = \mathbb{C}\setminus\overline{\mathbb{D}}$; if $s\not\in K_s$, then $\phi^\infty_s$ can be injectively extended to a neighborhood $U_s$ of $\infty$ such that $s\in\partial U_s$ and $\phi_s^\infty(U_s) = \mathbb{C}\setminus\overline{\mathbb{D}}_\rho$, where $\mathbb{D}_\rho$ is the open disk centered at $0$ with radius $\rho = e^{G_s(s)}> 1$. Let $r_s = G_s(s)$, set $\varphi_s: \mathbb{H}_{r_s}\longrightarrow \mathbb C$ to be $\varphi_s(w) := (\phi^\infty_s)^{-1}(e^w)$. For a $(N,\omega)$-nonescaping $L_s$ and $r\geq \omega$, define the {\it $(N,r)$-escaping region} by $E^{N}_s(r) := L^{-N}_s(\varphi_s(\mathbb{H}_r))$. 

By Lemma \ref{lem.classification.C1}, $\psi_s:\psi_s^{-1}(\mathbb{C}\setminus K_s)\longrightarrow \mathbb{C}\setminus K_s$ is the universal covering map. Therefore there exists a unique component of $\psi_s^{-1}(R^\infty_s(0))$, denoted by $\mathfrak{R}_s(0)$, that is $T_1$-invariant, where $R^\infty_s(0)$ is the external ray of angle 0 for $g_s$. We take a large $v>0$ and a petal $P_s$ such that $\phi_s|_{P_s}:P_s\longrightarrow \mathbb{H}_v$ is injective. Suppose $L_s$ is $(1,\omega)$-nonescaping. Define $C^1_{s,(0)}(\omega)$ to be the connected component of $\phi_s^{-1}(E^1_s(\omega))$ containing $(\phi_s|_{P_s})^{-1}(\mathfrak{R}_s(0)\cap\mathbb{H}_v)$. 

\begin{lem}\label{lem.botthcer-lavaurs.coordinate}
    Let $s\in\mathcal{H}$ satisfy Assumption \ref{assum.s2}. Suppose $L_s$ is $(1,\omega)$-nonescaping. Then there exists a unique holomorphic isomorphism $\Phi_{s,(0)}: C^1_{s,(0)}(\omega)\longrightarrow \mathbb{H}_\omega$ such that $\varphi_s\circ{\Phi_{s,(0)}} = L_s$, $\mu\circ\Phi_{s,(0)} = \Phi_{s,(0)}\circ g_s$, where $\mu$ is the multiplication by 3.
\end{lem}
\begin{proof}
   By Lemma \ref{lem.classification.C1}, we can lift the universal covering map $\phi_s^\infty \circ L_s:C^1_{s,0}(\omega)\longrightarrow \mathbb{C}\setminus\overline{\mathbb{D}_{e^{\omega}}}$ by $z\mapsto e^z$ to a map $\Phi_{s,(0)}: C^1_{s,(0)}(\omega)\longrightarrow \mathbb{H}_\omega$ such that $$\Phi_{s,(0)}((\phi_s|_{P_s})^{-1}(\mathfrak{R}_s(0)\cap\mathbb{H}_v))\subset \mathbb{R}^+.$$ 
   It is easy to check that the required properties for $\Phi_{s,(0)}$ are satisfied. 
\end{proof}

\begin{defn}\label{def.botthcer-lavaurs}
    Let $L_s$ be $(N,\omega)$-nonescaping with $N\geq 1$, and $C^N_s(\omega)$ a connected component of $E^N_s(\omega)$. Define the {\it Böttcher-Lavaurs coordinate} on $C^N_s(\omega)$ by $$\Phi_{s,C^N_s(\omega)}:C^N_s(\omega)\longrightarrow \mathbb{H}_\omega,\quad \Phi_{s,C^N_s(\omega)} := \mu^{-m}\circ\Phi_{s,(0)}\circ g^m_s\circ L^{N-1}_s,$$ 
where $m\geq 0$ is any integer such that $g^m_s\circ L^{N-1}_s(C^N_s(\omega))\subset C^1_{s,(0)}(\omega)$. 
\end{defn}

The following proposition is an immediate consequence of Lemma \ref{lem.properness.lavaurs} and \ref{lem.botthcer-lavaurs.coordinate}.
\begin{prop}\label{Propositionbotthcer-lavaurs.coordinate}
    $\Phi_{s,C^N_s(\omega)}$ is an isomorphism and does not depend on the choice of $m$. Moreover, if $C^1_{s,(0)}$ satisfies case \ref{lem.classification.C1:first} or \ref{lem.classification.C1:second} in Lemma \ref{lem.classification.C1}, and if $C^N_s$ does not intersect the inverse orbit of $s$ (under $g_s$), then $\Phi_{s,C^N_s(\omega)}$ extends to an isomorphism $\Phi_{s,C^N_s}: C^{N}_s\longrightarrow\mathbb{H}$ and we have the following commutative diagram:
    \begin{equation}\label{diag.commutative.criticalvalue}
    \begin{tikzcd}
    {C}^N_{s}\arrow[d, "g^m_{s}\circ L^{N-1}_s" swap]\arrow[rrr, bend left=18, "\mu^m\circ\Phi_{s,C^N_s}"]&&&\mathbb{H}\arrow[d, "e^z"]\\
    C^1_{s,(0)}  \arrow[r, "\phi_{s}"] \arrow[rrru, bend left = 12, "\Phi_{s,(0)}"]& \psi^{-1}_{s}(\mathbb{C}\setminus K_{s})\arrow[r, "\psi_{s}"]\arrow[rru, bend left = 10, "\Psi_{s}" swap]  & \mathbb{C}\setminus K_{s} \arrow[r] \arrow[r,"\phi^\infty_{s}"] & \mathbb{C}\setminus\overline{\mathbb{D}} 
   \end{tikzcd}
\end{equation}    
If $C^1_{s,(0)}$ satisfies case \ref{lem.classification.C1:third} in Lemma \ref{lem.classification.C1}, we have the following commutative diagram:
\begin{equation}\label{diag.commutative.criticalvalue_case(3)}
  \begin{tikzcd}
    {C}^1_{s}(\omega)\arrow[d, "g^m_{s}" swap]\arrow[rrr, bend left=18, "\mu^m\circ\Phi_{s,C^1_s(\omega)}"]&&&\mathbb{H}_{3^m\omega}\arrow[d, "e^z"]\\
   C^1_{s,(0)}(3^m\omega_{s})  \arrow[r, "\phi_{s}"] \arrow[rrru, bend left = 12, "\Phi_{s,(0)}"]& \psi^{-1}_{s}(\varphi_s(\mathbb{H}_{3^m\omega}))\arrow[r, "\psi_{s}"]\arrow[rru, bend left = 10, "\Psi_{s}" swap]  & \varphi_s(\mathbb{H}_{3^m\omega}) \arrow[r] \arrow[r,"\phi^\infty_{s}"] & \mathbb{C}\setminus\overline{\mathbb{D}_{e^{3^m\omega}}}
   \end{tikzcd}
\end{equation}   
\end{prop}

\section{Proof of Theorem \ref{thm.implosion-main1}}\label{sec.main1}
The escaping locus of the family $L_s$ is defined to be 
\[\boldsymbol{\mathrm{Esc}} := \{s\in\mathcal{H};\,\exists N\geq 1 \text{ s.t. }L^N_s(s) \text{ or }L^N_s(1/s) \in \mathbb{C}\setminus K_s\},\]
Clearly $\boldsymbol{\mathrm{Esc}}$ is open. Proposition \ref{Propositionexistence} guarantees that it is not empty.
By Lemma \ref{lem.atmost-one-escape},
$\boldsymbol{\mathrm{Esc}}$ is naturally decomposed into
\begin{equation}
    \begin{aligned}
        &\boldsymbol{\mathrm{Esc}}_+:=\{s\in\mathcal{H};\,\exists N\geq 1 \text{ s.t. }L^N_s(s)\in \mathbb{C}\setminus K_s\},\\
        &\boldsymbol{\mathrm{Esc}}_-:= \{s\in\mathcal{H};\,\exists N\geq 1 \text{ s.t. }L^N_s(1/s) \in \mathbb{C}\setminus K_s\}.
    \end{aligned}
\end{equation}

Furthermore, each of them is decomposed into $\boldsymbol{\mathrm{Esc}}^N_+$ and $\boldsymbol{\mathrm{Esc}}^N_-$, where $N\geq 1$ represents the first escaping moment. 

By the symmetry $s\mapsto1/s$, it suffices to study one of them, say $\boldsymbol{\mathrm{Esc}}_+$. Let $\mathcal{E}^N$ be a connected component of $\boldsymbol{\mathrm{Esc}}^N_+$. Clearly any $s\in\mathcal{E}^N$ is $(N,\omega)$-nonescaping with some $\omega> 0$.
Denote by ${Cv}^N_{s}(\omega)$ the connected component of $E^N_s(\omega)$ containing the critical value $v_s := g_s(s)$. Now we are ready to prove the following proposition which implies directly Theorem \ref{thm.implosion-main1}.

\begin{prop}\label{Propositionparametristation.escape_locus}
    Let $\Phi_{s,{Cv}^N_s(\omega)}:{Cv}^N_{s}(\omega)\longrightarrow \mathbb{H}_{\omega_s}$ be the Böttcher-Lavaurs coordinate. The mapping $\Phi_{\mathcal{E}^N}:\mathcal{E}^N\longrightarrow \mathbb{H}$ defined by $s\mapsto\Phi_{s,{Cv}^N_s(\omega)}(v_s)$ is a conformal isomorphism. 
\end{prop}

\begin{proof}
    We will construct directly the inverse map of $\Phi := \Phi_{\mathcal{E}^N}$. 
    We start by constructing a dynamical holomorphic motion using Branner-Hubbard motion (cf. \cite{PT}) as follows: for any $u\in\mathbb{H}$, consider the mapping $l_u(z) = z\cdot|z|^{u-1}$. Clearly $(l_u(z))^n = l_u(z^n)$ for any $n\geq1$. Fix $s_0\in\mathcal{E}^N$, pull-back the 0-Beltrami coefficient on $\mathbb{C}\setminus\overline{\mathbb{D}}$ by $l_u\circ\phi^\infty_{s_0}$ to get a Beltrami coefficent $\mu_{s_0,u}$, and spread $\mu_{s_0,u}$ into $B^*_{s_0}$ by pulling back infinitely many times by $L_{s_0}$ and $g_{s_0}$. This operation is valid since $L_{s_0}$ and $g_{s_0}$ commute. Finally set elsewhere to be 0. Therefore we get a $L_{s_0}$ and $g_{s_0}$-invariant Beltrami coefficent $\mu_{s_0,u}$. Integrate it with some proper normalization we get a quasiconformal map $h_{s_0,u}$ conjugating $g_{s_0}$ to some $g_{s(v)}$. Moreover, it turns out that $h_{s_0,u}$ also conjugates $L_{s_0}$ to $L_{s(u)}$: recall that $L_{s_0} = \psi_{s_0}\circ\phi_{s_0}$, pull back $\mu_{s_0,u}$ by $\psi_{s_0}$ and integrate it to get $\chi_{s_0,u}$ fixing $0,1,\infty$. It is then direct that $h_{s_0,u}\circ\psi_{s_0}\circ\chi^{-1}_{s_0,u}$ is holomorphic and semi-conjugates $g_{s(u)}$ to translation $+1$. So is $\chi_{s_0,u}\circ\phi_{s_0}\circ h^{-1}_{s_0,u}$. Thus by uniqueness of Fatou coordinate, they are $\psi_{s(u)}$, $\phi_{s(u)}$ respectively. Hence $L_{s_(u)}\circ h_{s_0,u} = h_{s_0,u}\circ L_{s_0}$. The relation between the mappings are summarized in the commutative diagram below:

\begin{equation}\label{diag.commu.branner-hubbard}
    \begin{tikzcd}
       B^*_{s(u)} \arrow[d,"g_{s(u)}"]\arrow[rrr,bend left = 40, "\phi_{s(u)}"]&B^*_{s_0} \arrow[l,"h_{s_0,u}" swap] \arrow[d,"g_{s_0}"]\arrow[r,"\phi_{s_0}"] &\mathbb{C}\arrow[r,"\chi_{s_0,u}"]\arrow[d,"+1"]& \mathbb{C} \arrow[d,"+1"]\arrow[rrr,bend left = 40, "\psi_{s(u)}"]& \mathbb{C}\arrow[l,"\chi_{s_0,u}" swap]\arrow[d,"+1"]\arrow[r,"\psi_{s_0}"]&\mathbb{C}\arrow[d,"g_{s_0}"]\arrow[r,"h_{s_0,u}"]&\mathbb{C}\arrow[d,"g_{s(u)}"] \\
       B^*_{s(u)}\arrow[rrr,bend right = 40, "\phi_{s(u)}" ]&B^*_{s_0} \arrow[l,"h_{s_0,u}" swap]\arrow[r,"\phi_{s_0}"] &\mathbb{C}\arrow[r,"\chi_{s_0,u}"] &\mathbb{C}\arrow[rrr,bend right = 40, "\psi_{s(u)}"]&\mathbb{C}\arrow[l,"\chi_{s_0,u}" swap]\arrow[r,"\psi_{s_0}"]  &\mathbb{C}\arrow[r,"h_{s_0,u}"]&\mathbb{C}
    \end{tikzcd}
\end{equation}
 For simplicity of notations, we denote $\Phi_{s,{Cv}^N_s(\omega)}$ by $\Phi_s$. Recall that 
 \begin{equation}\label{eq.Phi_s}
     \Phi_s = \mu^{-m}\circ\Phi_{s,(0)}\circ g^m_s\circ L^{N-1}_s
 \end{equation} 
 for some $m\geq 0$ and $\mu$ is the multiplication by 3. By Proposition \ref{Propositionbotthcer-lavaurs.coordinate}, we have the commutative diagram
    \begin{equation}\label{diag.commu.lavaur-bottcher}
    \begin{tikzcd}
    {Cv}^N_{s_0}(\omega_{s_0})\arrow[d, "g^m_{s_0}\circ L^{N-1}_{s_0}" swap]\arrow[rr, bend left=10, "\mu^m\circ\Phi_{s_0}"]&&\mathbb{H}_{3^m\omega_{s_0}}\arrow[d, "e^z"]\arrow[ddd, bend left = 50, color =red, "\tilde{l}_u"]\\
 C^1_{s_0,(0)}(3^m\omega_{s_0})  \arrow[r, "L_{s_0}"] \ar[red]{d}{h_{s_0,u}} \arrow[rru, bend left = 12, color =red, "\Phi_{s_0,(0)}" swap]& \varphi_{s_0}(\mathbb{H}_{3^m\omega_{s_0}}) \arrow[d, "h_{s_0,u}"] \arrow[r] \arrow[r,"\phi^\infty_{s_0}"] & \mathbb{C}\setminus\overline{\mathbb{D}_{e^{3^m\omega_{s_0}}}} \arrow[d, "l_u"] \\
 C^1_{s(u),(0)}(3^m\omega_{s(u)}) \arrow[r, "L_{s(u)}"] \arrow[rrd, bend right = 12, color =red, "\Phi_{s(u),(0)}"] & \varphi_{s(u)}(\mathbb{H}_{3^m\omega_{s(u)}}) \arrow[r,"\phi^\infty_{s(u)}"]  & \mathbb{C}\setminus\overline{\mathbb{D}_{e^{3^m\omega_{s(u)}}}}\\
 {Cv}^N_{s(u)}(\omega_{s(u)})\arrow[u, "g^m_{s(u)}\circ L^{N-1}_{s(u)}"]\arrow[rr, bend right=10, "\mu^m\circ\Phi_{s(u)}" swap]&&\mathbb{H}_{3^m\omega_{s(u)}}\arrow[u, "e^z" swap]
\end{tikzcd}
\end{equation}
from which (the red loop) we obtain
\begin{equation}\label{eq.tilde_l}
    \Phi_{s(u),(0)} = \tilde{l}_{u}\circ\Phi_{s_0,(0)}\circ h^{-1}_{s_0,u},\text{ where } \tilde{l}_{u}(w) = w+\mathfrak{Re}\,w\cdot(u-1).
\end{equation}
Thus 
\begin{equation}
\begin{aligned}
    \Phi_{s(u)}(v_{s(u)})& =  \mu^{-m}\circ\Phi_{s(u),(0)}\circ g^m_{s(u)}\circ L^{N-1}_{s(u)}(v_{s(u)})\\
    &=\mu^{-m}\circ\tilde{l}_{u}\circ\Phi_{s_0,(0)}\circ h^{-1}_{s_0,u}\circ g^m_{s(u)}\circ L^{N-1}_{s(u)}(v_{s(u)})\\
    &=\mu^{-m}\circ\tilde{l}_{u}\circ\Phi_{s_0,(0)}\circ g^m_{s_0}\circ L^{N-1}_{s_0}(v_{s_0}),
\end{aligned}
\end{equation}
where the first equality uses (\ref{eq.Phi_s}); the second equality uses (\ref{eq.tilde_l}); the third equality uses (\ref{diag.commu.branner-hubbard}) and the fact that $h_{s_0,u}^{-1}(v_{s(u)}) = v_{s_0}$.
A straight forward computation shows that
\[\tilde{l}_u = \begin{bmatrix}
    &\mathfrak{Re}\,u & 0\,\,\,\,\\
    &\mathfrak{Im}\,u & 1\,\,\,\,
\end{bmatrix}\]
is a linear automorphism of $\mathbb{H}$ such that $\tilde{l}_u(1) = u$. Hence there exists $u_0$ such that $\Phi_{s(u_0)}(v_{s(u_0)}) = 1$. Let us set $s_* := s(u_0)$. Construct similarly the holomorphic motion $h_{s_*,u}$ with base point $s_*$. Conjugating $g_{s_*}$ by $h_{s_*,u}$, one gets $g_{s_*(u)}$. Thus we obtain an analytic map\footnote{
\cite[Thm. 2.7]{PT} guarantees that $\mathscr{A}$ is analytic.} $\mathscr{A}:\mathbb{H}\longrightarrow \mathcal{E}^N$ defined by $u\mapsto s_*(u)$. It follows by construction that $\Phi\circ \mathscr{A} = \mathrm{Id}_{\mathbb{H}}$ and $\mathscr{A}\circ\Phi  = \mathrm{Id}_{\mathcal{E}^N}$.
\end{proof}

\section{Proof of Theorem \ref{thm.implosion-main2}}\label{sec.main2}

For $s\in \boldsymbol{\mathrm{Esc}}^1_+$, let us recall Definition \ref{def.escape-region} that the {\it $1$-escaping region} for $L_s$ is defined to be 
$$E^{1}_s:= \{z\in B^*_s;\,L_s(z)\in \mathbb{C}\setminus K_s\}.$$ 
Let ${Cv}^1_s,{Cp}^1_s,{Cc}^1_s$ be the connected component of $E^1_s$ containing the critical value $v_s = g_s(s)$, the critical point $s$ and 
its co-critical point respectively. 
\begin{defn}\label{def.disjoint-adjacent}
    We say that a connected component $\mathcal{E}^1$ of $\boldsymbol{\mathrm{Esc}}^1_+$ is \textit{adjacent} if $Cv^1_s = Cp^1_s$, otherwise it is called \textit{disjoint}. See Figure \ref{fig.disjoint-adjacent}.
\end{defn}

\begin{figure}[ht]%H为当前位置，!htb为忽略美学标准，htbp为浮动图形
\centering %图片居中
\includegraphics[width=0.7\textwidth]{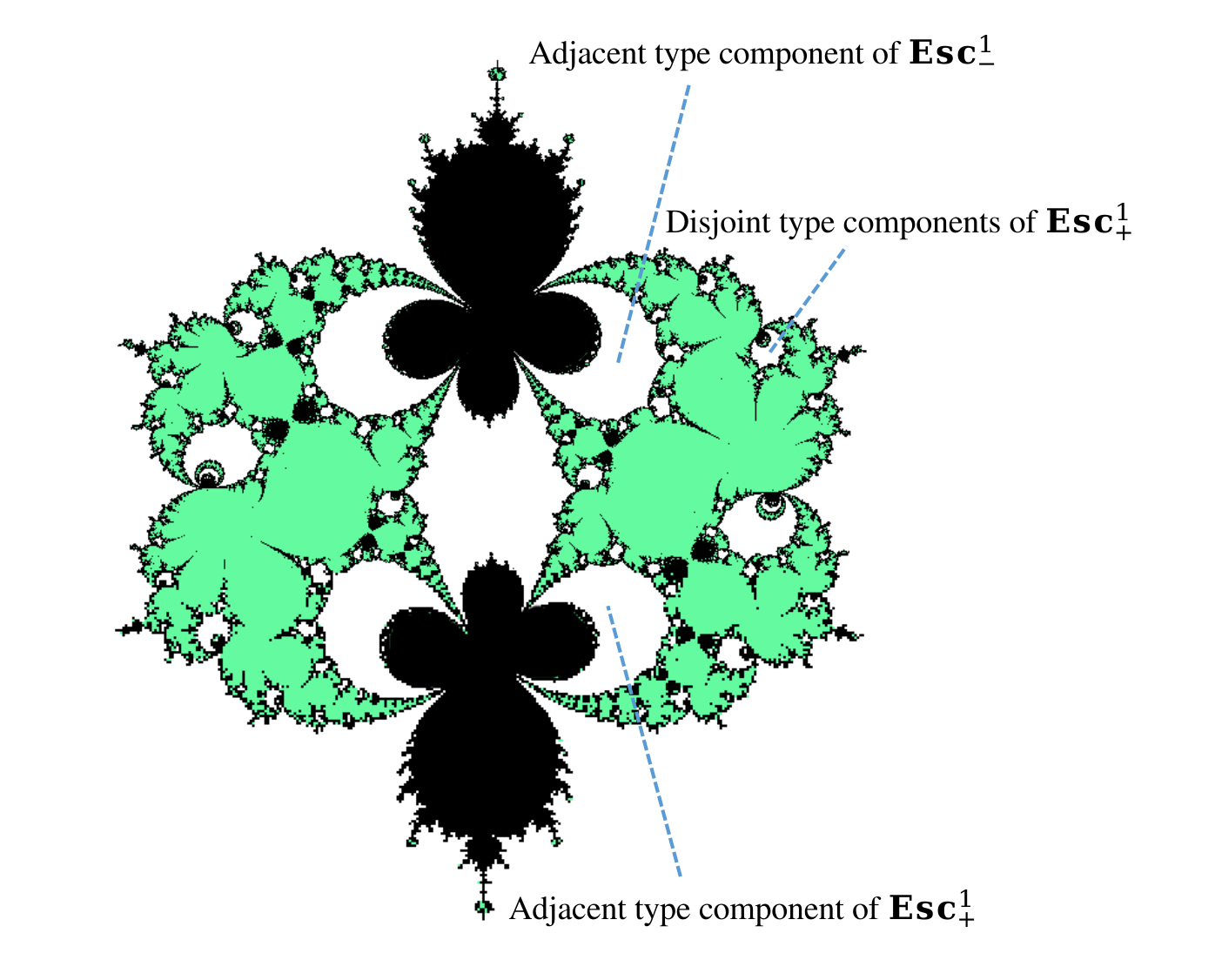} %插入图片，[]中设置图片大小，{}中是图片文件名
\caption{} %最终文档中希望显示的图片标题
\label{fig.disjoint-adjacent} %用于文内引用的标签
\end{figure}

In this section, we aim to prove the following theorem, which will be crucial to prove Theorem \ref{thm.implosion-main2}:
\begin{thm}\label{thm.implosion.family-gs}
      Let $\mathcal{E}^1\subset\boldsymbol{\mathrm{Esc}}^1_+$ be a connected component and $s_n\in\mathcal{E}^1$ be a sequence such that $s_n$ converges to $\partial\mathcal{E}^1$. Then 
      \begin{itemize}
          \item  $s_n$ converges to $\partial\mathcal{H}$ if and only if $\Phi_{\mathcal{E}^1}(s_n)$ tends to $0$ or $\infty$. 
          \item ${\partial\mathcal{E}^1}\cap\partial\mathcal{H} = \{s_0\}$ is a singleton: if $\mathcal{E}^1$ is disjoint type, then $s_0$ is Misiurewicz parabolic, that is, $g_{s_0}^N(s_0) = 0$ for some $N\geq 1$; if $\mathcal{E}^1$ is adjacent type, then $s_0=-i$.
      \end{itemize}
\end{thm}

\subsection{Control of the Julia set for near degenerate parabolic mappings}\label{subsec.julia-control}
 Consider the parabolic degenerate polynomial $g_{i} = g_{-i}$. It has two immediate basins at $0$. Denote by $B^*_i$ the one containing $i$. Let $\phi_{1,+,i}:B^*_i\longrightarrow \C$ be the attracting Fatou coordinate and $\phi_{1,-,i}:P_{1,-,i}:\longrightarrow\C$ be the repelling Fatou coordinate of repelling axis $\mathbb{R}^+$, where $P_{1,-,i}$ is a small repelling petal. In this subsection, we always make the following assumption if there is no extra explanation:
 \begin{assum}\label{assum.g_s-closeto-i}
     $s\in\mathcal{H}$ is close to $i$ or $-i$, such that $f_{a_s}$ satisfies the conditions in Lemma \ref{lem.perturb.Fatoucoordi} or \ref{lem.perturb.Fatoucoordi'}, where $f_{a_s}(z) = z+a_sz^2+z^3$ is affinely conjugate to $g_s$. 
 \end{assum}
In particular, $g_s$ has gate structure $(*,2)$ or $(2,*)$. In the following discussion, we focus on $(*,2)$ and it will be similar for $(2,*)$. 

Recall that by Lemma \ref{lem.perturb.Fatoucoordi}, $g_s$ has the corresponding attracting/repelling Fatou coordinates $\phi_{1,\pm,s}$  and attracting/repelling petals $P_{1,\pm,s}$. We take the Fatou coordinates with the preferred normalization (Definition \ref{def.prefered.nor}).
\paragraph{Upper main chessboard.} Let $\mathbb{U}$ be the upper-half plane bounded by the horizontal line passing through $\phi_{1,+,i}(i)$.
    Let $CB$ be the {\it upper main chessboard} in $B^*_i$: that is, the unique $g_{i}$-invariant connected component of $\phi_{1,+,i}^{-1}(\mathbb{U})$. Notice that $\partial\mathbb{U}$ is a Jordan curve, and $\mathbb{U}$ enters $P_{1,-,i}$. Thus by Lemma \ref{lem.perturb.Fatoucoordi}, there is a $g_{s}$-invariant Jordan disk $CB_{s}\subset B^*_{s}$, such that $s,0\in\partial CB_s$, $\phi_{1,+,s_n}(CB_s)$ is the upper-half plane bounded by the horizontal line $l_{s}$ passing through $\phi_{1,+,s}(s)$. We call $CB_s$ the {\it upper main chessboard} of $g_s$. We parametrize naturally $l_{s}$ by $l_s(t) = t+i\mathfrak{Im}\,l_s$. Finally notice that $g_{s}:CB_s\longrightarrow CB_s$ is injective.

\paragraph{Puzzle piece bounding the Julia set.}
Let $s$ either satisfies Assumption \ref{assum.g_s-closeto-i} and has gate structure $(*,2)$, or $s= i$ (it will be the same for $s=-i$).
From the discussion in \cite[\S 3]{Ro}, for each fixed $k_0\geq 2$, one can construct $k_0$-periodic internal ray $R_{k_0}(g_s)$ with internal angle $\theta_{k_0}=\frac{1}{2^{k_0}-1}$. It starts from $\partial P_{1,+,s}$ and lands at a repelling $k_0$-periodic point $x_{k_0}(g_s)\in\partial B^*_s$ that admits a $\theta_{k_0}$-rotation under the action of $g_{s}$. We may take $k_0$ large enough (not depending on $s$) such that $x_{k_0}(g_s),g_i(x_{k_0}(g_s))\in P_{1,-,i}\cap P_{1,-,s}$. Let $R^\infty_{t_0}(g_s)$ be the external ray landing at $x_{k_0}(g_s)$. Then $R^\infty_{3t_0}(g_s)$ lands at $g_{s}(x_{k_0}(g_s))$. Let $E(g_s)\subset P_{1,-,i}\cap P_{1,-,s}$ be a segment of equipotential connecting $R^\infty_{t_0}(g_s)$ and $R^\infty_{3t_0}(g_s)$; let $\gamma_s\subset \partial P_{1,+,s}$ be an arc connecting the starting point of $R_{k_0}(g_s)$ and that of $g_{s}(R_{k_0}(g_s))$ so that $\phi_{1,+,s}(\gamma_s)$ does not depend on $s$. Let $Q_s$ be the bounded connected component enclosed by (see Figure \ref{fig.julia_control}) 
$${R_{k_0}(g_s)\cup g_s(R_{k_0}(g_s))\cup\gamma_s}\cup R^\infty_{t_0}(g_s)\cup E(g_s).$$

\begin{figure}[ht]%H为当前位置，!htb为忽略美学标准，htbp为浮动图形
\centering %图片居中
\includegraphics[width=0.9\textwidth]{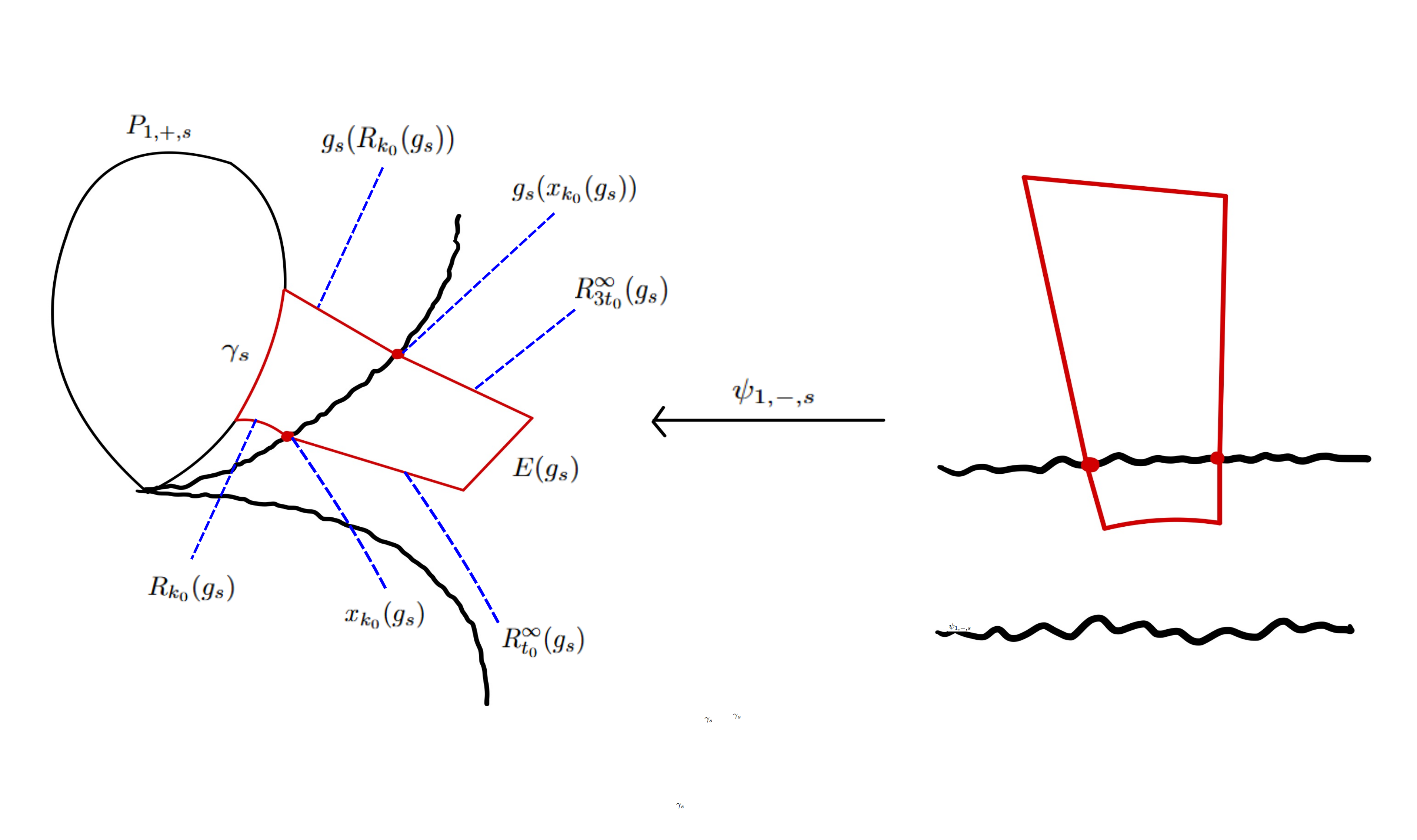} %插入图片，[]中设置图片大小，{}中是图片文件名
\caption{The puzzle piece $Q_s$ is enclosed by the red curve on the left-hand side; the red curve on the right-hand side is the pre-image under $\psi_{1,-,s}$.} %最终文档中希望显示的图片标题
\label{fig.julia_control} %用于文内引用的标签
\end{figure}

From the continuity dependence of perturbed Fatou coordinate (Lemma \ref{lem.perturb.Fatoucoordi}), the continuity dependence of repelling Koenigs coordinate and Böttcher coordinate, $\partial Q_s$ moves continuously under a mapping
\begin{equation}\label{eq.continuous-mouvement}
    h_s:\partial Q_i\longrightarrow \partial Q_s
\end{equation}
In particular $\partial Q_s$ converges uniformly to $\partial Q_i$ as $s$ tends to $i$. Let $\psi_{1,-,s}:\C\longrightarrow\C$ be the extended inverse Fatou coordinate of $\phi_{1,-,s}$. Notice that $\overline{Q_i}$ does not contain any critical value of $\psi_{1,-,i}$. Hence we can lift $Q_i$ by $\psi_{1,-,i}$ bijectively to some topological disk $\tilde{Q}_i$. Then there exists $N$ such that $\psi_{1,-,i}(\tilde{Q}_i-N)\subset P_{1,-,i}$. Set $Q_i^N = \psi_{1,-,i}(\tilde{Q}_i-N)$. Notice that $g_i^N$ sends $\overline{Q_i^N}$ bijectively to $\overline{Q_i}$. Thus for $s$ close enough to $i$, we can lift $h_s$ by $g_i^N$ and $g_s^N$ to $\overline{Q_i^N}$. Set $Q^N_s = h_s({Q_i^N})$. Thus by continuity dependence of perturbed Fatou coordinate (Lemma \ref{lem.perturb.Fatoucoordi}) $\phi_{1,-,s}(Q^N_s)$ converges uniformly to $\tilde{Q}_i-N$ as $s$ tends to $i$. By construction, $\partial[\phi_{1,-,s}(Q^N_s)]$ intersects $\psi_{1,-,s}^{-1}(J_s)$ at exactly two points $w$ and $w+1$. This leads to the following control on the size of Julia set for $g_s$:
\begin{prop}\label{Propositioncontrol.juliaset}
    Suppose that $s$ satisfies Assumption \ref{assum.g_s-closeto-i}. Let $y$ be any point in $\psi_{1,-,s}^{-1}(J_s)$ (recall that $\psi_{1,-,s}^{-1}(J_s)$ is the union of two $T_1$-invariant curves). Let $\mathfrak{l}$ be the segment of $\psi_{1,-,s}^{-1}(J_s)$ between $y$ and $y+1$. Then there exists $M>0$ not depending on $y$ and $s$, such that
$$\sup_{z_1,z_2\in\mathfrak{l}}|z_1-z_2|\leq M.$$
\end{prop}
\begin{proof}
    If $y$ belongs to the upper component of $\psi_{1,-,s}^{-1}(J_s)$, then the proposition follows from the boundedness of $\phi_{1,-,s}(Q_s)$. Suppose that $y$ belongs to the lower component. We may assume that $y$ has large negative real part so that $\psi_{1,-,s}(y)\in P_{1,-,s}$. Then there exists a $y'$ belonging to the upper component with smallest real part such that $\psi_{1,-,s}(y') = y$. It is easy to see that $y'$ is uniformly bounded for $s$ close to $i$. Let $\mathfrak{l}'$ be the segment of $\psi_{1,-,s}^{-1}(J_s)$ between $y'$ and $y'+1$. Then $\psi_{1,-,s}(\mathfrak{l})\subset\psi_{1,-,s}(\mathfrak{l'})$. Hence from the continuity dependence of $\psi_{1,-,s}$, we get the uniform bound for $\mathfrak{l}$.
\end{proof}

\begin{rem}\label{rem.puzzle.piecebounding}
    Suppose that $s$ does \textbf{not} satisfy Assumption \ref{assum.g_s-closeto-i} and is close to some $s_0\in\partial\mathcal{H}\setminus\{i,-i\}$. Then the same construction of the puzzle piece $Q_s$ also works for $s$ in a small neighborhood of $s_0$. Moreover, the continuous motion (\ref{eq.continuous-mouvement}) can be improved to holomorphic motion. 
\end{rem}

\subsection{A technical lemma}

\begin{lem}\label{lem.i.notonboundary}
    Let $\mathcal{E}^1\subset \boldsymbol{\mathrm{Esc}}^1_+$ be an adjacent component. Then $i\not\in\partial\mathcal{E}^1$ (see Figure \ref{fig.disjoint-adjacent}).
\end{lem}
\begin{proof}
    Assume by contradiction that $i\in\partial\mathcal{E}^1$. Then we can take a sequence $\mathcal{E}^1\ni s_n\to i$. By Lemma \ref{lem.perturb.Fatoucoordi}, \ref{lem.perturb.Fatoucoordi'} and \ref{lem.parabolic.attracting}, $g_{s_n}$ has either gate structure $(*,2)$ or $(2,*)$. Let us recall that for $s\in\mathcal{E}^1$, $\partial C^1_{s,(0)}$ contains two curves $\overline{\mathfrak{L}_s^u},\overline{\mathfrak{L}_s^l}$ invariant by $g_s$ (see Lemma \ref{lem.classification.C1} \ref{lem.classification.C1:third} and the third picture in Figure \ref{fig.escaperegion1}), where $\overline{\mathfrak{L}_s^u}$ links $0$ to the repelling fixed point $\alpha_s$ and $\overline{\mathfrak{L}_s^l}$ is a Jordan curve containing $0$. Moreover, $\mathfrak{L}_s^u$ is above $\mathfrak{L}_s^l$, $s$ is contained in the connected component $\Delta_s$ of $B_s^*\setminus\mathfrak{L}_s^u$ that contains $\mathfrak{L}_s^l$. Thus $g_{s_n}$ cannot have gate structure $(2,*)$, otherwise $\mathfrak{L}_{s_n}^l$ will be above $\mathfrak{L}_{s_n}^u$.
    
    The case less obvious to rule out is when $g_{s_n}$ has gate structure $(*,2)$. See Figure \ref{fig.case-to-ruleout}. Let us recall the perturbed Fatou coordinates $\phi_{1,\pm,s_n}$ and perturbed attracting/repelling petals $P_{1,\pm,s_n}$ for $g_{s_n}$ given by Lemma \ref{lem.perturb.Fatoucoordi}; also recall the upper main chessboard $CB_{s_n}$ constructed in \S \ref{subsec.julia-control}: it is a $g_{s_n}$-invariant Jordan disk $CB_{s_n}\in B^*_{s_n}$, such that $s,0\in\partial CB_{s_n}$ and $\phi_{1,+,s_n}(CB_{s_n})$ is the upper-half plane bounded by the horizontal line $l_{s_n}$ passing through $\phi_{1,+,s_n}(s_n)$. We naturally parametrize $l_{s_n}$ by $l_{s_n}(t) = t +i\mathfrak{Im}\,{l_{s_n}}$.
    
   % Let us first consider the parabolic degenerate polynomial $g_{i}$. It has two immediate basins at $0$. Denote by $B^*_i$ the one containing $i$. Let $\phi_{1,+,i}:B^*_i\longrightarrow \C$ be the attracting Fatou coordinate and $\phi_{1,-,i}:P_{1,-,i}:\longrightarrow\C$ be the repelling Fatou coordinate of repelling axis $\mathbb{R}^+$, where $P_{1,-,i}$ is a small repelling petal. We take the Fatou coordinates to have preferred normalization (Definition \ref{def.prefered.nor}).
%Let $\mathbb{U}$ be the upper-half plane bounded by the horizontal line passing through $\phi_{1,+,i}(i)$. Let $CB$ be the upper main chessboard in $B^*_i$: that is, the unique $g_{i}$-invariant connected component of $\phi_{1,+,i}^{-1}(\mathbb{U})$. Notice that $\partial\mathbb{U}$ is a Jordan curve, and $\mathbb{U}$ enters $P_{1,-,i}$. By Lemma \ref{lem.perturb.Fatoucoordi}, there exist 

    Set $\mathfrak{l}^u_{s_n} = \phi_{1,+,s_n}(\mathfrak{L}^u_{s_n})$. This is a non self-intersecting curve which is $T_1$-invariant. Let $\mathfrak{l}^u_{s_n}(t)$ with $t\in\mathbb{R}$ be any parametrization of $\mathfrak{l}^u_{s_n}$ such that $\mathfrak{l}^u_{s_n}(t+1) = \mathfrak{l}^u_{s_n}(t)+1$. We set $w = \phi_{1,+,s_n}(s_n)$. Notice that $\mathfrak{l}^u_{s_n}$ cannot be strictly above $l_{s_n}$, otherwise $\mathfrak{L}^u_{s_n}$ will be a Jordan curve. Since $s_n$ is the escaping critical point, $w$ is beneath $\mathfrak{l}^u_{s_n}$. Thus $\mathfrak{l}^u_{s_n}$ intersects $l_{s_n}$ and $\mathfrak{l}^u_{s_n}\cap l_{s_n}$  is $T_1$-invariant. For any $k\in\mathbb{Z}$, let $x_k$ be the left-hand side (i.e. $\mathfrak{Re}\,x_k<\mathfrak{Re}\,w+k$) closest point to $w+k$ in $l_{s_n}\cap \mathfrak{l}^u_{s_n}$. Let $\mathfrak{l}_k$ be the segment of $\mathfrak{l}^u_{s_n}$ between $x_k$ and $x_k+1$. Notice that we have $x_k+1=x_{k+1}$ (see Figure \ref{fig.+1invariant-curves}). Therefore for any curve $\gamma_k$ starting from $x_k$ ending at $x_{k+1}$ and above $l_{s_n}$, we have ($\sim$ represents homotopy):
    \begin{equation}\label{eq.homotop}
         \mathfrak{l}_k\sim\gamma_k\,\,\, \mathrm{rel}\,\,\{x_k,x_{k+1}\}\text{ on }\C\setminus\{w+1-m;\,m\in\mathbb{N}\},\,\,\text{for all }k\in\mathbb{Z}.
    \end{equation}
   Let $w' =  \phi_{1,+,s_n}(1/s_n)$. By Theorem \ref{thm.oudkerk} \ref{thm.oudkerk:sixth}, $\mathfrak{Re}\,w'\to-\infty$ as $n$ goes to infinity. By Proposition \ref{Propositioncontrol.juliaset}, we can take $n$ large enough, such that $$\mathfrak{Re}\,w'<\inf_{z\in\mathfrak{l}_k}\mathfrak{Re}\,z,\,\,\text{for all } k\geq 1.$$
    Thus we can choose $\gamma_k$ in (\ref{eq.homotop}) close to $l_{s_n}$, so that 
     \begin{equation}\label{eq.homotop'}
         \mathfrak{l}_k \sim\gamma_k\,\,\, \mathrm{rel}\,\,\{x_k,x_{k+1}\}\text{ on }\{w'+1-m;\,m\in\mathbb{N}\},\,\,\text{for all } k\geq 1. 
    \end{equation}
   For $k\in\mathbb{Z}$, let $\eta_k = (\phi_{1,+,s_n}|_{\overline{CB_{s_n}}})^{-1}(\gamma_k)$ set $z_k = (\phi_{1,+,s_n}|_{\overline{CB_{s_n}}})^{-1}(x_k)$. We associate any parametrization $\eta_k(t)$ so that $g_{s_n}(\eta_k(1)) = \eta_k(0)$. Then $\eta_k(1) = z_k$, $\eta_1(0) = z_{k+1}$. 
    Let $\mathfrak{L}_1(t)$ be the lift of $\mathfrak{l}_1$ by $\phi_{1,+,s_n}$ (notice that $\phi_{1,+,s_n}$ is globaly defined on the immediate parabolic basin $B^*_{s_n}$) such that $\mathfrak{l}_1(0) = z_1$. Then by (\ref{eq.homotop}), (\ref{eq.homotop'}) and Homotopy lifting theorem, 
    \begin{equation}\label{eq.L1homotop}
\mathfrak{L}_1\sim\eta_1\,\,\mathrm{rel}\,\, \{z_1,z_2\}\text{ on }
    B^*_{s_n}\setminus\{g_{s_n}(s_n),g_{s_n}(1/s_n)\}.
    \end{equation}
For $k\leq 0$, we define by induction $\mathfrak{L}_k$ so that $\mathfrak{L}_k$ is the lift of $\mathfrak{L}_{k+1}$ by $g_{s_n}$ with $\mathfrak{L}_{k+1}(0) = \mathfrak{L}_k(1)$. We prove by induction that 
 \begin{equation}\label{eq.Lkhomotop}
\mathfrak{L}_k\sim\eta_k\,\,\mathrm{rel}\,\, \{z_k,z_{k+1}\}\text{ on }
    B^*_{s_n}\setminus\{g_{s_n}(s_n),g_{s_n}(1/s_n)\}.
    \end{equation}
The case $k=1$ is just (\ref{eq.L1homotop}). Assume that the case $k$ is proved and consider the case $k-1$. Notice that by induction, $\mathfrak{L}_{k-1}$ and $\eta_{k-1}$ have the same end points $z_{k-1},z_k$. By (\ref{eq.homotop}), $\mathfrak{L}_{k-1}\cup\eta_{k-1}$ does not wind $g_{s_n}(s_n)$. It suffices to prove that it can neither wind $g_{s_n}(1/s_n)$. Recall that $l_{s_n}$ is the horizontal line passing through $w=\phi_{1,s_n,+}(s_n)$. Denote by $\xi := (\phi_{1,+,s_n}|_{{\overline{CB}_n}})^{-1}((0,+\infty)+i\mathfrak{Im}\,{l_{s_n}})$. Notice that $\phi_{1,+,s_n}^{-1}((0,+\infty)+i\mathfrak{Im}\,{l_{s_n}})$ has a connected component $\xi'$ that connects $\xi$ to a pre-image of $0$. Thus $\overline{\xi\cup\xi'}$ cuts $B^*_{s_n}$ into two connected components. Suppose by contradiction that $\mathfrak{L}_{k-1}\cup\eta_{k-1}$ winds $g_{s_n}(1/s_n)$. Then $\mathfrak{L}_{k-1}$ intersects with $\overline{\xi\cup\xi'}$, say at $z_*$, since it needs to pass through the $(*,2)$ gate. On the other hand, there exists $z_*'\in\mathfrak{L}_{k-1}$, such that $\mathfrak{Re}\,\phi_{1,+,s_n}(z_*')<\mathfrak{Re}\,\phi_{1,+,s_n}(g_{s_n}(1/s_n))$. By Theorem \ref{thm.oudkerk} \ref{thm.oudkerk:sixth}, $\mathfrak{Re}\,\phi_{1,+,s_n}(g_{s_n}(1/s_n))$ is very negative for $n$ large. But $\mathfrak{Re}\,\phi_{1,+,s_n}(z_*)$ is positive, which contradicts to the bound in Proposition \ref{Propositioncontrol.juliaset}. Hence (\ref{eq.Lkhomotop}) is proved. Thus we get that 
$\bigcup_{k\leq 1}\mathfrak{L}_{k}$ converges to $0$. This contradicts the fact that the one of the two end points of $\mathfrak{L}_{s_n}^u$ is the repelling fixed point of $g_{s_n}$. 
\end{proof}
\begin{figure}[ht]%H为当前位置，!htb为忽略美学标准，htbp为浮动图形
\centering %图片居中
\includegraphics[width=0.8\textwidth]{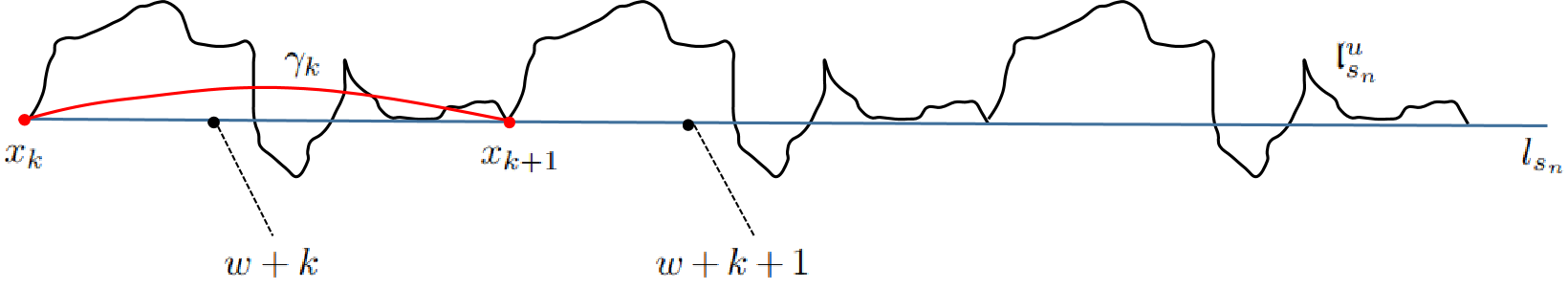} %插入图片，[]中设置图片大小，{}中是图片文件名
\caption{} %最终文档中希望显示的图片标题
\label{fig.+1invariant-curves} %用于文内引用的标签
\end{figure}
\begin{figure}[ht]%H为当前位置，!htb为忽略美学标准，htbp为浮动图形
\centering %图片居中
\includegraphics[width=0.4\textwidth]{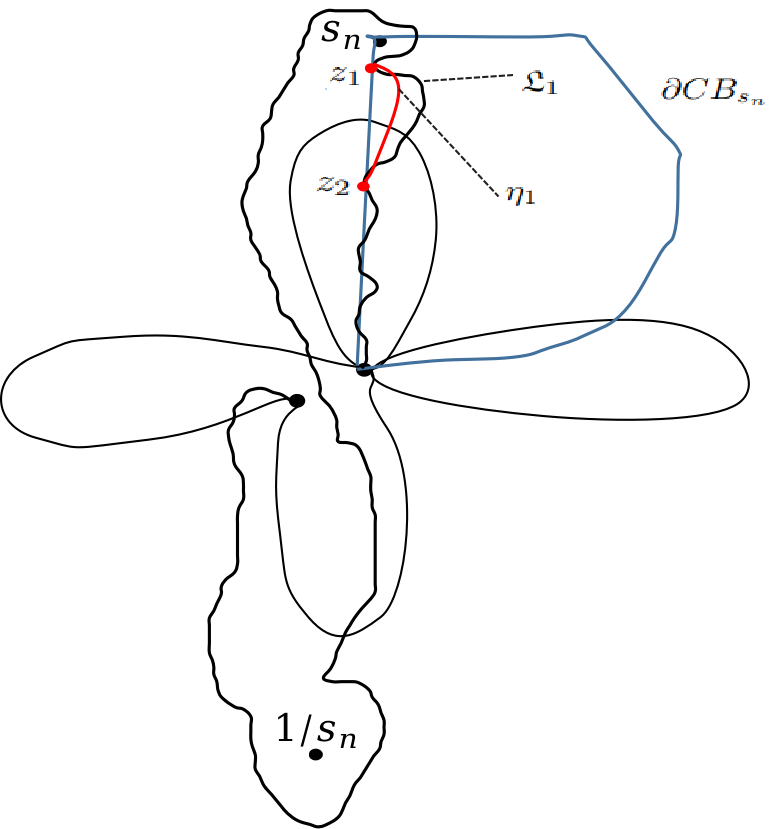} %插入图片，[]中设置图片大小，{}中是图片文件名
\caption{Illustration of the proof of Lemma \ref{lem.i.notonboundary}.} %最终文档中希望显示的图片标题
\label{fig.case-to-ruleout} %用于文内引用的标签
\end{figure}

\subsection{Invesgating the accumulation set via parametrization}
The escaping region for the Lavaurs map $L_{s_0}$ with $s_0\in\partial\mathcal{H}\setminus\{i,-i\}$ has similar properties as presented in (\ref{diag.commutative.criticalvalue}):
\begin{equation}\label{diag.commutative.boundary}
    \begin{tikzcd}
    {C}^N_{s_0}\arrow[d, "g^m_{s_0}\circ L^{N-1}_{s_0}" swap]\arrow[rrr, bend left=18, "\mu^m\circ\Phi_{s_0}"]&&&\mathbb{H}\arrow[d, "e^z"]\\
 C^1_{s_0,(0)}  \arrow[r, "\phi_{s_0}"] \arrow[rrru, bend left = 12, "\Phi_{s_0,(0)}"]& \psi^{-1}_{s_0}(\mathbb{C}\setminus K_{s_0})\arrow[r, "\psi_{s_0}"]\arrow[rru, bend left = 10, "\Psi_{s_0}" swap]  & \mathbb{C}\setminus K_{s_0} \arrow[r] \arrow[r,"\phi^\infty_{s_0}"] & \mathbb{D} 
\end{tikzcd}
\end{equation}
Notice that $\Psi_{s_0}$ is a conformal isomorphism, and the Julia set $J_{s_0}$ of $g_{s_0}$ is locally connected \cite{Ro}, thus $\Psi^{-1}_{s_0}$ extends continuously to $\overline{\mathbb{H}}$.

\begin{lem}\label{lem.accumulation.inside}
    Let $\mathcal{E}^1$ be a component of $\boldsymbol{\mathrm{Esc}}^1_+$. Suppose $\{s_n\}\subset\mathcal{E}^1$ is a sequence such that $$\Phi_{\mathcal{E}^1}(s_n)\to iy_0\in\partial{{\mathbb{H}}}\cup\{\infty\}.$$ 
    Moreover assume that \{i,-i\} does not intersect the accumulation set of $s_n$. Then the accumulation set of $\{s_n\}$ is contained in $\mathcal{H}$ if and only if $y_0\not\in\{0,\infty\}$.
\end{lem}
\begin{proof}
{\it The “if” part}: suppose by contradiction that $\{s_n\}$ converges to $s_0\in\partial\mathcal{H}\setminus\{i,-i\}$. Up to taking a subsequence of $\Phi_{\mathcal{E}^1}(s_n)$, we may assume that there is a curve $l\subset\mathbb{H}$ containing all $\Phi_{\mathcal{E}^1}(s_n)$ and converging to $iy_0$. Clearly $s_0$ is contained in the accumulation set of $\mathcal{R} := \Phi_{\mathcal{E}^1}^{-1}(l)$. Fix $N\geq0$ such that $\Psi^{-1}_{s_0}(iy_0/3^N)$ belongs to some left half-plane $\mathbb{H}_{-v}$ on which $\psi_{s_0}$ is injective. Denote by $z_0 = \Psi^{-1}_{s_0}(iy_0/3^N)$. By Remark \ref{rem.puzzle.piecebounding}, there exists a puzzle piece $Q_{s_0}\subset\psi_{s_0}(\mathbb{H}_{-v})$ containing $\psi(z_0)$, whose boundary consists of \begin{itemize}
        \item  two external rays in $\mathbb{C}\setminus K_0$, each of which colands with an internal ray in $B^*_{s_0}$, 
        \item a segment of equipotential in $\mathbb{C}\setminus K_0$ and a segment of equipotential $B^*_{s_0}$ (defined by the vertical line in the Fatou coordinate).
    \end{itemize}
If $s$ is close enough to $s_0$, we can follow holomorphically $\partial Q_{s_0}$, that is, there exists a dynamical holomorphic motion $h_s:\partial Q_{s_0}(z_0)\longrightarrow \partial Q_{s}$, where $Q_s$ is the puzzle piece for $g_s$. See \cite[\S 5.2-5.4]{Z} for details of the construction and the holomorphic motion of internal rays in the parabolic basin.

\iffalse
    On the other hand, by (\ref{diag.commutative.criticalvalue}) in Proposition \ref{Propositionbotthcer-lavaurs.coordinate}, there exists a smallest $k\geq0$ such that for $s\in\mathcal{E}^1$, we have the commutative diagram:
\begin{equation}\label{diag.specialcritical}
    \begin{tikzcd}
    {Cv}^1_{s}\arrow[d, "g^k_{s}" swap]\arrow[rrr, bend left=18, "\mu^k\circ\Phi_{s}"]&&&\mathbb{H}\arrow[d, "e^z"]\\
    C^1_{s,(0)}  \arrow[r, "\phi_{s}"] \arrow[rrru, bend left = 12, "\Phi_{s,(0)}"]& \psi^{-1}_{s}(\mathbb{C}\setminus K_{s})\arrow[r, "\psi_{s}"]\arrow[rru, bend left = 10, "\Psi_{s}" swap]  & \mathbb{C}\setminus K_{s} \arrow[r] \arrow[r,"\phi^\infty_{s}"] & \mathbb{D} 
   \end{tikzcd}
\end{equation}
\fi
By Proposition \ref{Propositionbotthcer-lavaurs.coordinate}, the condition that $\Phi_{\mathcal{E}^1}(s)(=\Phi_{s}(v_s))$ converges to $iy_0$ as $s$ converges to $s_0$ along $\mathcal{R}$ tells us that 
    \begin{equation}\label{eq.belonging}
        \Psi^{-1}_{s}(3^{-N}\Phi_{s}(v_s))\in(\psi_s|_{\mathbb{H}_{-v}})^{-1}(Q_s),\,\,s\in\mathcal{R}\text{ close enough to }s_0. 
    \end{equation}
Here we modify $v>0$ to be larger if necessary so that for $s$ close to $s_0$, $\psi_s$ is injective on $\mathbb{H}_{-v}$. Since $\partial Q_s$ is close to $\partial Q_{s_0}$, $\psi_s|_{\mathbb{H}_{-v}}$ is close to $\psi_{s_0}|_{\mathbb{H}_{-v}}$ (because of the holomorphic dependence of Fatou coordinate), (\ref{eq.belonging}) implies that $\phi_s(v_s)-N-k$ is bounded in $s$. However this contradicts Lemma \ref{lem.unbounded.curveR}.

\vspace{0.1in}
{\it The “only if” part}: since the accumulation set of $s_n$ is contained in $\mathcal{H}$, $\phi_{s_n}(v_{s_n})$ is bounded. By Lemma \ref{lem.psi} (\ref{diag.psi}) and Proposition \ref{Propositionbotthcer-lavaurs.coordinate}, we have $\Phi_{s_n,(0)}(v_{s_n}) = \Psi_{s_n}(\phi_{s_n}(v_{s_n}))$. Notice that $\Psi_{s_n}:\C\setminus K_s\longrightarrow \mathbb{H}$ is an isomorphism and $\Psi_{s_n}^{-1}$ extends homeomorphically to $\partial\overline{\mathbb{H}}\setminus\{0\}$ such that $\Psi(w)\to\infty$ if $\mathfrak{Re}\,w\to+\infty$ and $\Psi(w)\to0$ if $\mathfrak{Re}\,w\to-\infty$. Also notice that since $s$ is contained in $\mathcal{H}$, $K_s$ moves holomorphically, and so does $\overline{\psi_s^{-1}(\C\setminus K_s)}$. Thus for any fixed compact set $K\subset\C$ and any small neighborhood $\mathcal{V}$, $\Psi(s,z) = \Psi_s(z)$ extends to a continuous function on the following compact set:
$$\{(s,z);\,s\in\overline{\mathcal{V}},\,z\in K\cap \overline{\psi_s^{-1}(\C\setminus K_s)}\}.$$
Thus $\Psi_{s_n}(\phi_{s_n}(v_{s_n}))$ does not accumulate to $0,\infty$. 
\end{proof}

\subsection{Disjoint type components}

\begin{lem}\label{lem.disjoint.type1}
      Let $\mathcal{E}^1$ be a disjoint component of $\boldsymbol{\mathrm{Esc}}^1_+$. Then for any $s\in\mathcal{E}^1$, $C^1_{s,(0)}$ is type \ref{lem.classification.C1:first} in Lemma \ref{lem.classification.C1}.
\end{lem}
\begin{proof}
    Clearly $C^1_{s,(0)}$ cannot be type \ref{lem.classification.C1:third}. Suppose the contrary that $C^1_{s,(0)}$ is type \ref{lem.classification.C1:second}. Let $R^\infty_s(t_0),R^\infty_s(t_0')$ be the two external rays landing at $\partial Cv^1_s$. Without loss of generality, we may assume that $0<t_0<t_0'<\frac{1}{2}$. Therefore, the pre-images of $t_0,t_0'$ under the angle tripling map and $t_0,t_0'$ satisfy
    $$\frac{t_0}{3}<\frac{t_0'}{3}<t_0<t_0'<\frac{t_0+1}{3}<\frac{t_0'+1}{3}<\frac{1}{2}<\frac{t_0+2}{3}<\frac{t_0'+2}{3}.$$
    Hence $R^\infty_s(\frac{t_0}{3}),R^\infty_s(\frac{t_0'}{3}),R^\infty_s(\frac{t_0+1}{3}),R^\infty_s(\frac{t_0'+1}{3})$ are the four external rays landing at $\partial Cp^1_s$; $R^\infty_s(\frac{t_0+2}{3}),R^\infty_s(\frac{t_0'+2}{3})$ are the two external rays landing at $\partial Cc^1_s$. 
    
    Take any equipotential $Eq$ of $g_s$. Let $V$ be the connected component of $\C\setminus(\overline{R^\infty_s(t_0)\cup R^\infty_s(t_0')\cup Cv^1_s})$
    containing $g_s(1/s)$. Notice that $U$ must contain $C^1_{s,(0)}$, otherwise $Cv^1_s$ would have 3 pre-images under $g_s$ with total degree 4, but $g_s$ is a cubic polynomial. Let's take a slightly smaller $\tilde{V}\subset V$ such that $\tilde{V}$ is a Jordan domain and $C^1_{s,(0)}\subset\tilde{V}$. Let $\tilde{U}$ be the connected component of $g_s^{-1}(\tilde{V})$ containing $C^1_{s,(0)}$. Then $\tilde{U}\Subset\tilde{V}$ and $g_s:\tilde{U}\longrightarrow\tilde{V}$ is a degree two polynomial-like map in the sense of Douady-Hubbard \cite{DoHu2}, which is quasiconformally conjugate to a quadratic polynomial. However $g_s|_{\tilde{U}}$ has two fixed access to its filled in Julia set: $R^\infty_s(0)$ and $R^\infty_s(\frac{1}{2})$. This is impossible for a quadratic polynomial, a contradiction.
\end{proof}

\begin{lem}\label{lem.cross.gate}
    Let $\mathcal{E}^1$ be a disjoint component of $\boldsymbol{\mathrm{Esc}}^1_+$. For $s\in\mathcal{E}^1$, $Cp_s^1$ is a bi-croissant. Let $\{z_1,z_2\} = Cp_s^1\cap \partial B^*_s$ and $R^\infty_s(t_1),R^\infty_s(t_2)$ $(t_1<t_2)$ be the external rays landing at them respectively. Then $0<t_1<\frac{1}{2}<t_2<1$.
\end{lem}
\begin{proof}
The fact that $Cp_s^1$ is a bi-croissant follows immediately from Lemma \ref{lem.disjoint.type1}. Suppose the contrary that $t_1<t_2$ are not ordered as stated in the Lemma. Let $t_0$ be the angle of the external ray landing at $\partial Cv_s^1$. We may assume $0<t_0<\frac{1}{2}$. Then 
$$\frac{t_0}{3} = t_1<t_0<t_2 = \frac{t_0+1}{3},\quad \frac{1}{2}<t_3 := \frac{t_0+2}{3}.$$ 

Take any equipotential $Eq$ of $g_s$. Let $V$ be the connected component of 
$$\mathbb{C}\setminus(Eq(r_0)\cup R^\infty_s(t_0)\cup\overline{Cv_s^1})$$ containing $g_s(1/s)$. Clearly $U$ contains the two fixed points of $g_s$.
Now one can procced the same argument as for Lemma \ref{lem.disjoint.type1} to conclude.
\end{proof}

\begin{lem}\label{lem.nondouble.parabolic}
   Let $\mathcal{E}^1$ be a disjoint component of $\boldsymbol{\mathrm{Esc}}^1_+$. Then $\partial\mathcal{E}^1$ contains neither the two degenerate parabolic parameters $s=i,-i$.  
\end{lem}
\begin{proof}
    Suppose by contradiction that $i\in\partial\mathcal{E}^1$ (it will be the same for $-i$). By Lemma \ref{lem.perturb.Fatoucoordi}, \ref{lem.perturb.Fatoucoordi'} and \ref{lem.parabolic.attracting}, there exists $s_0\in\mathcal{E}^1$ arbitrarily close to $i$ such that $g_{s_0}$ has gate structure $(*,2)$ or $(2,*)$. We give the proof for $(*,2)$ and the other is similar. Let $P^\pm_{s_0,1},P_{s_0}$ be the petals corresponding to those in Lemma \ref{lem.perturb.Fatoucoordi}. Notice that $$\mathbb{C}\setminus(R^\infty_{s_0}(0)\cup R^\infty_{s_0}(\frac{1}{2})\cup\overline{P_{s_0}})$$
    has two connected components, therefore by Lemma \ref{lem.cross.gate}, $Cc^1_{s_0}$ passes through the gate $P_{s_0}$. Recall the Böttcher-Lavaurs coordinate $\Phi_{s,(0)}$ in Lemma \ref{lem.botthcer-lavaurs.coordinate}. Since $\mathcal{E}^1$ is disjoint, $\Phi_{s_0,(0)}$ extends bijectively to $C^1_{s_0,(0)}\longrightarrow\mathbb{H}$, where $C^1_{s_0,(0)}$ is the unique connected component of $E^1_{s_0}$ such that $g_{s_0}(C^1_{s_0,(0)})=C^1_{s_0,(0)}$. Suppose $k$ is the smallest integer such that $g^k_{s_0}(Cp^1_{s_0}) = C^1_{s_0,(0)}$. Let ${e}$ be the half-line passing through the origin and $\Phi_{s_0,(0)}(g^k_{s_0}(s_0))$. Set $e_0 := \Phi^{-1}_{s_0,(0)}(e)$ and thus $e_0 = e_0(t)$ is naturally parametrized by $e$. The pre-image of $e_0$ by $g^k_{s_0}$ in $Cp^1_{s_0}$ is a figure “eight” with joint $s_0$. Consider the curve $\Gamma$ contained in the figure eight passing through $z_1,s_0,z_2$ such that $g^k_{s_0}$ sends bijectively $\Gamma$ to $e_0$, with $\Phi_{s_0,(0)}(g^k_{s_0}(z))\to\infty$ as $\Gamma\ni z\to z_1$. Thus we can parametrize $\Gamma = \Gamma(t)$ by $e$. Consider the last moment $t_1$ that $\Gamma$ intersects $l_{2,+,s_0}$ and the first moment $t_2$ after $t_1$ such that $\Gamma$ intersects $g^k_{s_0}(l_{2,+,s_0})$ (thus $t_2>t_1$). See Figure \ref{fig.curve-pass-gate}. Thus $\Gamma(t_1,t_2)\subset  P_{s_0}$. Since $k$ does not depend on $s\in\mathcal{E}^1$, we may take $s_0$ close enough to $i$ such that $g_{s_0}^k(\Gamma([t_1,t_2]))\subset P_{s_0}$, $g_{s_0}^{2k}(\Gamma([t_1,t_2]))\subset P_{s_0}$. By construction, $g_{s_0}^k(\Gamma([t_1,t_2]))\subset e_0=\Phi^{-1}_{s_0,(0)}(e)$ and $e_0$ is ${g_{s_0}}$-invariant. Therefore $g^{2k}_{s_0}(\Gamma(t_1)) = g^{k}_{s_0}(\Gamma(t_2))$. Thus $g^k_{s_0}(\Gamma(t_2))$ has two pre-images under $g^{k}_{s_0}$ in $P_{s_0}$, while $g^k_{s_0}$ is injective on $P_{s_0}$. This leads to a contradiction.
\end{proof}
\begin{figure}[ht]%H为当前位置，!htb为忽略美学标准，htbp为浮动图形
\centering %图片居中
\includegraphics[width=0.6\textwidth]{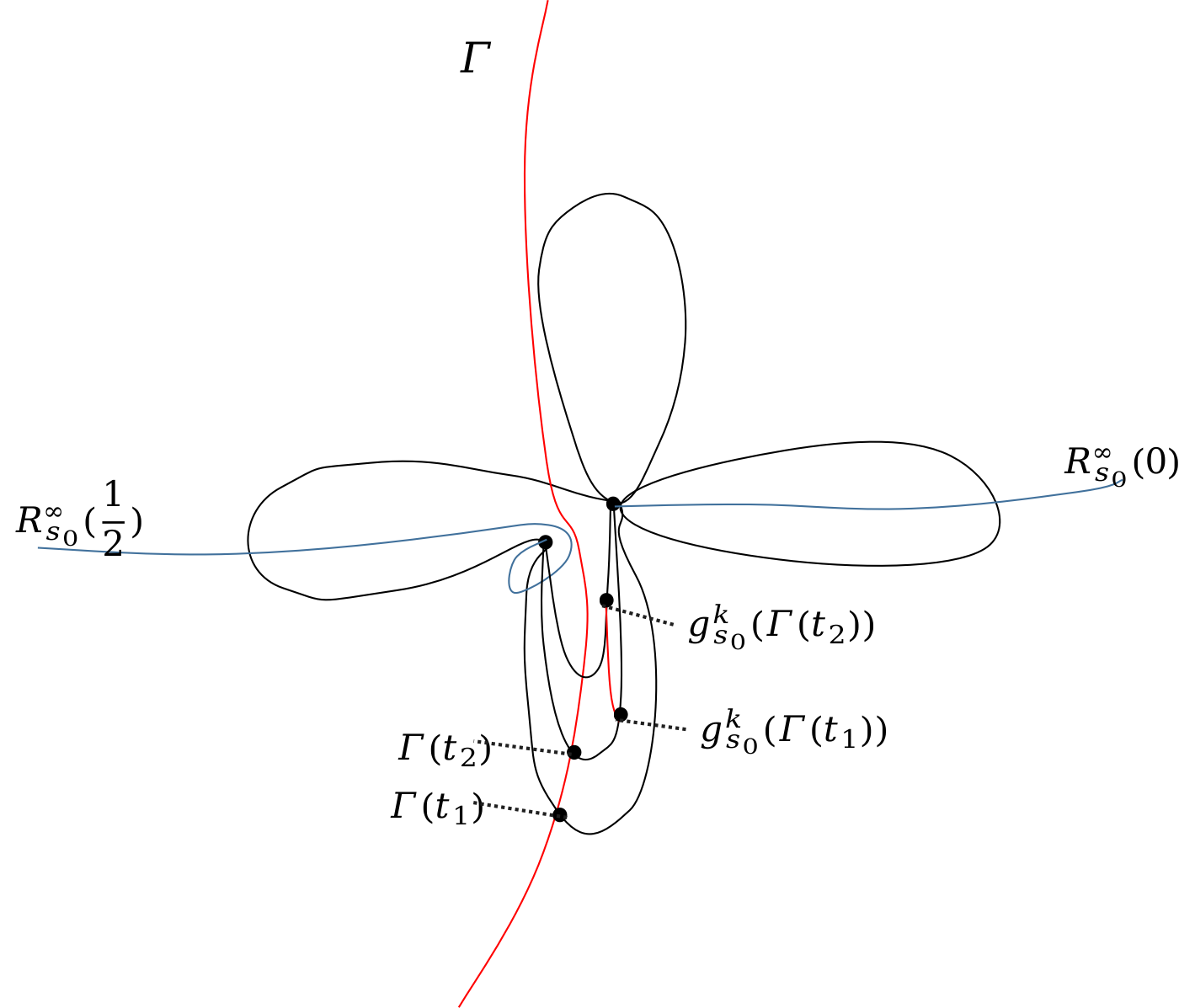} %插入图片，[]中设置图片大小，{}中是图片文件名
\caption{Illustration of the proof for Lemma \ref{lem.nondouble.parabolic}.} %最终文档中希望显示的图片标题
\label{fig.curve-pass-gate} %用于文内引用的标签
\end{figure}

\begin{lem}\label{lem.existence.disjoint}
    Let $s_0\in\partial\mathcal{H}$ be a Misiurewicz parabolic parameter such that $g_{s_0}^n(s_0) = 0$ for some $n\geq 1$. Then there exists a unique component $\mathcal{E}^1\subset \boldsymbol{\mathrm{Esc}}^1_+$ of disjoint type such that $s_0\in\partial\mathcal{E}^1$. Moreover $s_0$ is accessible by a curve in $\mathcal{E}^1$.
\end{lem}
\begin{proof}
    By Diagram (\ref{diag.commutative.boundary}), there exists a component $C_{s_0}\subset E^1_{s_0}$ and $k\geq 0$ such that $v_{s_0}\in \partial C_{s_0}\cap \partial B_{s_0}$ and $g_{s_0}^k(C_{s_0}) = C^1_{s_0,(0)}$. For $s$ near $s_0$, denote by $\mathfrak{R}_{s}(0)$ the unique component of $\psi_{s}^{-1}(R^\infty_{s}(0))$ that is $T_1$-invariant ($R^\infty_{s}(0)$ is the external ray of angle 0 for $g_{s}$). By continuity dependence on Fatou coordinate, $\psi_s|_{\mathbb{H}_{-v}}$ is injective for some $v>0$ large enough. Therefore $\mathfrak{R}_{s}(0)\cap \mathbb{H}_{-2v}$ moves holomorphically by $\psi^{-1}_s\circ\zeta_s\circ\psi_{s_0}$, where $\zeta_s$ is the holomorphic motion of $R^\infty_{s_0}(0)$ induced by $(\phi^\infty_{s})^{-1}\circ\phi^\infty_{s_0}$. Since $\mathfrak{R}_{s_0}$ is $T_1$-invariant, we can extend the motion to $\mathfrak{R}_{s_0}$. Again by continuity dependence on Fatou coordinate, $\phi_s$ is injective on some attracting petal $P_s$ so that $\phi_s(P_s) = \mathbb{H}_{v}$. Therefore $(\phi_{s_0}|_{\mathbb{H}_{2v}})^{-1}(\mathfrak{R}_{s_0})$ moves holomorphically by $(\phi_{s}|_{\mathbb{H}_{2v}})^{-1}\circ\phi_{s_0}$. By Lemma \ref{lem.classification.C1}, there exists $N>0$ such that the pull back of $(\phi_{s_0}|_{\mathbb{H}_{2v}})^{-1}(\mathfrak{R}_{s_0})$ by $g^N_{s_0}$ enters the repelling petal of $g_{s_0}$. By continuity dependence of Fatou coordinate, the pull back of $(\phi_{s}|_{\mathbb{H}_{2v}})^{-1}(\mathfrak{R}_{s})$ by $g^N_s$ is also contained in the repelling Fatou petal of $g_s$ if $s$ is close to $s_0$. Therefore the holomorphic motion $(\phi_{s}|_{\mathbb{H}_{2v}})^{-1}\circ\phi_{s_0}$ can be pulled back infintely by $g_s$, so that we get a motion $h_s$ of a $g_{s}$-invariant Jordan curve, denoted by $R_{s,C^1}(0)$, so that $L_s(R_{s,C^1}(0)) = R^\infty_s(0)$, $R_{s,C^1}(0)\cap\partial K_{s} = \{0\}$, $h_s(R_{s_0,C^1}(0)) = R_{s,C^1}(0)$.

    Notice that the orbit $v_{s_0},g_{s_0}(v_{s_0}),...,g^k_{s_0}(v_{s_0}) = 0$ moves homolorphically for $s$ near $s_0$. Now pull back $h_s$ by $g_s$ along this orbit, we get a holomorphic motion of $R_{s,C}(0)$ which is a Jordan curve attached at a $k$-th pre-image of $0$. Apply Rouché's Theorem to the equation of $s$: \begin{equation}\label{eq.Rouche}
        h_s(z) =v_s.
    \end{equation} 
    Then one can find a curve $\Gamma\subset\mathcal{H}$ converging to $s_0$ so that for $s\in\Gamma$, $s$ is a solution to (\ref{eq.Rouche}) and $\Phi_{s,C^1_{s,(0)}}(g_s^k(v_s))\to 0$ as $s\to s_0$ along $\Gamma$. This implies that $\Gamma$ is contained in a component $\mathcal{E}^1$ of $\boldsymbol{\mathrm{Esc}}^1_+$. Moreover, we claim that such curve is unique: notice that the motion $h_s$ can be extended to the external ray $R_{s_0}^{\infty}(t_0)$ that lands at $v_{s_0}$; while there is a unique parameter external ray $\mathcal{R}^{\infty}$ landing at $s_0$, hence $s=s_0$ is the unique solution to (\ref{eq.Rouche}) when $z = v_{s_0}$.

    It remains to prove that $\mathcal{E}^1$ is of disjoint type. Suppose the contrary. Then $\mathcal{E}^1$ is of adjacent type.  
    Let $s\in\Gamma$ and suppose that $g_s$ is $(1,\omega)$-non escaping. Recall Proposition \ref{Propositionbotthcer-lavaurs.coordinate} and Diagram (\ref{diag.commutative.criticalvalue_case(3)}), we get that $v_s\in\partial C^1_{s,(0)}(3\omega)$. Notice that $C^1_{s,(0)}(3\omega)$ is a topological disk attached at $0$ and is forward invariant by $g_s$. On the other hand, notice that for $s\in\Gamma$, $$\phi_s(v_s) = \phi_s(g_s^k(v_s))-k\in\mathfrak{R}_s(0).$$
    Thus from the construction of $R_{s,C^1}(0)$, we get $v_s\in R_{s,C^1}(0)$, $s\in R_{s,C^1}(0)$. But $s_0\not\in R_{s_0,C^1}(0)$, and $R_{s,C^1}(0)$ moves holomorphically. This leads to a contradiction.
    \end{proof}

\begin{rem}\label{rem.motion}
   It is easy to see that the construction of the holomorphic motion of $R_{s_0,C^1}(0)$ in Lemma \ref{lem.existence.disjoint} works for any $s_0\in\partial\mathcal{H}\setminus\{i,-i\}$.
\end{rem}

\begin{cor}\label{cor.internal-curve-mustland}
   Let $\mathcal{E}^1\subset\boldsymbol{\mathrm{Esc}}^1_+$ be a connected component. Let $l\subset\mathbb{H}$ be any curve parametrized by $l(t)$, $t\in[0,+\infty)$ such that $l(t)$ converges to $y_0\in\partial\mathbb{H}\cup\{\infty\}$ as $t$ tends to infinity. Then $\mathcal{R}(t) := \Phi_{\mathcal{E}^1}^{-1}(l(t))$ also converges to a singleton as $t\to+\infty$.
\end{cor}
\begin{proof}
     Suppose the contrary. Then the accumulation set $\mathrm{accum}(\mathcal{R})$ of $\mathcal{R}$ is a non trivial continuum. By Lemma \ref{lem.existence.disjoint}, for every Misiurewicz parameter, there is a component of $\mathrm{Esc}^1$ of disjoint type attached to it. Thus $\mathrm{accum}(\mathcal{R})$ is not contained in $\partial\mathcal{H}$, since Misiurewicz parabolic parameter is dense in $\partial\mathcal{H}$ by Theorem \ref{thm.runze}. Thus $\mathrm{accum}(\mathcal{R})\not\subset\partial\mathcal{H}$. Take any $s_0\in\mathrm{accum}(\mathcal{R})\cap\mathcal{H}$ and let $s_n\in\mathcal{R}$ be a sequence converging to $s_0$.
     \begin{itemize}
         \item   If $y_0\not = \infty$, then by Diagram (\ref{diag.commutative.criticalvalue}) and (\ref{diag.commutative.criticalvalue_case(3)}), $\phi_{s_n}^\infty(L_{s_n}(v_{s_n}))$ converges to $e^{iy_0}$. Since $s_0\in\mathcal{H}$, $g_{s_0}$ is $J$-stable in the family $g_s$. Thus $J_s$ moves holomorphically for $s$ near $s_0$. In particular, $(\phi^\infty_{s})^{-1}(e^{iy_0})$ moves holomorphically. By taking limit in $n$, we get $L_{s_0}(v_{s_0}) = (\phi^\infty_{s_0})^{-1}(e^{iy_0})$. Hence the holomophic map $s\mapsto L_s(v_s)-(\phi^\infty_{s})^{-1}(e^{iy_0})$ is constantly zero by analytic continuation. On the other hand, it does not vanish for $s\in\mathcal{R}$. This leads to a contradiction.
     \item If $y_0 = \infty$, then by Diagram (\ref{diag.commutative.criticalvalue}) and (\ref{diag.commutative.criticalvalue_case(3)}), $\phi_{s_n}(v_{s_n})$ tends to infinity. This contradicts $s_0\in\mathcal{H}$.
     \end{itemize}
\end{proof}

\begin{lem}\label{lem.disjoint.accumulation}
    Let $\mathcal{E}^1$ be a disjoint component of $\boldsymbol{\mathrm{Esc}}^1_+$. Suppose $\{s_n\},\{s'_n\}\subset\mathcal{E}^1$ are two sequences such that $\Phi_{\mathcal{E}^1}(s_n)\to 0$ and $\Phi_{\mathcal{E}^1}(s'_n)\to\infty$. Then $s_n$ and $s'_n$ converge to the same Misiurewicz parabolic parameter on $\partial\mathcal{H}$.
\end{lem}

\begin{proof}
    By Lemma \ref{lem.accumulation.inside}, $s_n,s_n'$ do not accumulate to point in $\mathcal{H}$. 
    It remains to show that $s_n,s'_n$ converge to the same Misiurewicz parabolic parameter of $\partial\mathcal{H}$. First we prove the special case when $\Phi_{\mathcal{E}^1}(s_n)\to 0, \Phi_{\mathcal{E}^1}(s_n)(s'_n)\to\infty$ along the real axis. By Lemma \ref{lem.nondouble.parabolic}, the accumulation point $s_0,s'_0\not\in\{i,-i\}$. Recall the holomorphic motion of $R_{s_0,C^1}(0),R_{s'_0,C^1}(0)$ constructed in Lemma \ref{lem.existence.disjoint} (Remark \ref{rem.motion}). By hypothesis $v_{s}\in R_{s,C^1}(0)$ for $s = s_n$ or $s'_n$. By taking limit we get that $v_{s_0},v_{s'_0}$ is a pre-image of $0$ by $g_{s_0},g_{s'_0}$ respectively, i.e. the accumulation is Misiurewicz parabolic. It is clear that $s_0 = s'_0$, since for $s\in\mathcal{E}^1$, the angle of the external ray attached at $R_{s,C}(0)$ remains constant (recall that $R_{s,C}(0)$ is the pull back of $R_{s,C^1}(0)$) along the critical value orbit).

    For general case of convergence of $s_n,s_n'$, we only prove for $s_n$ and it will be the same for $s_n'$. By passing to a subsequence, we can take a curve $l\subset\mathbb{H}$ such that $l\cap(0,+\infty) = \emptyset$, $l$ contains all $\Phi_{\mathcal{E}^1}(s_n)$ and $l$ converges to $0$. Set $\mathcal{R} = \Phi_{\mathcal{E}^1}^{-1}(l)$. By Lemma \ref{lem.accumulation.inside} and Corollary \ref{cor.internal-curve-mustland}, $\mathcal{R}$ converges to some $\tilde{s}_0\in\partial\mathcal{H}$. From be above paragraph, $\Phi_{\mathcal{E}^1}^{-1}((0,+\infty))$ converges to $s_0$. Take an arc $\gamma\subset\mathbb{H}$ linking $l$ and $(0,\infty)$ so that the three of them bound a Jordan domain ${U}$. Then for any curve $\tilde{l}\subset U$ converging to $0$, $\tilde{\mathcal{R}} := \Phi^{-1}_{\mathcal{E}^1}(\tilde{l})$, converges to a single point on $\partial\mathcal{H}$ by Lemma \ref{lem.accumulation.inside} and Corollary \ref{cor.internal-curve-mustland}. Thus $\Phi^{-1}_{\mathcal{E}^1}(U)$ is also a Jordan domain whose boundary consists of $\Phi^{-1}_{\mathcal{E}^1}(\overline{U}\setminus\{0\})$ and a closed arc between $s_0,\tilde{s}_0$ in $\partial\mathcal{H}$ (recall that $\partial\mathcal{H}$ is a Jordan curve by Theorem \ref{thm.runze}). But this is impossible, since this contradicts Carathéodory's extension theorem. Thus $s_0$ has to be $\tilde{s}_0'$ and we have finished the proof.
\end{proof}

\subsection{Adjacent components and conclusion}
\begin{lem}\label{lem.adjacent.accumulation}
    Let $\mathcal{E}^1\subset \boldsymbol{\mathrm{Esc}}^1_+$ be an adjacent component. Let $\{s_n\}\subset\mathcal{E}^1$ be a sequence such that $\Phi_{\mathcal{E}^1}(s_n)\to iy_0\in\partial{{\mathbb{H}}}$. Then $s_n\to-i$ if and only if $y_0\in\{0,\infty\}$.
\end{lem}
\begin{proof}
    {\it The “if” part}. To fix the idea, let us assume that $y_0=0$; the case $y_0=\infty$ will be similar. Suppose the contrary, then there is a subsequence $s_{n_i}$ such that $s_{n_i}$ does not converges to $-i$. Up to taking a subsequence of $s_{n_i}$, we may assume that $s_{n_i}$ converges to $s_0$ and that there is a curve $\mathcal{R}\subset\mathcal{E}^1$ such that $\mathcal{R}$ passes through all $s_{n_i}$, $\Phi_{\mathcal{E}^1}(\mathcal{R})$ converges to $0$ and $\mathcal{R}\cap\Phi^{-1}_{\mathcal{E}^1}((0,+\infty)) = \emptyset$. By Lemma \ref{lem.i.notonboundary} and \ref{lem.accumulation.inside}, $s_0\in\partial\mathcal{H}\setminus\{i,-i\}$; by Corollary \ref{cor.internal-curve-mustland}, $\mathcal{R}$ converges to $s_0$.

    Again by Corollary \ref{cor.internal-curve-mustland}, $(\Phi_{\mathcal{E}^1})^{-1}(t)$ converges to a singleton as $t\to 0^+$. We claim that this singleton is either $i$ or $-i$. Suppose the contrary, then by Lemma \ref{lem.accumulation.inside}, the accumulation set contains some $s'_0\in\partial\mathcal{H}\setminus\{i,-i\}$. Under the same construction and notations as in the proof of Lemma \ref{lem.existence.disjoint}, there is a holomorphic motion of a $g_{s'_0}$-invariant Jordan curve $R_{s'_0,C^1}(0)$ for $s$ close to $s'_0$. Moreover, in the dynamical plane of $g_{s}$, $s\not\in R_{s,C^1}(0)$. However, for $s\in\Phi^{-1}_{\mathcal{E}^1}((0,+\infty))$, when one pulls back $(\Phi_{s,C^1_{s,(0)}})^{-1}(\omega,+\infty)$ by $g_s$ with some large $\omega$, the pulled back curve will eventually hit $s$, since $\mathcal{E}^1$ is of adjacent type. This contradicts the construction of $R_{s,C^1}(0)$.

   Thus by Lemma \ref{lem.i.notonboundary}, $\lim\limits_{t\to 0^+}(\Phi_{\mathcal{E}^1})^{-1}(t) = -i$. Link $\Phi_{\mathcal{E}^1}(\mathcal{R})$ and $(0,+\infty)$ by an arc $\gamma$ so that the three of them bounds a Jordan domain $U$. Let $\tilde{{l}}\subset U$ be any curve converging to $0$. Then by Corollary \ref{cor.internal-curve-mustland} and Lemma \ref{lem.accumulation.inside}, $\tilde{\mathcal{R}} := \Phi^{-1}_{\mathcal{E}^1}(\tilde{l})$, converges to a single point on $\partial\mathcal{H}$. Thus $\Phi^{-1}_{\mathcal{E}^1}(U)$ is also a Jordan domain whose boundary consists of $\Phi^{-1}_{\mathcal{E}^1}(\overline{U}\setminus\{0\})$ and a closed arc between $-i,s_0$ in $\partial\mathcal{H}$ (recall that $\partial\mathcal{H}$ is a Jordan curve by Theorem \ref{thm.runze}). But this is impossible, since this contradicts Carathéodory's extension theorem. Thus $s_0$ has to be $-i$ and we have finished the proof of the “if” part.

    \vspace{0.1in}
    {\it The “only if” part}. Suppose $s_n\to-i$. Take a curve $l\subset\mathbb{H}$ such that $l\cap(0,+\infty) = \emptyset$, $l$ contains all $\Phi_{\mathcal{E}^1}(s_n)$ and $l$ converges to $y_0$. Set $\mathcal{R} = \Phi_{\mathcal{E}^1}^{-1}(l)$. By Corollary \ref{cor.internal-curve-mustland}, $\mathcal{R}$ converges to $-i$. While by the “if” part, $\Phi_{\mathcal{E}^1}^{-1}((0,+\infty))$ converges to $-i$. Take an arc $\gamma\subset\mathcal{E}^1$ linking $\mathcal{R}$ and $\Phi_{\mathcal{E}^1}^{-1}((0,\infty))$ so that three of them bounds a Jordan domain $\mathcal{U}$. Then $\Phi_{\mathcal{E}^1}(\mathcal{U})$ is the Jordan domain bounded by $l$, $(0,+\infty)$ and $\Phi_{\mathcal{E}^1}(\gamma)$. But since $y_0\not\in\{0,\infty\}$, this contradicts Carathéodory's extension theorem.

\end{proof}

\begin{proof}[Proof of Theorem \ref{thm.implosion.family-gs}]
Let $s_n\subset\mathcal{E}^1$ be a sequence converging to $\partial\mathcal{E}^1$.
\begin{itemize}
    \item Suppose that $\mathcal{E}^1$ is adjacent type. Then by Lemma \ref{lem.i.notonboundary}, $s_n$ does not converge to $i$. Thus by Lemma \ref{lem.accumulation.inside} and \ref{lem.adjacent.accumulation}, $s_n$ converges to $\partial\mathcal{H}$ if and only if $\Phi_{\mathcal{E}^1}(s_n)$ converges to 0 or $\infty$. In particular $\partial\mathcal{E}\cap\partial\mathcal{H} = \{-i\}$.
    \item  Suppose that $\mathcal{E}^1$ is disjoint type. Then by Lemma \ref{lem.nondouble.parabolic}, $s_n$ does not converge to $\{i,-i\}$. Thus by Lemma \ref{lem.accumulation.inside}, $s_n$ converges to $\partial\mathcal{H}$ if and only if $\Phi_{\mathcal{E}^1}(s_n)$ converges to 0 or $\infty$. By Lemma \ref{lem.disjoint.accumulation}, $\partial\mathcal{E}\cap\partial\mathcal{H}$ is a singleton which is a Misiurewicz parabolic parameter.
\end{itemize}
\end{proof}

\begin{proof}[Proof of Theorem \ref{thm.implosion-main2}] Let's verify the items one by one: 
\begin{itemize}
    \item Notice that 
    \ref{thm.implosion-main2.first} and \ref{thm.implosion-main2.third} follow immediately from Theorem \ref{thm.implosion.family-gs}.

    \item  Next we prove \ref{thm.implosion-main2.second}. If $a_0\in\partial\mathcal{K}_1$ is Misiurewicz parabolic, then it follows from Lemma \ref{lem.existence.disjoint}. Therefore suppose that $a_0=0$. The existence of an adjacent component $\mathcal{E}$ attached to $a=0$ is given by Proposition \ref{Propositionexistence}. To fix the idea, let's assume that $\mathcal{E}\subset\mathbb{H}$ (recall that by Theorem \ref{thm.runze}, $\mathcal{K}_1\setminus\{0\}$ has two components, symmetric with respect to the $y$-axis). Assume by contradiction that there is another $\mathcal{E}'\subset\mathbb{H}\cap\mathbf{Esc}^1$ attached to $a=0$. By \ref{thm.implosion-main2.first} and \ref{thm.implosion-main2.third}, $\mathcal{E}'$ must be adjacent type and $\overline{\Phi_{\mathcal{E}'}^{-1}((0,+\infty))}$ is a Jordan curve only intersecting $\partial\mathcal{K}_1$ at $a=0$. Notice that for $\epsilon>0$ small enough, $f_{a}$ has gate structure $(*,2)$ for $a\in\Phi_{\mathcal{E}'}^{-1}((0,\epsilon))$, hence 
 $\Phi_{\mathcal{E}'}^{-1}((0,\epsilon))\subset\Omega_{\eta,r}$ (recall lemma \ref{lem.perturb.Fatoucoordi}). However, parameters in $\Phi_{\mathcal{E}'}^{-1}((0,\epsilon))$ will satisfy Equation (\ref{eq.orbit-correspondence}) in Lemma \ref{lem.orbit.correspond} if we keep the same $x(a,t),y(a,t)$ constructed in Proposition \ref{Propositionexistence} (recall (\ref{eq.shooting.externalray})). This contradicts the uniqueness of the solution given by Lemma \ref{lem.orbit.correspond}. We complete the proof of \ref{thm.implosion-main2.second}.
 \end{itemize}
From now on we assume that the phase $\tau$ satisfies either $\Im\,\tau>0$ or $\tau\in\mathbb{Q}$.
 
 \begin{itemize}
 \item We proceed the proof of \ref{thm.implosion-main2.fourth}. Let us assume that both $\mathcal{E},\mathcal{E}'\subset\mathbb{H}$. By \ref{thm.implosion-main2.first} and \ref{thm.implosion-main2.second}, at least one of $\mathcal{E}$ and $\mathcal{E}'$ (let's say $\mathcal{E}$) is disjoint type. Suppose the contrary that $a_0\in\partial\mathcal{E}\cap\partial\mathcal{E}'$. Let $E^1_{a_0}$ be the first level escaping region of $f_{a_0}$ (Definition \ref{def.escape-region}). Since $z=0$ is either an attracting or parabolic fixed point for the projective horn map $H_{a,\tau}$ (Definition \ref{def.hornmap}) and it attracts at least one critical orbit (see the proof of Lemma \ref{lem.at-least-one}). Hence for any component $C\subset E^1_{a_0}$, $\overline{C}$ contains at most one critical point of $f_{a_0}$. Let $C^1_{a_0,(0)}\subset E^1_{a_0}$ be the component attached to $z=0$. Since $\partial B^*_{a_0}$ ($B^*_{a_0}$ is the immediate parabolic basin of $f_{a_0}$) moves holomorphically for $a$ near $a_0$, $\partial C^1_{a_0,(0)}$ intersects at least one critical orbit of $f_{a_0}$. Hence this critical orbit is unique and let's say that it is associated to the critical point $c_{a_0}$ with critical value $v_{a_0}$. Moreover, since $\mathcal{E}$ is disjoint type, $C^1_{a_0,(0)}$ is a croissant by Lemma \ref{lem.classification.C1}. Let $Cv_{a_0}\subset E^1_{a_0}$ be the component such that $v_{a_0}\in\partial Cv_{a_0}$ and $Cv_{a_0}$ is sent to $C^1_{a_0,(0)}$ by the smallest iteration of $f_{a_0}$. Notice that $Cv_{a_0}$ is also a croissant and $c_{a_0}\in\partial Cv_{a_0}$. Let $Cv_{a_0}'\subset E^{1}_{a_0}$ be the other component attached to $c_{a_0}$. Since $Cv_{a_0},Cv_{a_0}'$ are both sent to $C^1_{a_0,(0)}$ but under different iterations of $f_{a_0}$, $c_{a_0}$ must belong to $\partial C^1_{a_0,(0)}$ and hence $C^1_{a_0,(0)} = Cv_{a_0}$. Now $f_{a_0}^{-1}(C^1_{a_0,(0)})$ has three components: $C^1_{a_0,(0)},Cv_{a_0}'$ and $Cc_{a_0}$, where $\overline{Cc_{a_0}}$ contains the co-critical point and is disjoint from $\overline{C^1_{a_0,(0)}\cup Cv_{a_0}'}$. On the other hand,  $C^1_{a,(0)}$ is a croissant and $Cv_{a}$ is a bi-croissant for $a\in\mathcal{E}$. Moreover, the Hausdorff limit of $\overline{C^1_{a,(0)}\cup Cv_a}$ as $a\in\mathcal{E}$ tends to $a_0$ is $\overline{C^1_{a_0,(0)}\cup Cv'_{a_0}\cup Cc_{a_0}}$, which is a contradiction. 
\item Finally we prove \ref{thm.implosion-main2.fifth}. First suppose that $\mathcal{E}$ is disjoint type. From the proof of \ref{thm.implosion-main2.fourth}, we see that for $a\in\partial\mathcal{E}$, $\overline{Cv_{a_0}}$ does not contain any critical point of $f_{a_0}$. Hence $\overline{Cv_{a_0}}$ moves holomorphically. The result then follows from a standard transversality argument (\cite[Lem. 3.2]{Sh2}).

Now suppose that $\mathcal{E}$ is adjacent type. Then $v_{a_0}\in\mathfrak{L}$ where $\mathfrak{L}$ is either $\mathfrak{L}^u_{a_0}$ or $\mathfrak{L}^l_{a_0}$ (recall Lemma \ref{lem.classification.C1}). Let $\mathfrak{L'}\subset\mathfrak{L}$ be the subcurve starting from $z=0$ and ending at $v_{a_0}$. Then $\mathfrak{L'}$ moves holomorphically for $a$ near $a_0$. The result follows again from the transversality argument.
\end{itemize}

\end{proof}


\begin{thebibliography}{99999}
\addcontentsline{toc}{section}{\textbf{Bibliography}}

\bibitem[BEE]{BEE} X. Buff, J. Ecalle, A. Epstein, {\it Limits of degenerate parabolic quadratic rational maps.} Geometric and Functional Analysis, 23, 02 2013.




\bibitem[CEP]{CEP} A. Chéritat, A. Epstein, C. Petersen, {\it Perspectives on Parabolic Points in Holomorphic Dynamics.} Survey, Banff International Research Station, (2015).



\bibitem[Do]{Do} A. Douady, {\it Does a Julia set depend continuously
on the Polynomial?}. Proceedings of Symposia in Applied Mathematics.
Volume 49, (1994).

\bibitem[DoHu1]{DoHu1} A. Douady, J. Hubbard, {\it Etude dynamique des polynômes complexes.} Publ. Math. d’Orsay, 1984.

\bibitem[DoHu2]{DoHu2} A. Douady, J. Hubbard, {\it On the dynamics of polynomial-like mappings}. Ann. Sci. l'\'Ecole Norm. Sup., 18 (1985), 287-343.


\bibitem[Er]{Er} A. Eremenko, {\it Singularities of inverse functions}.	arXiv:2110.06134.


\bibitem[Ka]{Ka} A. Kapiamba, {\it An optimal Yoccoz inequality for near-parabolic quadratic polynomials}. http://arxiv.org/abs/2103.03211, (2021).



\bibitem[La]{La} P. Lavaurs, {\it Systemes dynamiques holomorphes: explosion de points périodiques paraboliques.} PhD
thesis, Thèse de doctrat de l’Université de Paris-Sud, Orsay, France, 1989.



\bibitem[Mi]{Mi} J. Milnor. Cubic polynomial maps with periodic critical orbit, part i. Complex Dynamics: Families and
Friends, pages 333–412, 01 2009.








\bibitem[Ou]{Ou} R. Oudkerk, \emph{The parabolic implosion for $f_0(z) = z+z^{\nu+1}+\mathcal{O}(z^{\nu+2})$}, PhD thesis,
University of Warwick, England, (1999).









\bibitem[PT]{PT} C. Petersen, L. Tan \emph{Branner-Hubbard Motions and attracting dynamics}, Dynamics on the
Riemann sphere, 45–70, Eur. Math. Soc., Zurich, (2006).



\bibitem[Ro]{Ro} P. Roesch, \emph{Cubic polynomials with a parabolic point}, Ergodic Theory and Dynamical Systems, 30(6),
1843-1867. doi:10.1017/S0143385709000820, 2010.

\bibitem[Sh1]{Sh1} M. Shishikura, \emph{The Hausdorff dimension of the boundary of the Mandelbrot set and Julia sets}, Ann. Math. 147, 25?267 (1998).

\bibitem[Sh2]{Sh2} M. Shishikura, \emph{Bifurcation of parabolic fixed points}, The Mandelbrot set, Theme and variations, London Math. Soc. 274, Cambridge: Cambridge Univ. Press, (2000).


\bibitem[Ta]{Ta} L. Tan, \emph{Local properties of the Mandelbrot set at parabolic points}, The Mandelbrot set, Theme and variations, London Math. Soc. 274, Cambridge: Cambridge Univ. Press, (2000).



%\bibitem[YaZh]{YaZh} J. Yang, R. Zhang \emph{Rigidity of bounded type cubic Siegel polynomials}, arXiv:2311.00431, (2023).


\bibitem[Z]{Z} R. Zhang, \emph{On Dynamical Parameter Space of Cubic Polynomials with a Parabolic Fixed Point}, Journal of the London Mathematical Society, Volume 110 (2024) no. 6 | DOI:10.1112/jlms.70038.		
\end{thebibliography}
\end{document}